\documentclass[preprint,10pt]{elsarticle}
\usepackage{amssymb}
\usepackage{amsmath}
\usepackage{amsfonts}
\usepackage{float}
\usepackage{algorithm}
\usepackage{algpseudocode}
\usepackage{amsthm}
\usepackage{graphicx}
\usepackage{subfigure}
\usepackage{graphics}
\usepackage{epsfig}
\usepackage{mathrsfs}
\usepackage{lineno}
\usepackage{makecell}
\usepackage{txfonts}
\usepackage{booktabs}
\usepackage{multirow}
\usepackage{threeparttable}
\usepackage{paralist}
\usepackage{algorithm}
\usepackage{graphics}
\usepackage{epsfig}
\usepackage{natbib}
\usepackage{epstopdf}
\usepackage{natbib}
\allowdisplaybreaks[4]
\biboptions{numbers,sort&compress}
\usepackage[colorlinks=true]{hyperref}
\hypersetup{urlcolor=red, citecolor=blue}
\usepackage[top=1in,bottom=1in,left=1in,right=1in]{geometry}
\numberwithin{figure}{section}
\makeatletter
\@addtoreset{algorithm}{section}
\makeatother
\newtheorem{theorem}{Theorem}[section]

\newtheorem{definition}[theorem]{Definition}
\newtheorem{lemma}{Lemma}[section]
\newtheorem{example}{Example}[section]
\numberwithin{equation}{section}

\newtheorem{remark}{Remark}[section]

\numberwithin{table}{section}
\begin{document}
	\begin{frontmatter}
		\title{A sketch-and-project method for solving the matrix equation $AXB=C$  \footnote{	
				The first author's research was supported by the Fundamental Research Funds for the Central Universities (grant number 18CX02041A) and the Shandong Provincial Natural Science Foundation (grant number ZR2020MD060). The fourth author's research was supported by the National Natural Science Foundation of China (grant number 42176011, 62231028)}}
		\date{}
			\author{Wendi Bao}
		\ead{baowendi@sina.com}
		\author{Zhiwei Guo }
		\ead{gzw\_13605278246@163.com}
		\author{Weiguo Li}
		\ead{liwg@upc.edu.cn}
		\author{Ying Lv}
		\ead{lyrr1017@163.com}
		\author{Jichao Wang}
		\ead{wangjc@upc.edu.cn}
		\address{College of Science,
			China University of Petroleum, Qingdao 266580, P.R. China}
			
		\begin{abstract}
			In this paper, based on an optimization problem, a sketch-and-project method for solving the linear matrix equation $AXB=C$ is proposed. We provide a thorough convergence analysis for the new method and derive a lower bound on the convergence rate and some convergence conditions including the case that the coefficient matrix is rank deficient. By varying three parameters in the new method and convergence theorems, the new method recovers an array of well-known algorithms and their convergence results. Meanwhile, with the use of Gaussian sampling, we can obtain the Gaussian global randomized Kaczmarz (GaussGRK) method which shows some advantages in solving the matrix equation $AXB=C$. Finally, numerical experiments are given to illustrate the effectiveness of recovered methods.
		\end{abstract}
		
		
	\end{frontmatter}
		
	\section{Introduction}
	\label{sec:intro}
	
	In this paper, we consider the linear matrix equation
	\begin{equation}\label{eq:1.1}
		AXB=C,
	\end{equation}
	where coefficient matrices $A\in\mathbb{R}{^{p\times m}}$ and $B\in\mathbb{R}{^{n\times q}}$ , a right-hand side $C\in\mathbb{R}{^{p\times q}}$, and an unknown matrix $X\in\mathbb{R}^{m\times n}$. We shall assume throughout that the equation is consistent, there exists an $X^*$ satisfying $AX^*B=C$. This assumption can be relaxed by choosing the least norm solution when the system has multiple solutions. The large-scale linear matrix equation arises in computer science, engineering, mathematical computing, machine learning, and many other fields such as surface fitting in computer-aided geometric design (CAGD) \cite{LH2018}, signal and image processing \cite{MR1989}, photogrammetry, etc.
	
		Classical solvers for the matrix equation \eqref{eq:1.1} are generally fall into two categories: direct and iterative methods. Direct methods, such as the generalized singular value decomposition and QR-factorization-based algorithms \cite{HD1990, ZH1995} are attractive when $A$ and $B$ are small and dense, while iterative methods are usually more practical in the field of large-scale system of equations \cite{FD2005, XW2013, ZT2017}. It is universally known that the matrix equation \eqref{eq:1.1} can be written as the following equivalent matrix-vector form by the Kronecker product
	\begin{equation}\label{eq:1.2}
		\left(B^{\top} \otimes A\right)vec\left(X\right)=vec\left(C\right),
	\end{equation}
	where the Kronecker product $\left(B^{\top} \otimes A\right) \in\mathbb{R}{^{pq\times mn}}$, the right-side vector $vec\left(C\right) \in\mathbb{R}{^{pq\times 1}}$, and the unknown vector $vec\left(X\right) \in\mathbb{R}{^{mn\times 1}}$. Many iteration methods are proposed \cite{DS2008, PZ2010} to solve the matrix equation \eqref{eq:1.1} by applying the Kronecker product. When the dimensions of A and B are large, the dimension of linear system \eqref{eq:1.2} increases sharply, which increases the memory usage and calculation cost of numerical algorithms. Many iterative methods frequently use the matrix-matrix product operation. Consequently, a lot of computing time consumes.

	Many recent researches show that Kaczmarz-type methods are suitable for large-scale problems since each Kaczmarz iterate requires only one row of the coefficient matrix and no matrix-vector product. In \cite{NYQ2022}, to solve large-scale consistent linear matrix equations \eqref {eq:1.1}, Niu and Zheng proposed the global randomized block Kaczmarz (GRBK) algorithm and the global randomized average block Kaczmarz (GRABK) algorithm. Based on greedy ideas, Wu et al. \cite{WNC2022} introduced the relaxed greedy randomized Kaczmarz (ME-RGRK) method and the maximal weighted residual Kaczmarz (ME-MWRK) method for solving consistent matrix equation $AXB = C$.
	In \cite{DK2021}, Du et al. extended Kaczmarz methods to the randomized block coordinate descent (RBCD) method for solving the matrix least-squares problem $\min\limits_{X \in\mathbb{R}{^{m\times n}}} \left\|C-AXB\right\|_F$.  Meanwhile, by applying the Kaczmarz iterations and the hierarchical approach, Shafiei and Hajarian obtained new iterative algorithms for solving the Sylvester matrix equation in \cite{SGS2022}. For linear systems $Ax=b$, Robert M. Gower et al. \cite{GRM} constructed a sketch-and-project method, which unifies a variety of randomized iterative methods
	including both randomized Kaczmarz and coordinate descent along with all of their block variants. The general sketch-and-project framework has not yet been analyzed for the matrix equation $AXB=C$.
	
	Inspired by the idea in \cite{GRM} and \cite{SGS2022}, we propose a sketch-and-project method for solving the matrix equation \eqref{eq:1.1}.	The convergent analysis of the proposed method is investigated and existing complexity results for known variants can be obtained. A lower bound on the convergence rate is explored for the evolution of the expected iterates. Numerical experiments are given to verify the validity of recovered methods.
	
	The main contribution of our work is summarized as follows.
		\begin{itemize}
		\setlength{\itemsep}{1pt}
		\setlength{\parskip}{0pt}
		\setlength{\parsep}{0pt}
		\item[(1)] {\rm\bf  New method.}
		By introducing three different parameters, we induce a sketch-and-project method for the matrix equation \eqref{eq:1.1}. The iteration scheme is as follows:
		\begin{equation*}
			X^{k+1}=X^k-Z_1' \left(X^k -X^* \right)Z_2,
		\end{equation*}
		where $Z_1'=G^{-1}A^\top S\left( S^\top AG^{-1}A^\top S \right)^\dag S^\top A, \  Z_2=B P\left( P^\top B^\top BP \right)^\dag P^\top B^\top.$
		  $S\in\mathbb{R}{^{p\times \tau_1}}$,  $P\in\mathbb{R}{^{q\times \tau_2}}$ and $G\in\mathbb{R}{^{m\times m}}$  are three parameters. 
		\item[(2)] {\rm\bf  Complexity: general results.}  
		The convergence analysis of the proposed method is given, which is summarized in Table 1. In particular, we provide an explicit convergence rate $\rho$ for the exponential decay
		of the expected norm of the error (line 2 of Table 1)  and  the norm of the expected error of the iterates (line 3 of Table 1).  Furthermore, since $\rho$ is always bounded between 0 and 1, Theorem \ref{lem:7.2} provides a lower bound on $\rho$ that shows that the rate can potentially improve as the number $d$ increases.
		\begin{table}[!htbp]%
			\centering \renewcommand\arraystretch{1.4}
			\caption{Our main complexity results. }\label{tab:RKRGS}
			\begin{tabular}{|c|c|}
				\hline
				$\mathbf{E}\left[X^{k+1}-X^*\right]= \mathbf{E}\left[X^k -X^* \right]-\mathbf{E}\left[Z_1'(X^k -X^*)Z_2\right]$ &{\bf Theorem \ref{thm:2.1}}\\
				$\mathbf{E}\left[ \left\|X^{k+1}-X^*\right\|_{F(G)}^2 \right]\leq\rho  \mathbf{E}\left[\left\|X^{k}-X^*\right\|_{F(G)}^2 \right]$ &{\bf Theorem \ref{thm:2.1}}\\
				$    			\left\|\mathbf{E}\left[X^{k+1}-X^* \right]\right\|_{F(G)} \leq \rho  \left\|\mathbf{E}\left[X^{k}-X^*\right]\right\|_{F(G)}
				$ &{\bf Theorem \ref{thm:2}}\\
				$\mathbf{E}\left[ \left\|X^{k}-X^*\right\|_{F(G)}^2 \right] \leq \rho_{\sigma}\left\|X^{k}-X^*\right\|_{F(G)}^2.$&{\bf Theorem \ref{theB.2}}\\
							\hline
				\end{tabular}
			\begin{tablenotes}
				\centering 
			\footnotesize
			\item[*] The convergence rate is $\rho=1-\lambda_{\min}\left( \mathbf{E}\left[Z_2 \otimes Z_1'\right] \right), \rho_{\sigma}=1-\sigma^2_{\min}\left( \mathbf{E}\left[ Z_2 \otimes Z_1' \right] \right) <1$. 
		\end{tablenotes}
		\end{table}
		\item[(3)] {\rm \bf Complexity: Special cases.}  As a generalized iterative method,  the parameter random matrices $S$, $P$ and $G$ are given specific values, some well known methods are obtained. Two convergence theorems for the generalized method are explored. Besides these generic results, which hold without any major restriction on the sampling matrix $S,P$ (in particular, it can be either discrete or continuous), we give a specialized result applicable to discrete sampling matrices $S,P$ (see Theorem \ref{thm:2.2}). Our analysis recovers the existing rates (see Table \ref{table:6}).
		\begin{table}[!ht]
			\centering	\renewcommand\arraystretch{1.5}
			\resizebox{\textwidth}{!}{
				\begin{threeparttable}
					\caption{ Summary of convergence guarantees of various sampling strategies for the sketch-and-project method.}
					\label{table:6}
					\begin{tabular}{|c|c|c|c|c|}
						\hline
						Method & Sampling Strategy & Convergence Rate Bound & Rate Bound Derived From \\
						\hline
						GRK    &	$p_i=\frac{\left\|A_{i,:}\right\|^2_2}{\left\|A\right\|^2_F}, p_j=\frac{\left\|B_{:,j}\right\|^2_2}{\left\|B\right\|^2_F}$              &	$1-\frac{\lambda_{min}\left(A^{\top}A\right)\lambda_{min}\left(BB^{\top}\right)}{\left\|A\right\|^2_F\left\|B\right\|^2_F}$                    & 	 Theorem \ref{thm:2.1} or Theorem \ref{thm:2.2}     \\
						\hline
						RK-A   &$p_i=\frac{\left\|A_{i,:}\right\|^2_2}{\left\|A\right\|^2_F}$      &$1-\frac{\lambda_{min}\left(AA^{\top}\right)}{\left\|A\right\|^2_F}$  & Theorem \ref{thm:2.1} or Theorem \ref{thm:2.2} \\
						\hline
						RCD    &$p_j=\frac{\left\|A_{:,j}\right\|^2_2}{\left\|A\right\|^2_F}$                &$1-\frac{\lambda_{min}\left(AA^{\top}\right)}{\left\|A\right\|^2_F}$ & Theorem \ref{thm:2.1} or Theorem \ref{thm:2.2}  \\
						\hline
						GaussGRK  &Gaussian sampling   & $1-\frac{4}{\pi^2Tr\left(\Omega_1\right)Tr\left(\Omega_2\right)}\cdot \lambda_{min}\left(\Omega_2^{\top}\bigotimes \Omega_1\right)$         & Theorem \ref{thm:4.1}    \\
						\hline
						GaussRK-A &Gaussian sampling   & $1-\frac{2\lambda_{min}\left(\Omega_1\right)}{\pi Tr\left(\Omega_1\right)}$                &Theorem \ref{thm:4.1}                 \\
						\hline
					\end{tabular}
					\begin{tablenotes}
						\footnotesize
						\item[*]  $\Omega_1,\Omega_2$ are defined as the following described theorems.
					\end{tablenotes}
				\end{threeparttable}
			}
		\end{table}
		
		\item[(4)] {\rm \bf Application and Extension.} 
		 We apply our algorithms to real-world applications, such as the real-world sparse data and CT data. Gaussian global randomized Kaczmarz (GaussGRK) method shows some advantages in solving the matrix equation $AXB=C$. Meanwhile, based on our approach, many avenues for further development and research can be explored. For instance, it is possible to extend the results to the case that $S$ and $P$ are count sketch transforms. One also can design randomized iterative algorithms for finding the generalized inverse of a very large matrix and the solutions with special structures such as symmetric positive definite matrices.
		
	\end{itemize}
	
	The rest of this paper is organized as follows. In Section \ref{sec:2}, some notations and preliminaries are introduced. In Section \ref{sec:3}, we derive the generalized iterative method for solving matrix equation \eqref{eq:1.1}.
	After that, Convergence analysis is explored. Convergence rate, convergence conditions and a low bound of convergenc rate are obtained in Section \ref{sec:4}. 
	We recover several existing methods by selecting appropriate parameters $G$, $S$ and $P$. Meanwhile, all the associated complexity results will be summarized in the final theorems in Section \ref{sec:5}.
	In Section \ref{sec:6}, we shall describe variants of our method in the case when parameters $S$ and $P$ are Gaussian vectors and establish the convergence theorem.
	In Section \ref{sec:7}, some numerical examples are presented to verify the efficiency of the proposed method and compared the convergence rate of it with other existing methods.
	At the end, some conclusions are given in Section \ref{sec:6}.
	
	\section{Notation and preliminary}
	\label{sec:2}
	For any matrix $M\in\mathbb{R}{^{m\times n}}$, we use $M^\top$, $Range\left( M \right)$, $M_{ij}$, $\sigma_{\max}\left( M \right)$ and $\sigma_{\min}\left( M \right)$ to denote its transpose, column space, the $(i,j)$-th entry, the largest and smallest nonzero singular values, respectively. When the matrix $M$ is square, then $Tr(M)$ represents its trace. Define the Frobenius inner product $\langle A,B\rangle_F:=Tr\left( A^{\top}B\right)=Tr\left( AB^{\top}\right)$, where $A,B\in\mathbb{R}{^{m\times n}}$. Specially, we denote the Frobenius norm as $\left\| A \right\|_F^2=\left\langle A, A\right\rangle_F$.
	Let $\left\|X\right\|_{F(G)}^2=Tr\left(X^{\top}GX \right)$, where $G$ is a parameter matrix which is symmetric positive definite. When $G$ is an identity matrix, it holds that $\left\|A\right\|_{F(G)}^2=\left\| A \right\|_F^2$, $\left\|  A \right\| _G= \underset{\left\|x \right\|_G=1}{max}\left\|  Ax \right\|_G$.
	$\lambda_{min}(A)$ and $\lambda_{max}(A)$,  respectively, are the smallest and largest eigenvalues values of the matrix $A$. $A \succeq B$ indicates that $A-B$ is positive semi-definite.
	
	\begin{lemma}[\cite{SV2016}]\label{lem:1.1} For the Kronecker product, some well-known properties are summarized as follows.
		\begin{itemize}
			\item $\text{vec}\left(ABC\right)=\left(C^\top \otimes A\right)\text{vec}\left( B \right),$
			\item $\left(AC\right)\otimes \left(BD\right)=\left(A\otimes B\right)\left(C\otimes D\right),$
			\item $\left\| A\otimes B\right\|_F=\left\|A\right\|_F \cdot \left\|B\right\|_F,$
			\item $\left(A\otimes B\right)^{\top}=A^\top \otimes B^\top,$
			\item $\left(A\otimes B\right)^{-1}=A^{-1} \otimes B^{-1},$
			\item $\lambda \left( A\otimes B \right) = \left\{ \lambda_i \mu_j : \lambda_i \in \lambda \left( A \right), \mu_j \in \mu \left( B \right),i=1,2,...,n; j=1,2,...,m \right\},$
		\end{itemize}
		where $\lambda \left( A \right),\mu \left( B \right)$ denote spectrums and the matrices $A$, $B$, $C$ and $D$ have compatible dimensions.
	\end{lemma}

	\begin{lemma}[\cite{Graham1981, SGS2022}] \label{lem:1.2} If $A,B,C$ and $X$ are four real matrices of compatible sizes, we have
		\begin{itemize}
			\item $\frac{\partial}{\partial X}Tr\left(AXB\right)=A^{\top}B^{\top}, $
			\item $\frac{\partial}{\partial X}Tr\left(AX^{\top}B\right)=BA, $
			\item $\frac{\partial}{\partial X}Tr\left(X^{\top}X\right)=\frac{\partial}{\partial X}Tr\left(XX^{\top}\right)=2X, $
			\item $\frac{\partial}{\partial X}Tr\left(X^{\top}AXB\right)=AXB+A^{\top}XB^{\top}, $
			\item $\frac{\partial}{\partial X}Tr\left\{\left(AXB+C\right)\left(AXB+C\right)^{\top}\right\}=2A^{\top}\left(AXB+C\right)B^{\top}.$
		\end{itemize}
	\end{lemma}

	


	\section{A Sketch-and-project Kaczmarz iterative method}
	\label{sec:3}
	 To solve the problem \eqref{eq:1.1},  starting from $X^{k}$ our method draws random
	matrices $S, P$ and uses them to generate a new point $X^{k+1}$. This iteration can be
	formulated in two seemingly different but equivalent ways (see Fig. \ref{figure:0}).
	\subsection{Two  formulations}
	\begin{itemize}
		\item {\bf Projection viewpoint: sketch-and-project.} $X^{k+1}$ is the nearest point to $X^k$ which solves a sketched version of the original linear system:
		\begin{equation}\label{eq:1.3}
			X^{k+1} = \arg\min\limits_{X \in\mathbb{R}{^{m\times n}}} \frac{1}{2}\left\| X-X^k\right\|_{F(G)}^2 \ \  \text{subject to} \ \ S^\top AXBP=S^\top CP,
		\end{equation}
		where  $S\in\mathbb{R}{^{p\times \tau_1}}$ and $P\in\mathbb{R}{^{q\times \tau_2}}$ are two parameters, each of them is drawn in an independent and identically distributed fashion at each iteration. We do not restrict the number of columns of $S$ and $P$, hence $\tau_1$ and $\tau_2$ are two random variables.
		\item {\bf Optimization viewpoint: constrain-and-approximate.}  The solution set of the random sketched equation contains all solutions of the original system. However, there are many solutions, so we have to define a method to select one of them. From the optimization viewpoint, we know that $X^{k+1}$ is the best approximation of $X^*$ in a random space passing through $X^k$. That is, we choose an affine space randomly which contains $X^k$ and constrain our method to choose the next iterate from this space. That is to say, consider the following problem
 \begin{equation}\label{eq:1.4}
			\begin{split}
				&X^{k+1} = \arg\min\limits_{X \in\mathbb{R}{^{m\times n}}}\frac{1}{2} \left\| X-X^*\right\|_{F(G)}^2,\\
				&\text{subject to} \ \ X=X^k+ G^{-1}A^{\top}SYP^{\top}B^{\top},\ Y {\rm \ is \ free.}
			\end{split}
		\end{equation}
 Then we pick $X^{k+1}$ as the point which best approximates $X^*$ on this space.
		
		\begin{figure}[htbp]
			\centering
			{
				\includegraphics[width=8cm]{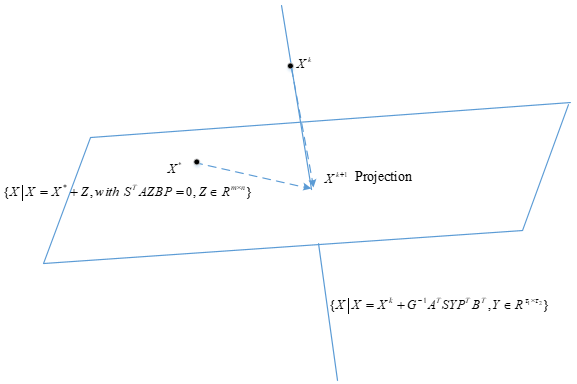}
			} \\
			\caption{The geometry of our algorithm. The next iterate, $X^{k+1}$,
				is obtained by projecting $X^{k}$ onto the affine space formed by intersecting $\{X|X=X^*+Z, \ with\  S^{\top}AZBP=0, Z\in R^{m \times n}\}$
				(see \eqref{eq:1.3}) and $\{X|X=X^k+G^{-1}A^{\top}SYP^{\top}B^{\top},\ Y\in R^{\tau_{1} \times \tau_{2}}\}$ (see \eqref{eq:1.4}).}\label{figure:0}
			%
	\end{figure}
\end{itemize}

	\subsection{Stochastic iterative algorithm}
		 Now we deduce the iterative scheme for the problem \eqref{eq:1.1}.  Based on the Lagrangian function of the problem \ref{eq:1.3}, we have
	\begin{equation}\label{eq:2.1.1}
		\begin{split}
			\mathcal{L}\left(X,Y\right) & = \frac{1}{2}\left\|X-X^k\right\|_{F(G)}^2+\left\langle Y,S^{\top}AXBP-S^{\top}CP \right\rangle \\
			& = \frac{1}{2}Tr\left( \left(X-X^k\right)^{\top}G\left( X-X^k\right) \right)+Tr\left( Y^{\top}\left(S^{\top}AXBP-S^{\top}CP\right)\right),
		\end{split}
	\end{equation}
	where $Y$ is a Lagrangian multiplier. By using \ref{lem:1.2}, we take the gradient of $\mathcal{L}\left(X,Y\right)$ and equate its components to zero for finding the stationary matrix:
	\begin{eqnarray*}
		\nabla_X \mathcal{L}\left(X,Y\right) |_{X^{k+1}}&=& \frac{1}{2}\left(G+G^{\top} \right)\left( X^{k+1}-X^k\right)+A^{\top}SYP^{\top}B^{\top}=0, \\
		\nabla_Y \mathcal{L}\left(X,Y\right) |_{X^{k+1}} &=& S^{\top}AX^{k+1}BP-S^{\top}CP=0.
	\end{eqnarray*}
	Since $G$ is symmetric positive definite, then we have
	\begin{equation*}
		S^{\top}AX^{k}BP-S^{\top}AG^{-1}A^{\top}SYP^{\top}B^{\top}BP=S^{\top}CP,
	\end{equation*}
	\begin{equation*}
		Y=\left(S^{\top}AG^{-1}A^{\top}S \right)^{\dag}S^{\top}\left(AX^{k}B-C\right)P\left(P^{\top}B^{\top}BP\right)^{\dag},
	\end{equation*}
	thus, the iteration form is as follows
	\begin{equation}\label{eq:2.1.2}
		\begin{split}
			X^{k+1}= & X^{k}-G^{-1}A^{\top}S\left(S^{\top}AG^{-1}A^{\top}S \right)^{\dag}
			 S^{\top}\left(AX^{k}B-C\right)P\left(P^{\top}B^{\top}BP\right)^{\dag}P^{\top}B^{\top}.
		\end{split}
	\end{equation}
	Let$$Z_1'=G^{-1}A^\top S\left( S^\top AG^{-1}A^\top S \right)^\dag S^\top A,\ Z_2=B P\left( P^\top B^\top BP \right)^\dag P^\top B^\top, $$ the above scheme becomes
	\begin{equation}\label{eq:2.1.3}
		X^{k+1} =  X^{k}- Z_1'\left(X^{k}-X^*\right)Z_2.
	\end{equation}
Therefore, the sketch-and-project method is obtained.

	Especially, when $\left\|X\right\|_{F(G)}^2=\left\| X \right\|_F^2$, i.e., the problem
\begin{equation*}
	X^{k+1} = \arg\min\limits_{X \in\mathbb{R}{^{m\times n}}}\frac{1}{2} \left\| X-X^*\right\|_F^2.
		\end{equation*}
	With a similar process, we have
	\begin{equation}\label{eq:2.1.4}
		\begin{split}
		X^{k+1}=&X^{k}-A^{\top}S\left(S^{\top}AA^{\top}S \right)^{\dag}
		S^{\top}\left(AX^{k}B-C\right)P\left(P^{\top}B^{\top}BP\right)^{\dag}P^{\top}B^{\top}\\
		=&X^k-Z_1\left(X^k-X^{*}\right)Z_2,
	\end{split}
	\end{equation}
	where $Z_1=A^\top S\left( S^\top AA^\top S \right)^\dag S^\top A,\ Z_2=B P\left( P^\top B^\top BP \right)^\dag P^\top B^\top.$
	
	\begin{algorithm}
		\caption{Stochastic Iterative Method for Matrix Equations $AXB=C$}\label{algorithm1}
		\begin{algorithmic}[1]
			\State \textbf{Input:} $A\in R^{p\times m},\ B\in R^{n\times q}$, $C\in R^{p\times q}$ and the positive definite matrix $G \in \mathbb{R}^{m\times m}$\State \textbf{Initialize:} arbitrary square matrix $X_0\in\mathbb{R}^{m\times n}$
			\For{$k=1,2,\cdots, $}
			\State Sample parameters: $P,\ S$ are distribution over random matrices;
			\State Compute $T_1=G^{-1}A^\top S\left( S^\top AG^{-1}A^\top S \right)^\dag S^\top,\ T_2=P\left( P^\top B^\top BP \right)^\dag P^\top B^\top$
		\State $X^{k+1} = X^{k}-T_1(C-AX^{k}B)T_2$
			\EndFor
			\State \textbf{Output: }last iterate $X^{k}$
		\end{algorithmic}
	\end{algorithm}

Recall that $S \in \mathbb{R}^{q\times \tau_1}, P\in \mathbb{R}^{q \times \tau_2}$ (with $\tau_1, \tau_2$ possibly being random) and $A \in \mathbb{R}^{q\times m}, B\in \mathbb{R}^{n \times q}, G \in \mathbb{R}^{m\times m}$. Let us define the random quantity $$d = rank\left((P^{\top}B^{\top}) \otimes (S^{\top}A)\right),$$
and notice that $d \leq \min\left\{\tau_1\tau_2, mn\right\}$, we have
$$ dim\left(Range\left((BP) \otimes (G^{-1}A^{\top}S)\right)\right) = d,  dim\left(Null\left((P^{\top}B^{\top}) \otimes (S^{\top}A)\right)\right)= mn-d.$$

\begin{lemma}\label{lem:7.1}
	With respect to the geometry induced by the $\left(I \otimes G\right)$-inner product, we have that
	\begin{enumerate}
		\item $Z_2 \otimes Z_1'$ projects orthogonally onto $d$-dimensional subspace $Range\left((BP) \otimes (G^{-1}A^{\top}S)\right).$
		\item $\left(I-Z_2 \otimes Z_1'\right)$ projects orthogonally onto $\left(mn-d\right)$-dimensional subspace $Null\left((P^{\top}B^{\top}) \otimes (S^{\top}A)\right).$
	\end{enumerate}
\end{lemma}
\begin{proof}
	See Appendix \ref{appendix1} for more details.
\end{proof}

\begin{lemma}\label{lem:7.1'} Let $\widetilde{Z_1}=GZ_1'$  and $\hat{Z_1}=G^{-1/2}\widetilde{Z_1}G^{-1/2}$ . For $Z_1,\ Z_1'$, $\widetilde{Z_1}$, $\hat{Z_1}$ and $Z_2$ ,  there exist the following relations.
	\begin{enumerate}
		\item $(Z_1')^2=Z_1'$,$Z_1=Z_1^{2},\ Z_1^{T}=Z_1$ and $Z_2=Z_2^{2},\ Z_2^{T}=Z_2$.
		\item $\widetilde{Z_1}=\widetilde{Z_1}G^{-1}\widetilde{Z_1}$,  $\widetilde{Z_1}^{T}=\widetilde{Z_1}$, $\hat{Z_1}=\hat{Z_1}^T$ and  $\hat{Z_1}^2=\hat{Z_1}$.
	\end{enumerate}
\end{lemma}
\begin{proof}
	We can verify them directly.
\end{proof}	
	
\section{Convergence analysis}	
	\label{sec:4}
Hereunder, we detail the convergence analysis for the scheme \ref{algorithm1}. From Lemma \ref{lem:7.1'}, it is easy to obtain the following relation.
\begin{equation}\label{rho}
	\begin{split}
		\rho & = 1- \lambda_{\min}\left(\mathbf{E}\left[Z_2 \otimes \hat{Z_1}\right]\right) \\
		& = 1- \lambda_{\min}\left(\left(I \otimes G^{-\frac{1}{2}}\right)\mathbf{E}\left[Z_2 \otimes \widetilde{Z_1}\right]\left(I \otimes G^{-\frac{1}{2}}\right)\right)  \\
		& = 1- \lambda_{\min}\left(\left(I \otimes G^{-1}\right)\mathbf{E}\left[Z_2 \otimes \widetilde{Z_1}\right]\right)  \\
		& = 1- \lambda_{\min}\left(\mathbf{E}\left[Z_2 \otimes Z_1'\right]\right).
	\end{split}
\end{equation}
Our convergence theorems depend on the above convergence rate $\rho$.

%
%
%
%
%
%
%

\subsection{Convergence theorem}

\begin{theorem}\label{thm:2.1}
	For every $X^*\in \mathbb{R}{^{m\times n}}$ satisfying $AX^*B=C$, we have
	\begin{equation*}
		\mathbf{E}\left[ \left\|X^{k}-X^*\right\|_{F(G)}^2 \right] \leq \rho^k \left\|X^{0}-X^*\right\|_{F(G)}^2,
	\end{equation*}
	where $\rho=1- \lambda_{\min}\left( \mathbf{E}\left[ Z_2 \otimes Z_1' \right] \right)$.
	Therefore, the iteration sequence generated by \eqref{eq:2.1.3} converges to $X^*$ if $0\leq\rho<1$.
\end{theorem}
\begin{proof}
	The iteration sequence \eqref{eq:2.1.3} can be rewritten as a simple fixed point formula
	\begin{equation}\label{eq:22}
		X^{k+1}-X^*=X^k-X^*-Z_1' \left(X^k -X^* \right)Z_2.
	\end{equation}
	According to the definition of Frobenius norm, we have
	\begin{equation}\label{eq:2.3}
		\begin{split}
			\left\|X^{k+1}-X^*\right\|_{F(G)}^2 = & \left\|X^{k}-X^*\right\|_{F(G)}^2  -Tr\left( \left(X^{k}-X^*\right)^\top G Z_1'\left(X^{k}-X^*\right)Z_2\right)\\
			&-Tr\left( Z_2^\top\left(X^{k}-X^*\right)^\top (Z_1')^\top G\left(X^{k}-X^*\right)\right)+\left\|Z_1'\left(X^{k}-X^*\right)Z_2\right\|_{F(G)}^2.
		\end{split}
	\end{equation}
	With Lemma \ref{lem:7.1'}, $Tr(MN)=Tr(NM)$ and $G^{\top}=G$, we can get
	\begin{equation*}
		\begin{split}
			\left\|Z_1'\left(X^{k}-X^*\right)Z_2\right\|_{F(G)}^2
			& = Tr\left( Z_2^{\top}\left(X^{k}-X^*\right)^{\top}Z_1'^{\top}G Z_1'\left(X^{k}-X^*\right)Z_2\right) \\
			& = Tr \left( Z_2^{\top}\left(X^{k}-X^*\right)^{\top} Z_1'^{\top}\widetilde{Z_1}\left(X^{k}-X^*\right)Z_2\right) \\
			& = Tr \left( \left(X^{k}-X^*\right)^{\top}\widetilde{Z_1}\left(X^{k}-X^*\right)Z_2 \right).
		\end{split}
	\end{equation*}
	Using the fact $Tr(M^\top N)=Tr(N^\top M)$, we have
	\begin{equation*}
		Tr\left( \left(X^{k}-X^*\right)^\top G Z_1'\left(X^{k}-X^*\right)Z_2\right)=Tr\left( \left(X^{k}-X^*\right)^\top \widetilde{Z_1}\left(X^{k}-X^*\right)Z_2\right),
	\end{equation*}
	\begin{equation*}
		\begin{split}
			Tr\left( Z_2^\top\left(X^{k}-X^*\right)^\top (Z_1')^\top G\left(X^{k}-X^*\right)\right)=Tr\left( \left(X^{k}-X^*\right)^\top \widetilde{Z_1} \left(X^{k}-X^*\right)Z_2\right).
		\end{split}
	\end{equation*}
	Then, by substituting the above three equations into \eqref{eq:2.3} and using the properties of Kronecker product, we can get
	\begin{equation*}
		\begin{split}
			\left\|X^{k+1}-X^*\right\|_{F(G)}^2 & = \left\|X^{k}-X^*\right\|_{F(G)}^2-\left\|Z_1'\left(X^{k}-X^*\right)Z_2\right\|_{F(G)}^2 \\
			& = \left\|X^{k}-X^*\right\|_{F(G)}^2-\left\| vec \left( Z_1'\left(X^{k}-X^*\right)Z_2\right) \right\|_{I\otimes G}^2 \\
			& = \left\|X^{k}-X^*\right\|_{F(G)}^2-\left\| \left(Z_2^{\top} \otimes Z_1'  \right) vec \left(X^{k}-X^*\right) \right\|_{I\otimes G}^2,
		\end{split}
	\end{equation*}
	From  $\left\|A\right\|_{F(G) }^2=\left\| vec \left(A\right)\right\|_{I\otimes G}^2$, the second equation holds.
	
	Taking conditional expectations, we get
	\begin{equation}\label{eq:2.4}
		\mathbf{E}\left[ \left\|X^{k+1}-X^*\right\|_{F(G)}^2 \right]= \left\|X^{k}-X^*\right\|_{F(G)}^2-\mathbf{E}\left[ \left\| \left(Z_2^{\top} \otimes Z_1'  \right) vec \left(X^{k}-X^*\right) \right\|_{I\otimes G}^2 \right].
	\end{equation}
	By Lemmas \ref{lem:1.1} and  \ref{lem:7.1'}, it holds
	\begin{equation}\label{eq:2.1.10}
		\begin{split}
			& \mathbf{E}\left[ \left\| \left(Z_2^{\top} \otimes Z_1'  \right) vec \left(X^{k}-X^*\right) \right\|_{I\otimes G}^2 \right] \\
			=  &\mathbf{E}\left[vec \left(X^{k}-X^*\right)^{\top} \left(Z_2^{\top} \otimes Z_1' \right)^{\top} (I\otimes G) \left(Z_2^{\top} \otimes Z_1'  \right) vec \left(X^{k}-X^*\right) \right] \\
			= & vec \left(X^{k}-X^*\right)^{\top} (I\otimes G^{1/2}) \mathbf{E}\left[\left(Z_2^{T} \otimes (G^{-1/2}\widetilde{Z_1}G^{-1/2} ) \right)\right] (I\otimes G^{1/2}) vec \left(X^{k}-X^*\right).
		\end{split}
	\end{equation}
	With the symmetries of $ Z_2$ and $ G^{-1/2}\widetilde{Z_1}G^{-1/2}$ in Lemma \eqref{lem:7.1'}, it results in
	\begin{equation}\label{eq:2.5}
		\begin{split}
			\mathbf{E}\left[ \left\| \left(Z_2^{\top} \otimes Z_1'  \right) vec \left(X^{k}-X^*\right) \right\|_{I\otimes G}^2 \right] &	\geq 	\lambda_{\min}\left( \mathbf{E}\left[ Z_2^{\top} \otimes \hat{Z_1} \right] \right) \left\|(I\otimes G^{1/2})vec \left(X^{k}-X^*\right) \right\|_2^2 \\
			&=  \rho_c \left\|X^{k}-X^*\right\|_{F(G)}^2,
		\end{split}
	\end{equation}
	where $\left\|(I\otimes G^{1/2}) vec \left(X^{k}-X^*\right) \right\|_2^2 = \left\| vec \left(X^{k}-X^*\right) \right\|_{F(G)}^2$ and $\rho_c = \lambda_{\min}\left( \mathbf{E}\left[ Z_2^{\top} \otimes \hat{Z_1} \right] \right)$.
	For the inequality, the following estimate
$\lambda_{\min}=\underset{x\neq 0}{\min}\dfrac{x^{\top}Ax}{x^{\top}x} 	$
	is used.
	
	Therefore, combining \eqref{eq:2.4} and \eqref{eq:2.5}, we can obtain an estimate as follows
	\begin{equation*}
		\mathrm{E}\left[ \left\|X^{k+1}-X^*\right\|_{F(G)}^2 \right]=\left(1-\rho_c \right)\left\|X^{k}-X^*\right\|_{F(G)}^2.
	\end{equation*}
	Taking the full expectation of both sides, we can get that
	\begin{equation*}
		\mathbf{E}\left[ \left\|X^{k+1}-X^*\right\|_{F(G)}^2 \right]=\left(1-\rho_c \right) \mathbf{E}\left[ \left\|X^{k}-X^*\right\|_{F(G)}^2 \right].
	\end{equation*}
	By induction on the above process, the proof is completed.
\end{proof}

\begin{theorem}\label{thm:2}
	For every $X^*\in \mathbb{R}{^{m\times n}}$ satisfying $AX^*B=C$, we have the norm of expectation as follows
	\begin{equation*}
		\left\|\mathbf{E}\left[X^{k+1}-X^* \right]\right\|_{F(G)} \leq \rho^{k}  \left\|X^{0}-X^*\right\|_{F(G)},
	\end{equation*}
	where $\rho = \lambda_{\max} \left(I- \mathbf{E}\left[ Z_2 \otimes Z_1'\right] \right)=1-\lambda_{\min}  (\mathbf{E}\left[ Z_2 \otimes Z_1'\right])$.
\end{theorem}
\begin{proof}
	By the Kronecker product, the iterative formula \eqref{eq:2.1.3} can be written as follows
	\begin{equation}\label{eq:22.1}
		vec\left(X^{k+1}-X^*\right)=\left(I-\left( Z_2^{\top}\otimes Z_1'\right) \right)vec\left(X^k-X^*\right).
	\end{equation}
	It is evident that the transform $X \rightarrow vec\left(X\right)$ gives a linear isomorph of $\mathbb{R}{^{m\times n}} \rightarrow \mathbb{R}{^{mn}}$.
	Since $Z_2^{\top}=Z_2$, taking expectations conditioned on $X^k$ in \eqref{eq:22.1} we have
	\begin{equation*}
		\mathbf{E}\left[vec\left(X^{k+1}-X^*\right)|vec\left(X^k\right) \right] = \left(I-\mathbf{E}\left[ Z_2\otimes Z_1'\right] \right)vec\left(X^k-X^*\right).
	\end{equation*}
	Taking expectations again gives
	\begin{equation*}
		\mathbf{E}\left[vec\left(X^{k+1}-X^*\right) \right] = \left(I-\mathbf{E}\left[ Z_2\otimes Z_1'\right] \right)\mathbf{E}\left[vec\left(X^k-X^*\right)\right].
	\end{equation*}
	Applying Lemma \ref{lem:7.1'} and the norms to both sides we obtain the estimate
	\begin{equation*}
		\begin{split}
			\left\|\mathbf{E}\left[vec\left(X^{k+1}-X^*\right) \right]\right\|^2_{I\otimes G} &=\left\|\left(I-\mathbf{E}\left[ Z_2\otimes Z_1'\right] \right)\mathbf{E}\left[vec\left(X^k-X^*\right)\right] \right\|^2_{I\otimes G}
			\\& \leq \left\|\left(I-\mathbf{E}\left[ Z_2\otimes Z_1'\right] \right) \right\|^2_{I\otimes G} \left\| \mathbf{E}\left[vec\left(X^k-X^*\right)\right] \right\|^2_{I\otimes G}
		\end{split}
	\end{equation*}

	\begin{equation*}
		\begin{split}
			\left\|\left(I-\mathbf{E}\left[ Z_2\otimes Z_1'\right] \right) \right\|^2_{I\otimes G}
			&= \underset{\left\|(I\otimes G^{1/2})x\right\|_2=1}{\max}\left\|(I\otimes G^{1/2})\left(I-\mathbf{E}\left[ Z_2\otimes Z_1'\right] \right) x\right\|^2_2.
		\end{split}
	\end{equation*}
	
	Substituting $y=G^{1/2}x$ in the above gives
	\begin{equation*}
		\begin{split}
			\left\|\left(I-\mathbf{E}\left[ Z_2\otimes Z_1'\right] \right) \right\|^2_{I\otimes G}
			& = \underset{\left\| y\right\|_2=1}{\max}\left\|(I\otimes G^{1/2})\left(I-\mathbf{E}\left[ Z_2\otimes Z_1'\right] \right)(I\otimes G^{-1/2}) y\right\|^2_2\\
			&=\underset{\left\| y\right\|_2=1}{\max}\left\|\left(I-\mathbf{E}\left[ Z_2\otimes (G^{-1/2}\widetilde{Z_1}G^{-1/2})\right] \right) y\right\|^2_2\\
			&=\lambda_{\max}^{2}\left(I-\mathbf{E}\left[ Z_2\otimes (G^{-1/2}\widetilde{Z_1}G^{-1/2})\right] \right),
		\end{split}
	\end{equation*}
	the third equality we used the symmetry of $I-\mathbf{E}\left[ Z_2\otimes (G^{-1/2}\widetilde{Z_1}G^{-1/2})\right]$ when
	passing from the operator norm to the spectral radius. Note that the symmetry of $\mathbf{E}\left[ Z_2\otimes (G^{-1/2}\widetilde{Z_1}G^{-1/2})\right]$
	derives from the symmetries of $ Z_2$ and $ G^{-1/2}\widetilde{Z_1}G^{-1/2}$ in Lemma \ref{lem:7.1'}.
	Considering that the vector operator is isomorphic and $\left\|vec(X^k-X^*)\right\|^2_{I\otimes G}=\left\|X^k-X^*\right\|^2_{F(G)}$, with the formula \eqref{rho} we have
	\begin{equation*}
		\begin{split}
			\left\|\mathbf{E}\left[X^{k+1}-X^* \right] \right\|_{F(G)} & = \left\|\mathbf{E}\left[vec\left(X^{k+1}-X^*\right) \right]\right\|_{I\otimes G} \\
			& \leq \rho \left\| \mathbf{E}\left[X^k-X^*\right] \right\|_{F(G)}.
		\end{split}
	\end{equation*}
	By induction, the conclusion follows. The proof is completed.
\end{proof}

\subsection{Convergence rate and convergence conditions}

 To show that the rate $\rho$ is meaningful, in Lemma \ref{lem:7.2} we prove that $0\leq \rho \leq1$. We also provide a meaningful lower bound for $\rho$.

\begin{theorem}\label{lem:7.2}
	The quantity $\rho = 1-\lambda_{\min}\left(Z_2 \otimes Z_1'\right)$ satisfies
	\begin{equation*}
		0\leq 1-\frac{\mathbf{E}\left[d\right]}{mn}\leq \rho \leq 1,
	\end{equation*}
	where $d=rank\left((P^{\top}B^{\top}) \otimes (S^{\top}A)\right)$.
\end{theorem}
\begin{proof}
	Through Lemma \ref{lem:7.1'} we known $Z_2 \otimes \hat{Z_1}$ is a projection, then we get
	\begin{equation*}
		\left[\left(I \otimes G^{-1/2}\right)\left(Z_2 \otimes \widetilde{Z_1}\right)\left(I \otimes G^{-1/2}\right)\right]^2 = \left(I \otimes G^{-1/2}\right)\left(Z_2 \otimes \widetilde{Z_1}\right)\left(I \otimes G^{-1/2}\right),
	\end{equation*}
	where$\widetilde{Z_1}=GZ_1$, whence the spectrum of $\left(I \otimes G^{-1/2}\right)\left(Z_2 \otimes \widetilde{Z_1}\right)\left(I \otimes G^{-1/2}\right)$ is contained in $\left\{0,1\right\}$. Using this, combined with the fact that the mapping $A\rightarrow \lambda_{\max}\left(A\right)$ is convex on the set of symmetric matrices and Jensen's inequality, we have
	\begin{equation*}
		\begin{split}
			\lambda_{\max}\left(\mathbf{E}\left[Z_2 \otimes Z_1'\right]\right) & = \lambda_{\max}\left(\left(I \otimes G^{-1/2}\right)\mathbf{E}\left[Z_2 \otimes \widetilde{Z_1}\right]\left(I \otimes G^{-1/2}\right)\right) \\
			& \leq \mathbf{E}\left[\lambda_{\max}\left(Z_2 \otimes G^{-1/2}\widetilde{Z_1} G^{-1/2}\right)\right] \leq 1.
		\end{split}
	\end{equation*}
	Analogously, with the convexity of the mapping $A\rightarrow -\lambda_{\min}\left(A\right)$, it holds $\lambda_{\min}\left(\mathbf{E}\left[Z_2 \otimes Z_1'\right]\right) \geq0$. Thus $\lambda_{\min}\left(\mathbf{E}\left[Z_2 \otimes Z_1'\right]\right) \in \left[0,1\right]$ which implies $0\leq \rho \leq1$. To the lower bound, we use the fact that the trace of a matrix is the sum of its eigenvalues, and have
	\begin{eqnarray*}
		\mathbf{E}\left[Tr\left(Z_2 \otimes Z_1'\right)\right] = Tr\left(\mathbf{E}\left[Z_2 \otimes Z_1'\right]\right) \geq mn\lambda_{\min}\left(\mathbf{E}\left[Z_2 \otimes Z_1'\right]\right).
	\end{eqnarray*}
	Since $Z_2 \otimes Z_1'$ is project on a d-dimensional subspace from Lemma \ref{lem:7.1}, it results in $Tr\left(Z_2 \otimes Z_1'\right)=d$. Thus, from the above formula, we can get
	\begin{equation*}
		\rho = 1-\lambda_{\min}\left(\mathbf{E}\left[Z_2 \otimes Z_1'\right]\right) \geq 1-\frac{\mathbf{E}\left[d\right]}{mn}.
	\end{equation*}
\end{proof}

\begin{lemma}\label{lem:7.3}
	If $\mathbf{E}\left[Z_2 \otimes \widetilde{Z_1}\right]$ is invertible, then $\rho=1- \lambda_{min}\left(\mathbf{E}\left[Z_2 \otimes Z_1'\right]\right) < 1$, $B^{\top} \otimes A$ and $(BP)^{\top} \otimes (S^{\top}A)$ have full column rank, and $X^*$ is
	unique.
\end{lemma}
\begin{proof}
See Appendix for more details.
\end{proof}
\begin{lemma}[\cite{Golub2014}]\label{lem:2.0} If $A\in \mathbb{R}^{n\times n}$ is symmetric positive definite and $X\in \mathbb{R}^{n\times k}$ has rank $k$, then
	$B =X^{\top} AX\in \mathbb{R}^{k\times k}$ is also symmetric positive definite.
\end{lemma}
\begin{lemma}\label{EP}
	For an arbitrary symmetric positive definite matrix $A$, there exists 
	$\mathbf{E}\left[A^2\right]\succeq \left(\mathbf{E}\left[A\right]\right)^{\top}\mathbf{E}\left[A\right].$
\end{lemma}	
\begin{proof}
	See Appendix for more details.
\end{proof}

\begin{lemma}\label{lem:2.01} $(BP)^{\top} \otimes (S^{\top}A)$ is full column rank if and only if $\mathbf{E}\left[Z_2 \otimes \widetilde{Z_1}\right]$ is symmetric positive definite .
\end{lemma}
\begin{proof}
	Since $(BP)^{\top} \otimes (S^{\top}A)$ is full column rank and $G$ is symmetric positive definite, we can get $Z_2 \otimes \widetilde{Z_1}$ is symmetric positive definite by Lemma \ref{lem:2.0}. Since $Z_2 \otimes \widetilde{Z_1}$ is symmetric positive definite,  it holds $Z_2^{\top} \otimes \widetilde{Z_1}=H^{\top}H$ where $H$ is symmetric positive definite. Then,
	for every $y\in \mathbb{R}^{mn} \neq 0$,
	\begin{equation*}
		\begin{split}
			y^{\top}\mathbf{E}\left[\left(Z_2^{\top} \otimes \widetilde{Z_1} \right)\right]y&=
			y^{\top}\mathbf{E}\left[H^{\top}H \right]y\\
			&\geq y^{\top}\mathbf{E}\left[H\right]^{\top}\mathbf{E}\left[H\right]y\\
			&> 0.
		\end{split}
	\end{equation*}
%
The first inequality is obtained by Lemma \ref{EP}. It is easy to know that $y^{\top}\mathbf{E}\left[H\right]^{\top}\mathbf{E}\left[H\right]y\geq 0 $.
	If $y^{\top}\mathbf{E}\left[H\right]^{\top}\mathbf{E}\left[H\right]y=0,$ then $y^{\top}\mathbf{E}\left[H\right]y=0$, i.e., $y^{\top}Hy=0$ which contradicts with the property that $H$ is symmetric positive definite. Thus, the necessity exists. Meanwhile,
	the sufficiency is obtained by Lemma \ref{lem:7.3}.
\end{proof}

\begin{lemma}[\cite{GRM}]\label{lem:7.4}
	Let $Z = Z_2 \otimes  \widetilde{Z_1}$. If $\mathbf{E}\left[Z\right]$ is symmetric positive definite, then
	\begin{equation*}
		\left\langle\mathbf{E}\left[Z\right]y,y\right\rangle \geq \left(1-\rho\right)\left\|y\right\|^2_{I\otimes G}, \ \ \text{for all\ } y \in \mathbb{R}^{mn},
	\end{equation*}
	where $\rho = 1-\lambda_{\min}\left(Z_2 \otimes Z_1'\right)$ and $G$ symmetric positive definite.
\end{lemma}
\begin{proof}
Note that $\mathbf{E}\left[Z_2 \otimes \widetilde{Z_1}\right]$ and $G$ are symmetric positive definite, we get
		\begin{equation*}
		\begin{split}
			1-\rho&=1- \lambda_{\min}\left(\left(I \otimes G^{-\frac{1}{2}}\right)\mathbf{E}\left[Z_2 \otimes \widetilde{Z_1}\right]\left(I \otimes G^{-\frac{1}{2}}\right)\right) \\ & =\underset{t}{max}\left\{t|\left(I \otimes G^{-\frac{1}{2}}\right)\mathbf{E}\left[Z_2 \otimes \widetilde{Z_1}\right]\left(I \otimes G^{-\frac{1}{2}}\right)-tI\succeq 0 \right\}  \\
			& =\underset{t}{max}\left\{t|\mathbf{E}\left[Z_2 \otimes \widetilde{Z_1}\right]-t (I\otimes G)\succeq 0 \right\}.  \\
				\end{split}
	\end{equation*}
			Therefore, $\mathbf{E}\left[Z_2 \otimes \widetilde{Z_1}\right]\succeq I\otimes G$, and the conclusion is obtained.
\end{proof}

\begin{theorem}\label{lem:7.5}
	If $(BP)^{\top} \otimes (S^{\top}A)$ is full column rank, then
	\begin{equation*}
		\mathbf{E}\left[\left\|X^k - X^*\right\|^2_{F(G)}\right] \leq \rho^k \left\|X^0 - X^*\right\|^2_{F(G)},
	\end{equation*}
	where $\rho< 1$ is given in Lemma \ref{lem:7.3}.
\end{theorem}

\begin{proof}
	Since $(BP)^{\top} \otimes (S^{\top}A)$ is full column rank, from Lemma \ref{lem:2.01} we know that  $\mathbf{E}\left[Z_2 \otimes \widetilde{Z_1}\right]$ is symmetric positive definite. Let $r^k = vec\left(X^k-X^*\right)$. From \eqref{eq:2.1.3} we have $r^{k+1} = r^k - \left(Z_2^{\top} \otimes Z_1'\right)r^k$. Taking expectation in $\left\|r^{k+1}\right\|^2_{I\otimes G}$ conditioned on $r^k$ gives
	\begin{equation*}
		\begin{split}
			\mathbf{E}\left[\left\|r^{k+1}\right\|^2_{I\otimes G} | r^k\right] & = \mathbf{E}\left[\left\|\left(I-Z^{\top}_2 \otimes Z_1'\right)r^k \right\|^2_{I\otimes G} | r^k\right] \\
			& = \mathbf{E}\left[\left\langle \left[\left(I \otimes G\right)-\left(Z_2^{\top} \otimes GZ_1'\right)\right]r^k, r^k \right\rangle | r^k\right]  \\
			& = \left\|r^k\right\|^2_{I\otimes G} - \left\langle\mathbf{E}\left[Z_2^{\top} \otimes \widetilde{Z_1}\right]r^k, r^k\right\rangle  \\
			& \leq \rho \left\|r^k\right\|^2_{I\otimes G}  \ \ \left(by\  Lemma\ \ref{lem:7.4}\right).
		\end{split}
	\end{equation*}
	Since $\left\|r^k\right\|^2_{I\otimes G}=\left\|X^k-X^*\right\|^2_{F(G)}$, then the conclusion follows.
	
\end{proof}

\begin{remark}
	When $A$ is full column rank or $x\in Range(A^{\top})$,  the following estimate holds
$	\left\|Ax\right\|_2^2 \ge \sigma_{\min}^2(A)\|x\|_2^2.
$ For one case of the matrix with full column rank, \ref{lem:7.4} give the convergence of the generalized method.
For the other case, some conclusions are given in the following.  For special methods, some results have been obtained (for the GRBK method, see \cite{NYQ2022}). 
\end{remark}
\begin{lemma}[\cite{DK2021}]\label{lem:7.0}Let $A\in \mathbb{R}^{p\times m}$ and $B \in \mathbb{R}^{n\times q}$ be given. Denote
	\begin{equation}\label{set:M}
		\mathcal{M}=\left\{M\in\mathbb{R}^{m\times n} |\exists Y\in\mathbb{R}^{p\times q}\ s.t. \ M=A^{\top}YB^{\top}\right\}.
	\end{equation}
	Then, for any matrix $M\in \mathcal{M}$, it holds
	\begin{equation*}
		\left\|AMB\right\|^2_{F} \geq \sigma_{\min}^2\left(A\right)\sigma_{\min}^2\left(B\right)\left\|M\right\|^2_{F}.
	\end{equation*}
\end{lemma}	
\begin{lemma}\label{M12}
	Let the two sets $M_1$ and $M_2$ be defined by
	$$M_1=\left\{X_1\in \mathbf{R}^{m\times n} | A^{\top}Y_1B^{\top}=X_1\  \text{ for some } Y_1\in \mathbf{R}^{p\times q}\right\}, M_2=\left\{X_2\in \mathbf{R}^{m\times n} |  A^{\top}AY_2BB^{\top}=X_2 \ \text{  for some } Y_2\in \mathbf{R}^{m\times n}\right\}.$$ Then, it holds $M_1=M_2$.
\end{lemma}	
\begin{proof}
	See Appendix for more details.
\end{proof}	
\begin{theorem}\label{theB.2}
	Assume that $Z_1'$ and $Z_2$ are independent random variables. Let $$
	\mathcal{\bar{M}}=\left\{\bar{M}\in\mathbb{R}^{m\times n} |\exists Y\in\mathbb{R}^{p\times q}\ s.t. \ \bar{M}=\mathbf{E}\left[ Z_1'\right]  Y\mathbf{E}\left[ Z_2\right]\right\}.
	$$  If $(X^{k}-X^*)\in \mathcal{\bar{M}},$
	then for every $X^*\in \mathbb{R}{^{m\times n}}$ satisfying $AX^*B=C$, we have
	\begin{equation*}
		\mathbf{E}\left[ \left\|X^{k}-X^*\right\|_{F(G)}^2 \right] \leq \rho_{\sigma}^k \left\|X^{0}-X^*\right\|_{F(G)}^2.
	\end{equation*}
	with $\rho_{\sigma}=1-\sigma^2_{\min}\left( \mathbf{E}\left[ Z_2 \otimes Z_1' \right] \right) <1$.
	Therefore, the iteration sequence generated by \eqref{eq:2.1.4} converges to $X^*$.
\end{theorem}
\begin{proof}
	From Lemma \ref{lem:7.1'}, we can get $Z_2^{T} \otimes \hat{Z_1}=\left(Z_2^{T} \otimes \hat{Z_1}\right)^{T}$, $\left(Z_2^{T} \otimes \hat{Z_1}\right)^2=Z_2^{T} \otimes \hat{Z_1}$ . Thus, for every $y\in R^{mn}$, we have
	\begin{equation*}
		\begin{split}
			y^{T}\mathbf{E}\left[\left(Z_2^{T} \otimes \hat{Z_1}  \right)\right]y&=
			y^{T}\mathbf{E}\left[\left(Z_2^{T} \otimes   \hat{Z_1}\right)^{\top}\left(Z_2^{T} \otimes \hat{Z_1} \right)\right]y\\
			&\geq y^{T}\mathbf{E}\left[\left(Z_2^{T}\otimes \hat{Z_1}  \right)\right]^{\top}\mathbf{E}\left[\left(Z_2^{T} \otimes \hat{Z_1}  \right)\right]y.
		\end{split}
	\end{equation*}
	
	Using \eqref{eq:2.1.10}, we can get
	\begin{equation}\label{eq:2.51}
		\begin{split}
			\mathbf{E}\left[ \left\| \left(Z_2^{\top} \otimes Z_1'  \right) \text{vec} \left(X^{k}-X^*\right) \right\|_{I\otimes G}^2 \right] & \geq \left\|\mathbf{E}\left[\left(Z_2^{T} \otimes \hat{Z_1}  \right)\right](I\otimes G^{1/2}) \text{vec} \left(X^{k}-X^*\right)\right\|_2^2\\
			&= \left\|\mathbf{E}\left[\hat{Z_1} \right]\hat{R^k} \mathbf{E}\left[ Z_2\right] \right\|_F^2 \\
			&\geq \sigma_{\min}^2 \left(\mathbf{E}\left[ Z_2\right]\right) \sigma_{\min}^2 \left(\mathbf{E}\left[ \hat{Z_1}\right]\right)\left\|\hat{R^k} \right\|_F^2 \\
			&= \sigma^2_{\min}\left( \mathbf{E}\left[ Z_2 \otimes Z_1' \right] \right)\left\|X^{k}-X^*\right\|_{F(G)}^2,
		\end{split}
	\end{equation}
	where $\hat{R^k}=G^{1/2}(X^{k}-X^*)$. The first inequality is obtained by $ \mathbf{E}\left[\left\|X-\mathbf{E}\left[X\right]\right\|^{2}\right]=\mathbf{E}\left[\left\|X\right\|^{2}\right]-\left\|\mathbf{E}\left[X\right]\right\|^{2}$. From $(X^{k}-X^*)\in \mathcal{\bar{M}}$, $\hat{Z_1}=\hat{Z_1}^T$ and  $Z_2^T=Z_2$, it yields  $$\hat{R^k}\in
	\mathcal{\hat{M}}=\left\{\hat{M}\in\mathbb{R}^{m\times n} |\exists Y\in\mathbb{R}^{p\times q}\ s.t. \ \hat{M}=\mathbf{E}\left[ \hat{Z}_1\right]^T  Y\mathbf{E}\left[ (Z_2)\right]^T\right\}.
	$$It is easy to know that $vec(\hat{R^k})=(I\otimes G^{1/2})\text{vec}\left(X^{k}-X^*\right)$,  $\left\|\hat{R^k}\right\|_F^2 = \left\| X^{k}-X^* \right\|_{F(G)}^2 $.
	Then, with the use of Lemma \ref{lem:7.0} the second inequality holds. Since $\sigma_{\min} \left(\mathbf{E}\left[ \hat{Z}_1\right]\right)=\sigma_{\min} \left(\mathbf{E}\left[ Z_1'\right]\right)$, we have $\sigma_{\min}^2 \left(\mathbf{E}\left[ Z_2\right]\right) \sigma_{\min}^2 \left(\mathbf{E}\left[ \hat{Z_1}\right]\right)=\sigma^2_{\min}\left( \mathbf{E}\left[ Z_2 \otimes Z_1' \right] \right).$ 
	With the similar process of Theorem \ref{thm:2.1}, the conclusion is obtained.
\end{proof}

\begin{remark}
It is easy to prove that $(X^{k}-X^*)\in \mathcal{\bar{M}}$ is equivalent to $(I\otimes G^{1/2}) \text{vec} \left(X^{k}-X^*\right)\in Range\left(\mathbf{E}\left[\left(Z_2^{T} \otimes \hat{Z_1}  \right)\right]\right)$.	In the global randomized Kaczmarz method,
$\mathbf{E}\left[ Z_1'\right]=\mathbf{E}\left[ \hat{Z}_1\right]=\frac{ A ^{\top} A }{\left\|A \right\|^2_F}$, $\mathbf{E}\left[ Z_2\right]=\frac{BB^T }{\left\|B \right\|^2_F}$, it has
$$\mathcal{\bar{M}}=\left\{\bar{M}\in\mathbb{R}^{m\times n} |\exists Y\in\mathbb{R}^{p\times q}\ s.t. \ \bar{M}=A^T A Y BB^T\right\}.$$
Obviously, $\mathcal{\bar{M}}$ is well defined because $X^{0}=O\in \mathcal{\bar{M}}$ and $A^{\dagger}CB^{\dagger} \in \mathcal{\bar{M}}$. Meanwhile, it is easy to verify that	$(X^{k}-X^*)\in \mathcal{M}$ in \eqref{set:M} with the iteration scheme \eqref{eq:2.1.4} for the global randomized Kaczmarz method. Then it results in $(X^{k}-X^*)\in \mathcal{\bar{M}}$ from Lemma \ref{M12}.
\end{remark}
	
	%
	

\subsection{Convergence with convenient probabilities}

\begin{definition} [\cite{GRM}]
Let the random matrix $S,P$ be discrete distributions. $(S,P)$ will be called a complete discrete sampling pair if $S=S_i \in \mathbb{R}^{p\times \tau_{1i}}$ with probability $p^1_i>0$, where $S_i^{\top}A$ has full row rank and $\tau_{1i} \in \mathbb{N}$ for $i = 1,...,r
_1$.  ${\bf S}=:\left[S_1,...,S_{r_1}\right]\in \mathbb{R}^{p\times \Sigma^{r_1}_{i=1}\tau_{1i}}$ is such that $A^{\top}{\bf S}$ has full row rank. Meanwhile, 
 $P=P_i \in \mathbb{R}^{q\times \tau_{2i}}$ with probability $p^2_i>0$, where $P_i^{\top}B^{\top}$ has full row rank and $\tau_{2i} \in \mathbb{N}$ for $i = 1,...,r_2$. ${\bf P}=:\left[P_1,...,P_{r_2}\right]\in \mathbb{R}^{q\times \Sigma^{r_2}_{i=1}\tau_{2i}}$ is such that $B{\bf P}$ has full row rank.
\end{definition}

Assume that $S, P$ is a complete discrete sampling pair, then $S^{\top}A$ and $(BP)^{\top}$ have full row rank and
\begin{eqnarray*}
	\left( S^{\top}AA^{\top}S \right)^{\dagger} &=& \left( S^{\top}AA^{\top}S \right)^{-1} ,\\
	\left( P^{\top}B^{\top}BP \right)^{\dagger} &=& \left( P^{\top}B^{\top}BP \right)^{-1} .
\end{eqnarray*}
Therefore we replace the pseudoinverse in \eqref{eq:2.1.4} by the inverse. 
Define
\begin{equation}\label{eq:2.6}
	D_S=\text{diag} \left( \sqrt{p^1_1}\left( \left( S_1 \right)^{\top}AA^{\top}S_1 \right)^{-1/2},..., \sqrt{p^1_{r_1}}\left( \left( S_{r_1} \right)^{\top}AA^{\top}S_{r_1} \right)^{-1/2} \right),
\end{equation}
\begin{equation}\label{eq:2.8}
	D_P=\text{diag} \left( \sqrt{p^2_1}\left( \left( P_1 \right)^{\top}B^{\top}BP_1 \right)^{-1/2},\ ...\ , \sqrt{p^2_{r_2}}\left( \left( P_{r_2} \right)^{\top}B^{\top}BP_{r_2} \right)^{-1/2} \right),
\end{equation}
where $D_S$ and $D_P$ are block diagonal matrices, and are well defined and invertible, as $S_i^{\top}A$ has full row rank for $i=1,...,r_1$ and $P_j^{\top}B^{\top}$ has full row rank for $j=1,...,r_2$.
Taking the expectation of $Z_1$ and $Z_2$, we get
\begin{equation}\label{eq:2.7}
	\begin{split}
		\mathbf{E}\left[ Z_1 \right] & = \sum\limits_{i=1}^{r_1} A^{\top}S_i\left( S_i^{\top}AA^{\top}S_i \right)^{-1}S_i^{\top}Ap^1_i \\
		& = A^{\top} \left( \sum\limits_{i=1}^{r_1} S_i\sqrt{p^1_i}\left( S_i^{\top}AA^{\top}S_i \right)^{-1/2}\left( S_i^{\top}AA^{\top}S_i \right)^{-1/2} \sqrt{p^1_i}S_i^{\top} \right)A \\
		& = \left( A^{\top}{\bf S}D_S \right) \left( D_S{\bf S^{\top}}A \right),
	\end{split}
\end{equation}
and
\begin{equation}\label{eq:2.9}
	\begin{split}
		\mathbf{E}\left[ Z_2 \right] & = \sum\limits_{j=1}^{r_2} BP_j\left( P_j^{\top}B^{\top}BP_j \right)^{-1}P_j^{\top}B^{\top}p^2_j \\
		& = B \left( \sum\limits_{j=1}^{r_2} P_j\sqrt{p^2_j}\left( P_j^{\top}B^{\top}BP_j \right)^{-1/2}\left( P_j^{\top}B^{\top}BP_j \right)^{-1/2} \sqrt{p^2_j}P_j^{\top} \right)B^{\top} \\
		& = \left( B{\bf P}D_P \right) \left( D_P{\bf P}^{\top}B^{\top} \right).
	\end{split}
\end{equation}
Since $A^{\top} {\bf S}$ and $B {\bf P}$ have full row rank, and $D_S,\ D_P$ are invertible, we can get $\mathbf{E} \left[ Z_i \right], i=1,2$ are symmetric positive definite.
So, complete discrete sampling pairs guarantee the convergence of the resulting methods.

Next we develop a choice of probability distribution that yields a convergence rate that is easy to interpret. 
This result is new
and covers a wide range of methods, including randomized Kaczmarz method and randomized coordinate descent method, as well as their block variants.
However, it is more general and covers many other possible particular algorithms, which arise by choosing two particular sets of sample matrices $S_i$ for $i=1,...,r_1$ and $P_j$ for $j=1,...,r_2$.

\begin{theorem}\label{thm:2.2}
	Assume that $(S,P)$ is a complete discrete sampling pair with the following probabilities, respectively,
	\begin{equation}\label{eq:2.10}
		p^1_i=\frac{Tr\left( S_i^{\top}AA^{\top}S_i \right)}{\left\| A^{\top}{\bf S} \right\|^2_F}\ \ \text{for}\ \ i=1,...,r_1,\ 	p^2_j=\frac{Tr\left( P_j^{\top}B^{\top}BP_j \right)}{\left\| B{\bf P} \right\|^2_F}\ \ \text{for}\ \ j=1,...,r_2,
	\end{equation}
		Then the formula \eqref{eq:2.1.4} satisfies
	\begin{equation}\label{eq:2.12}
		\mathbf{E} \left[ \left\| X^k-X^* \right\|^2_F \right]\leq \rho^k \left\| X^0-X^* \right\|^2_F,
	\end{equation}
	where
	\begin{equation}\label{eq:2.13}
		\rho = 1-\frac{\lambda_{\min}\left( {\bf S}^{\top}AA^{\top}{\bf S} \right) \lambda_{\min}\left( {\bf P}^{\top}B^{\top}B{\bf P} \right) }{\left\| A^{\top}{\bf S}\right\|^2_F  \left\| B{\bf P} \right\|^2_F}.
	\end{equation}
\end{theorem}

\begin{proof}
	Let $u_i=Tr\left( \left(S_i\right)^{\top}AA^{\top}S_i \right)$, $v_j=Tr\left( \left(P_j\right)^{\top}B^{\top}BP_j \right)$. Taking \eqref{eq:2.10} into \eqref{eq:2.6} and \eqref{eq:2.8}, respectively, we have
	\begin{eqnarray*}
		D_S^2 &=& \frac{1}{\left\| A^{\top}S \right\|_F^2}\text{diag} \left( u_1\left( S_1^{\top}AA^{\top}S_1 \right)^{-1},...,u_{r_1}\left( S_{r_1}^{\top}AA^{\top}S_{r_1} \right)^{-1} \right),\\
		D_P^2 &=& \frac{1}{\left\| BP \right\|_F^2}\text{diag} \left( v_1\left( P_1^{\top}B^{\top}BP_1 \right)^{-1},...,v_{r_2}\left( P_{r_2}^{\top}B^{\top}BP_{r_2} \right)^{-1} \right).
	\end{eqnarray*}
	and thus
	\begin{eqnarray}\label{eq:2.14}
		\lambda_{\min} \left(D_S^2 \right) &=& \frac{1}{\left\| A^{\top}{\bf S} \right\|_F^2} \min\limits_{1\leq i \leq r_1} \left\{ \frac{u_i}{\lambda_{\max}\left( S_i^{\top}AA^{\top}S_i \right)} \right\} \geq \frac{1}{\left\| A^{\top}{\bf S} \right\|_F^2},\\
		\lambda_{\min} \left(D_P^2 \right) &=& \frac{1}{\left\| B{\bf P} \right\|_F^2} \min\limits_{1\leq j \leq r_2} \left\{ \frac{v_j}{\lambda_{\max}\left( P_j^{\top}B^{\top}BP_j \right)} \right\} \geq \frac{1}{\left\| B{\bf P} \right\|_F^2}.
	\end{eqnarray}
	Using the fact that for arbitrary matrices $M$, $N$ of appropriate sizes, $\lambda_{\min} \left(MN\right)=\lambda_{\min} \left(NM\right)$, hence
	\begin{equation*}
		\begin{split}
			\lambda_{\min} \left( \mathbf{E}\left[Z_1\right] \right) & = \lambda_{\min} \left( A^{\top}{\bf S}D_S^2{\bf S}^{\top}A\right) \\
			& = \lambda_{\min} \left( {\bf S}^{\top}AA^{\top}{\bf S}D_S^2 \right).
		\end{split}
	\end{equation*}
	Then, using the fact that $\lambda_{\min} \left(MN\right)=\lambda_{\min} \left(M\right) \lambda_{\min} \left(N\right)$ for $M,N\in \mathbb{R}^{n\times n}$ are symmetric positive definite, from \eqref{eq:2.14} we can obtain
	\begin{equation*}
		\begin{split}
			\lambda_{\min} \left( \mathbf{E}\left[Z_1\right] \right) & = \lambda_{\min} \left( {\bf S}^{\top}AA^{\top}{\bf S}D_S^2 \right) \\
			& \geq \lambda_{\min} \left( {\bf S}^{\top}AA^{\top} {\bf S}\right)\lambda_{\min} \left( D_S^2 \right) \\
			& \geq \frac{\lambda_{\min} \left( {\bf S}^{\top}AA^{\top}{\bf S} \right)}{\left\| A^{\top}{\bf S} \right\|_F^2}.
		\end{split}
	\end{equation*}
	Similarly, we can also obtain
	\begin{equation*}
		\begin{split}
			\lambda_{\min} \left( \mathbf{E}\left[Z_2\right] \right) & =  \lambda_{\min} \left( B{\bf P}D_P^2{\bf P}^{\top}B^{\top}\right)\\
			& = \lambda_{\min} \left({\bf P}^{\top}B^{\top}B{\bf P}D_P^2\right) \\
			& \geq \lambda_{\min} \left({\bf P}^{\top}B^{\top}B{\bf P}\right) \lambda_{\min} \left(D_P^2\right)\\
			& \geq \frac{\lambda_{\min} \left({\bf P}^{\top}B^{\top}B{\bf P}\right)}{\left\| B{\bf P} \right\|_F^2}.
		\end{split}
	\end{equation*}
	Since
	\begin{equation*}
		\begin{split}
			\rho_c & = \lambda_{\min}\left( \mathbf{E}\left[ Z_2^{\top}\otimes Z_1 \right] \right) \\
			& =  \lambda_{\min}\left( \mathbf{E}\left[ Z_2^{\top} \right]\otimes \mathbf{E}\left[ Z_1 \right] \right) \\
			& \geq \lambda_{\min} \left( \mathbf{E}\left[ Z_2 \right] \right) \lambda_{\min} \left(\mathbf{E}\left[ Z_1 \right]\right)\\
			& \geq \frac{\lambda_{\min} \left( {\bf S}^{\top}AA^{\top}{\bf S} \right) \lambda_{\min} \left({\bf P}^{\top}B^{\top}B{\bf P}\right)}{\left\| A^{\top}{\bf S} \right\|_F^2 \left\| B{\bf P} \right\|_F^2}.
		\end{split}
	\end{equation*}
	Hence, by Theorem \ref{thm:2.1} we have
	\begin{equation*}
		\begin{split}
			\mathbf{E} \left[ \left\| X^k-X^* \right\|^2_F \right] & \leq \left( 1-\rho_c \right) \left\| X^{k-1}-X^* \right\|^2_F \\
			& \leq \left( 1- \frac{\lambda_{\min} \left( {\bf S}^{\top}AA^{\top}{\bf S} \right) \lambda_{\min} \left({\bf P}^{\top}B^{\top}B{\bf P}\right)}{\left\| A^{\top}{\bf S} \right\|_F^2 \left\| B{\bf P} \right\|_F^2}\right) \left\| X^{k-1}-X^* \right\|^2_F \\
			& = \rho \left\| X^{k-1}-X^* \right\|^2_F,
		\end{split}
	\end{equation*}
	where $\rho=\left( 1- \frac{\lambda_{\min} \left( {\bf S}^{\top}AA^{\top}{\bf S} \right) \lambda_{\min} \left({\bf P}^{\top}B^{\top}B{\bf P}\right)}{\left\| A^{\top}{\bf S} \right\|_F^2 \left\| B{\bf P} \right\|_F^2}\right)$. Finally, taking full expectation and by induction, we can get
	\begin{equation*}
		\mathbf{E} \left[ \left\| X^k-X^* \right\|^2_F \right] \leq \rho^k \left\| X^{0}-X^* \right\|^2_F.
	\end{equation*}
	Obviously, the coefficient $\rho$ satisfies $0\leq \rho <1$, so the method is convergent.
\end{proof}

\section{Special cases: Examples}
\label{sec:5}

In this section we briefly mention how by selecting the parameters $S$ and $P$ of our method we recover several existing methods. Furthermore, we propose some similar methods based on discrete sampling pairs. The list is by no means comprehensive and merely serves the purpose of an illustration of the flexibility of our algorithm. 

\subsection{Global randomized Kaczmarz method}
\label{subsec:3.1}
If we choose $S_i=e_i$ (the unit coordinate vector in $\mathbb{R}^{p}$) and $P_j=e_j$ (the unit coordinate vector in $\mathbb{R}^{q}$), in view of \eqref{eq:1.3}, this results in
\begin{equation*}
	X^{k+1} = \arg\min\limits_{X \in\mathbb{R}{^{m\times n}}} \frac{1}{2}\left\| X-X^k\right\|_{F}^2 \ \  \text{subject to} \ \ A_{i,:}XB_{:,j}=C_{ij}.
\end{equation*}
With the use of \eqref{eq:2.1.4},  the iteration can be calculated with
\begin{equation*}
	X^{k+1}=X^k-\frac{\left( A_{i,:} \right)^{\top} \left( A_{i,:}X^kB_{:,j}-C_{ij} \right) \left( B_{:,j} \right)^{\top}}{\left\|A_{i,:} \right\|^2_2 \left\| B_{:,j} \right\|^2_2 }.
\end{equation*}
This is recovered global randomized Kaczmarz (GRK) method.

When $i, j$ are selected at random, this is the global randomized Kaczmarz method. Applying Theorem \ref{thm:2.2}, we see the probability distributions $p_i=\frac{\left\|A_{i,:} \right\|^2_2}{\left\|A \right\|^2_F} , p_j=\frac{\left\|B_{:,j} \right\|^2_2}{\left\|B\right\|^2_F}$ result in a convergence with
\begin{equation}\label{eq:3.1.1}
	\mathbf{E}\left[ \left\| X^k-X^* \right\|^2_F \right] \leq \left( 1- \frac{\lambda_{\min}\left( A^{\top}A\right)\lambda_{\min}\left( BB^{\top}\right)}{\left\| A\right\|_F^2\left\|B\right\|_F^2} \right)^k \left\| X^0-X^* \right\|^2_F.
\end{equation}
About details of another convergence proof of the GRK method, we refer the reader to references \cite{NYQ2022}. We also provide new convergence results which based on the convergence of the norm of the expected error. Applying Theorem \ref{thm:2} to the GRK method gives
\begin{equation}\label{eq:3.1.2}
	\left\|\mathbf{E}\left[ X^k-X^* \right]\right\|^2_F \leq \left( 1- \frac{\lambda_{\min}\left( A^{\top}A\right)\lambda_{\min}\left( BB^{\top}\right)}{\left\| A\right\|_F^2\left\|B\right\|_F^2} \right)^{2k} \left\| X^0-X^* \right\|^2_F.
\end{equation}
Thought the expectation is moved inside the norm, which is weaker form of convergence. We can find that the convergence rate appears squared, which means it is a better rate. Similar results for the convergence of the norm of the expected error holds for all the methods we present, and we will not repeat to illustrate this in following methods.

\subsection{Global randomized block Kaczmarz method}
Our framework also extends to block formulations of the global randomized Kaczmarz method. Let $\tau_1$ be a random subset of $\left[p\right]$, and $S=I_{:,\tau_1}$ be a column concatenation of the columns of the $p\times p$ identity matrix I indexed by $\tau_1$. Similarly, let $\tau_2$ be a random subset of $\left[q\right]$, and $P=I_{:,\tau_2}$ be a column concatenation of the columns of the $q\times q$ identity matrix I indexed by $\tau_2$. Then \eqref{eq:1.4} specializes to
\begin{equation*}
  X^{k+1} = \arg\min\limits_{X \in\mathbb{R}{^{m\times n}}} \frac{1}{2}\left\| X-X^k\right\|_{F}^2 \ \  \text{subject to} \ \ A_{\tau_1,:}XB_{:,\tau_2}=C_{\tau_1,\tau_2}.
\end{equation*}
In view of \eqref{eq:2.1.4}, this can be equivalently written as
\begin{equation}\label{eq:3.2.1}
  \begin{split}
    X^{k+1} =&  X^k+ \left(A_{\tau_1,:}\right)^{\top}\left( A_{\tau_1,:} \left( A_{\tau_1,:} \right)^{\top}\right)^{\dagger} \left( C_{\tau_1,\tau_2} - A_{\tau_1,:}X^kB_{:,\tau_2} \right)\\
    &\left( \left( B_{:,\tau_2}\right)^{\top} B_{:,\tau_2} \right)^{\dagger}\left( B_{:,\tau_2}\right)^{\top}\\
       =& X^k+  A_{\tau_1,:}^{\dagger} \left( C_{\tau_1,\tau_2}- A_{\tau_1,:}X^kB_{:,\tau_2}\right)  B_{:,\tau_2}^{\dagger}.
  \end{split}
\end{equation}
This is recovered the global randomized block Kaczmarz (GRBK) method in \cite{NYQ2022}.

\begin{remark}\label{rem:1}
Now, let the sizes of block index sets be $\left|I_k\right|=1$ and $\left|J_k\right|=q$, i.e., parameter matrices $S=e_i \in \mathbb{R}^{p}$ and $P=I \in \mathbb{R}^{q\times q}$. In this case, the index $i_k \in [p]$ is selected according to a probability distribution $\mathbb{P}\left(i_k \right)=\frac{\left\|A_{i_k,:}\right\|^2_2}{\left\|A\right\|_F^2}$. Then, the update \eqref{eq:3.2.1} becomes
\begin{equation*}
  X^{k+1} = X^k+ \frac{A_{i_k,:}^{\top}\left(C_{i_k,:}- A_{i_k,:}X^{k}B\right)B^{\dag}}{\left\|A_{i_k,:}\right\|^2_2},
\end{equation*}
which is called the randomized Kaczmarz method of matrix A (RK-A). Assume that $B$ has full row rank. Using Theorem \ref{thm:2.2}, we get the convergence rate in the expectation of the form
\begin{equation}\label{eq:3.2.2}
  \mathbf{E}\left[ \left\| X^k-X^* \right\|^2_F \right] \leq \left( 1- \frac{\lambda_{\min}\left(A^{\top}A \right)}{\left\|A\right\|^2_F} \right)^k \left\| X^0-X^* \right\|^2_F.
\end{equation}
\end{remark}

\begin{remark}\label{rem:2}
Similar to the RK-A method, with the block index sets size $\left|I_k\right|=p$ and $\left|J_k\right|=1$, the index $j_k \in [q]$ is selected according to a probability distribution $\mathbb{P}\left(j_k \right)=\frac{\left\|B_{:,j_k}\right\|^2_2}{\left\|B\right\|_F^2}$, we have the randomized Kaczmarz method of matrix B (RK-B) update as follows
\begin{equation*}
  X^{k+1} = X^k+ \frac{A^{\dag}\left(C_{:,j_k}- AX^{k}B_{:,j_k}\right)B_{:,j_k}^{\top}}{\left\|B_{:,j_k}\right\|^2_2}.
\end{equation*}
Suppose that $A$ is full column rank. Also, we get the convergence rate in the expectation of the form
\begin{equation}\label{eq:3.2.3}
  \mathbf{E}\left[ \left\| X^k-X^* \right\|^2_F \right] \leq \left( 1- \frac{\lambda_{\min}\left(BB^{\top}\right)}{\left\|B\right\|^2_F} \right)^k \left\| X^0-X^* \right\|^2_F.
\end{equation}
\end{remark}

Comparing \eqref{eq:3.2.2} and \eqref{eq:3.2.3} with the convergence rate \eqref{eq:3.1.1}, we find the convergence factors of the RK-A and RK-B methods are smaller than that of the GRK method.

\subsection{Randomized coordinate descent method}
In this subsection, by choosing different parameters $P, S, G $, we induce two randomized coordinate descent algorithms. In the following two cases, we assume that $B$ has full row rank.
\subsubsection{Positive definite case}
\label{ssubsec:3.3.1}
If $A$ is symmetric positive definite, then we can choose $G = A$, $P=I$ and $S = e_i$ in  \eqref{eq:1.3} and obtain
\begin{equation*}
X^{k+1} = \arg\min\limits_{X \in\mathbb{R}{^{m\times n}}}\frac{1}{2} \left\| X-X^k\right\|_{A}^2 \ \ \text{subject to} \ \ \left(e_i\right)^{\top}AXB=\left(e_i\right)^{\top}C.
\end{equation*}
where we use the symmetry of $A$ to get $\left(e_i\right)^{\top} A = A_{i,:} = \left(A_{:,i}\right)^{\top}$. The solution to the above, given by \eqref{eq:2.1.3}, is
\begin{equation*}
  X^{k+1}=X^{k}-e_i\frac{\left(\left(A_{:,i}\right)^{\top}X^kB-C_{:,i}\right)\left(B^{\top}B\right)^{\dag}B^{\top}}{A_{i,i}}.
\end{equation*}

When $i$ is chosen randomly, this is the randomized coordinate descent (CD-pd) method.
Applying Theorem \ref{thm:2.2} with $\rho_c=\lambda_{\min} \left(\mathbf{E}\left[ Z_2 \otimes Z_1'\right] \right) $, we see the probability distribution $p_i = \frac{A_{ii}}{Tr\left(A\right)}$ results in a convergence with
\begin{equation*}
\mathbf{E}\left[ \left\| X^k-X^* \right\|_{A}^2 \right] \leq \left( 1- \frac{\lambda_{\min}\left( A\right)}{Tr\left( A\right)} \right)^k \left\| X^0-X^* \right\|^2_A.
\end{equation*}

\subsubsection{Least-squares version}
\label{ssubsec:3.3.2}
By choosing $S=Ae_i=A_{:,i}$ as the $i$th column of $A$, $P=I$ and $G=A^{\top}A$, the resulting iterative formula \eqref{eq:2.1.3} is given by
\begin{equation}\label{eq:3.4}
  X^{k+1}=X^{k}-e_i\frac{\left(A_{:,i}\right)^{\top}\left(AX^kB-C\right)\left(B^{\top}B\right)^{\dag}B^{\top}}{\left\|A_{:,i}\right\|^2_2}.
\end{equation}
When $i$ is selected at random, this is the randomized coordinate descent (RCD) method applied to the least-squares problem $\min\limits_{X \in\mathbb{R}{^{m\times n}}}\left\| AXB-C\right\|_F^2$. A similar result was established by Kui Du et. al. \cite{DK2021}.

Applying Theorem \ref{thm:2.2}, we see that selecting $i$ with probability proportional to the magnitude of column $i$ of $A$, that is, $p_i = \frac{\left\|A_{:,i}\right\|_2^2}{\left\|A\right\|_F^2}$, results in a convergence with
\begin{equation*}
\mathbf{E}\left[ \left\| X^k-X^* \right\|_{A^{\top}A}^2 \right] \leq \left( 1- \frac{\lambda_{\min}\left( A^{\top}A\right)}{\left\| A\right\|_F^2} \right)^k \left\| X^0-X^* \right\|^2_{A^{\top}A}.
\end{equation*}

\subsection{Variants: Gaussian sampling}

In this section, we shall develop a variant of our method. When parameter matrices $S$ and $P$ are Gaussian vectors with mean $0\in \mathbb{R}^{p}, 0\in \mathbb{R}^{q}$ and positive definite covariance matrices $ \Sigma_1 \in \mathbb{R}^{p\times p}, \Sigma_2 \in \mathbb{R}^{q\times q} $, respectively. That is, $ S=\zeta \sim N\left( 0,\Sigma_1 \right), P=\eta \sim N\left( 0,\Sigma_2 \right) $. When they are applied to \eqref{eq:2.1.4}, the iterative formula becomes
\begin{equation}\label{eq:4.1}
  X^{k+1}=X^k+\frac{A^{\top}\zeta \left( \zeta^{\top}C \eta-\zeta^{\top}AX^{k}B \eta \right)\eta^{\top}B^{\top} }{\left\| \zeta^{\top}A \right\|^2_2 \left\| B\eta \right\|^2_2}.
\end{equation}
Unlike the discrete methods in Section \ref{sec:3}, to calculate an iteration of \eqref{eq:4.1} we need to compute the product of a matrix with a dense vector.  However, in our numeric tests in Section 5, the faster convergence of the Gaussian method often pays off for its high iteration cost.

Before analyzing the convergence, we introduce some lemmas.
\begin{lemma}[\cite{GRM}]\label{lem:4.1} Let $D\in \mathbb{R}^{n\times n}$ be a positive definite diagonal matrix and $U\in \mathbb{R}^{n\times n}$ be an orthogonal matrix, and $\Omega = UDU^{\top}$. If $u \sim N\left( 0,D \right)$ and $\xi \sim N\left( 0,\Omega \right)$, then
\begin{equation}\label{eq:4.2}
  \mathbf{E} \left[ \frac{\xi\xi^{\top}}{\xi^{\top} \xi} \right]=U\mathbf{E}\left[ \frac{uu^{\top}}{u^{\top}u} \right]U^{\top},
\end{equation}
and
\begin{equation}\label{eq:4.3}
  \mathbf{E} \left[ \frac{\xi\xi^{\top}}{\xi^{\top} \xi} \right] \succeq \frac{2}{\pi}\frac{\Omega}{Tr\left( \Omega \right)}.
\end{equation}
\end{lemma}

\begin{lemma}[\cite{Golub2014}]\label{lem:4.2} If $A, A+B \in \mathbf{R}^{n\times n}$ are symmetric matrices, we shall use $\lambda_k \left(A\right)$ to designate the $k$th largest eigenvalue, i.e.,
\begin{equation*}
  \lambda_n \left(A\right) \leq \cdots \leq \lambda_2 \left(A\right) \leq \lambda_1 \left(A\right),
\end{equation*}
then we have
\begin{equation}\label{eq:4.4}
  \lambda_k\left( A \right)+\lambda_n\left(B \right) \leq \lambda_k\left( A+B \right) \leq \lambda_k\left( A \right)+\lambda_1\left( B \right), \ \ k=1, 2, \cdots, n.
\end{equation}
\end{lemma}

\begin{lemma}\label{lem:4.3}
Let $A_1,\ A_2,\ B_1,\ B_2 \in \mathbf{R}^{n\times n}$ be symmetric positive semi-definite matrices. If $A_2-A_1,\ B_2-B_1$ are also symmetric positive semi-definite matrices, we have
\begin{equation}\label{eq:4.5}
  \lambda_{\min}\left( A_2^{\top} \otimes B_2 \right) \geq \lambda_{\min} \left( A_1^{\top} \otimes B_1 \right).
\end{equation}
\end{lemma}
\begin{proof}
See Appendix for more details.
\end{proof}

To analyze the complexity of the resulting method, let $\mu = A^{\top}S , \nu= BP$ which are also Gaussian, distributed as $ \mu \sim N\left( 0,\Omega_1 \right), \nu \sim N\left( 0,\Omega_2 \right)$, with $\Omega_1 = A^{\top}\Sigma_1 A, \Omega_2 = B\Sigma_2 B^{\top}$. In this section, we assume that $A$ has full column rank and $B$ has full row rank so that $\Omega_1$ and $\Omega_2$ are always positive definite.

\begin{theorem}\label{thm:4.1}
Let $\mu = A^{\top}S , \nu= BP$ distributed as $ \mu \sim N\left( 0,\ \Omega_1 \right)$ and $\nu \sim N\left( 0,\Omega_2 \right)$. Then the iterative scheme \eqref{eq:4.1} satisfies
\begin{equation*}
\mathbf{E}\left[ \left\| X^k-X^* \right\|_{F}^2 \right] \leq  \rho^{k}\left\| X^0-X^* \right\|^2_{F},
\end{equation*}
where
\begin{equation*}
\rho = 1- \frac{4}{\pi^2 Tr\left(\Omega_2 \right)Tr\left(\Omega_1 \right)}\cdot \lambda_{\min} \left( \Omega_2^{\top} \otimes  \Omega_1 \right).
\end{equation*}
\end{theorem}

\begin{proof}
The complexity of the method can be established through
\begin{equation}\label{eq:4.8}
  \begin{split}
    \rho & = 1- \lambda_{\min}\left( \mathbf{E}\left[ Z_2 \otimes Z_1 \right]  \right) \\
      & = 1- \lambda_{\min}\left( \mathbf{E}\left[ \frac{\nu \nu^{\top}}{\left\|\nu \right\|^2_2}\otimes \frac{\mu \mu^{\top}}{\left\|\mu \right\|^2_2} \right] \right) \\
      & = 1- \lambda_{\min}\left( \mathbf{E}\left[ \frac{\nu \nu^{\top}}{\left\|\nu \right\|^2_2} \right] \otimes \mathbf{E}\left[ \frac{\mu \mu^{\top}}{\left\|\mu \right\|^2_2} \right] \right)
  \end{split}
\end{equation}
Using Lemma \ref{lem:4.1}, we have
$$  \mathbf{E}\left[ \frac{\mu \mu^{\top}}{\left\|\mu\right\|^2_2} \right]  \succeq  \frac{2}{\pi} \frac{\Omega_1}{Tr\left( \Omega_1 \right)}, \ 
  \mathbf{E}\left[ \frac{\nu \nu^{\top}}{\left\|\nu \right\|^2_2} \right]  \succeq  \frac{2}{\pi} \frac{\Omega_2}{Tr\left( \Omega_2 \right)}. $$
Then, combining Lemma \ref{lem:4.3}, Equation \eqref{eq:4.8} can be written as
\begin{equation*}
  \begin{split}
    \rho & = 1- \lambda_{\min}\left( \mathbf{E}\left[ \frac{\nu \nu^{\top}}{\left\|\nu \right\|^2_2} \right] \otimes \mathbf{E}\left[ \frac{\mu \mu^{\top}}{\left\|\mu \right\|^2_2} \right] \right) \\
      & \leq 1- \lambda_{\min}\left(  \frac{2}{\pi} \frac{\Omega_2^{\top}}{Tr\left( \Omega_2 \right)} \otimes \frac{2}{\pi} \frac{\Omega_1}{Tr\left( \Omega_1 \right)} \right) \\
      & = 1- \frac{4}{\pi^2 Tr\left(\Omega_2 \right)Tr\left(\Omega_1 \right)}\cdot \lambda_{\min} \left( \Omega_2^{\top} \otimes  \Omega_1 \right).
  \end{split}
\end{equation*}
Lemma \ref{lem:4.1} depicts that $\mathbf{E}\left[ \frac{\mu \mu^{\top}}{\left\|\mu\right\|^2_2} \right],\mathbf{E}\left[ \frac{\nu \nu^{\top}}{\left\|\nu \right\|^2_2} \right]$ are positive definite, and thus the expected norm of the error of Gaussian method converges exponentially to zero.
\end{proof}

\begin{remark}\label{rem:3}
Choosing $\Sigma_1 = \Sigma_2 = I$ so that $ S=\zeta \sim N\left( 0, I \right), P=\eta \sim N\left( 0, I \right) $, then we obtain the Gaussian global random Kaczmarz (GaussGRK) method.
 Let parameter matrices $S=\zeta \sim N\left( 0,I \right)$ and $P=I \in \mathbb{R}^{q\times q}$. Then, the update \eqref{eq:4.1} becomes
\begin{equation*}
  X^{k+1}=X^k+\frac{A^{\top}\zeta \left( \zeta^{\top}C-\zeta^{\top}AX^{k}B\right)B^{\dag} }{\left\| \zeta^{\top}A \right\|^2_2 },
\end{equation*}
which is called the Gaussian randomized Kaczmarz method about matrix $A$ (GaussRK-A).
Thus at each iteration, a random normal Gaussian vector $\zeta$ is drawn and a search direction is formed by $A^{\top}\zeta$.
Similarly, let $S=I \in \mathbb{R}^{p\times p}$ and $P=\eta \sim N\left( 0,\Sigma \right)$, we have the GaussRK-B method as follows
\begin{equation*}
  X^{k+1}=X^k+\frac{A^{\dag} \left(C \eta-AX^{k}B \eta \right)\eta^{\top}B^{\top} }{\left\| B\eta \right\|^2_2}.
\end{equation*}
\end{remark}

%
%

\section{Numerical results}
\label{sec:6}
In this section, we present several numerical examples to illustrate the performance of the iteration methods proposed in this paper for solving the matrix equation \eqref{eq:1.1}.
All experiments are carried out using MATLAB (version R2020a) on a personal computer with a 2.50 GHz central processing unit (Intel(R) Core(TM) i5-7200U CPU), 4.00 GB memory, and Windows operating system (64 bit Windows 10).

 To construct a matrix equation, we set $C = AX^*B$, where $X^*=ones(m,n)$ is the exact solution of this matrix equation. All computations are started from the initial guess $X_0 =O$, and terminated once the relative error (RE) of the solution, defined by
 \begin{equation*}
   RE=\frac{\left\|X^k-X^* \right\|_F^2}{\left\|X^*\right\|_F^2},
 \end{equation*}
 at the current iterate $X^k$, satisfies $RE < 10^{-6}$ or exceeded maximum iteration. IT and CPU denote the average number of iteration steps and the average CPU times (in seconds) for 10 times repeated runs, respectively. The item '$-$' represents that the number of iteration steps exceeds the maximum iteration (100000) or the CPU time exceeds 120s.
 
 We consider the following methods and variants:
 \begin{itemize}
   \item GRK with complete discrete sampling and as in Section \ref{subsec:3.1}.
   \item RCD with complete discrete sampling and as in Section \ref{ssubsec:3.3.2}.
   \item RK-A: \eqref{eq:2.1.4} with $S=e_i$ and $P=I$ as in Remark \ref{rem:1}.
   \item GaussGRK with Gaussian sampling and as in Section \ref{sec:4}.
   \item GaussRK-A: \eqref{eq:2.1.4} with $S=\zeta \sim N\left( 0,\Sigma \right)$ and $P=I$ as in \eqref{rem:3}.
 \end{itemize}

For the block methods GRBK\cite{NYQ2022}, we assume that $[p] = \left\{I_1, \cdots , I_s \right\}$ and $[q] = \left\{J_1, \cdots , J_t\right\}$, respectively, are partitions of $[p]$ and $[q]$, and the block samplings have the same size $\left|I_k\right| = \tau_1$ and $\left|J_k\right| = \tau_2$, where
\begin{equation*}
  I_i= \begin{cases}\left\{(i-1) \tau_1+1,(i-1) \tau_1+2, \cdots, i \tau_1\right\}, & i=1,2, \cdots, s-1, \\ \left\{(s-1) \tau_1+1,(s-1) \tau_1+2, \cdots, p\right\}, & i=s,\end{cases}
\end{equation*}
and
\begin{equation*}
  J_i= \begin{cases}\left\{(j-1) \tau_2+1,(j-1) \tau_2+2, \cdots, j \tau_2\right\}, & j=1,2, \cdots, t-1, \\ \left\{(t-1) \tau_2+1,(t-1) \tau_2+2, \cdots, q\right\}, & j=t.\end{cases}
\end{equation*}
We divided our tests into three categories: synthetic dense data, real-world sparse data, and CT Data.

\begin{example}\label {exp6.1}
Synthetic dense data. Random matrices for this test are generated as follows:
\begin{itemize}
  \item Type I \cite{NYQ2022}: For given $p, m$, and $r_1 = rank(A)$, we construct a matrix $A$ by $ A=U_1D_1V_1^{\top},$ where $U_1 \in \mathbb{R}^{p\times r_1}$ and $V_1 \in \mathbb{R}^{m\times r_1}$ are orthogonal columns matrices. The entries of $U_1$ and $V_1$ are generated from a standard normal distribution, and then columns are orthogonalization, i.e., $\left[U_1, \sim \right] = qr\left(randn(p,r_1),0 \right),\ \ \left[V_1, \sim \right] = qr\left(randn(m,r_1),0 \right).$ The matrix $D_1$ is an $r_1 \times r_1$ diagonal matrix whose diagonal entries are uniformly distribution numbers in $(1, 2)$, i.e., $ D_1=diag\left(1+rand(r_1,1)\right).$ Similarly, for given $n, q$, and $r_2 = rank(B)$, we construct a matrix $B$ by $ B=U_2D_2V_2^{\top},$ where $U_2 \in \mathbb{R}^{n\times r_2}$ and $V_1 \in \mathbb{R}^{q\times r_2}$ are orthogonal columns matrices, and the matrix $D_2$ is an $r_2 \times r_2$ diagonal matrix.
  \item Type II: For given $p, m, n, q$, the entries of $A$ and $B $ are generated from  standard normal distributions, i.e., $ A=randn(p, m),\ \ B=randn(n, q).$
\end{itemize}
\end{example} 
In Tables \ref{table:1} and \ref{table:2}, we report the average IT and CPU of GRK, GaussGRK, GRBK, RCD, RK-A, and GaussRK-A for solving matrix equations with Types I and II, where $A$ is full column rank, i.e., $r_1=m$ and $B$ full row rank i.e., $r_2=n$ in Type I. For the GRBK method, we use different block sizes in Table \ref{table:1} while fixed block sizes $ \tau_1=\tau_2=10$ in Table \ref{table:2}. 

From these two tables, we can see that the GaussGRK method is better than the GRK method in terms of IT and CPU time. The IT and CPU time of both the GRK and GaussGRK methods increase with the increase of matrix dimensions. However, the GaussGRK method has a small increase in terms of CPU time. In Fig. \ref{figure:1}, we plot the relative errors of GRK and GaussGRK for two matrix equations with Type I ($A = U_1D_1V_1^{\top}$ with $m = 50, p = 20, r_1 = 20$ and $B = U_2D_2V_2^{\top}$ with $n = 50, q =20, r_2 = 20$) and Type II ($A = randn(50, 20)$ and $B = randn(20, 50)$). From Fig. \ref{figure:1}, we can more intuitively see that the GaussGRK method is better than the GRK method in terms of IT and CPU time. However, as the matrix size continues to increase, the GaussGRK method and the GRK method require significant computational costs, so these two methods will not be considered in future experiments.

\begin{figure}[htbp]
  \centering
  \subfigure[Type I]
  {
   \includegraphics[width=3.8cm]{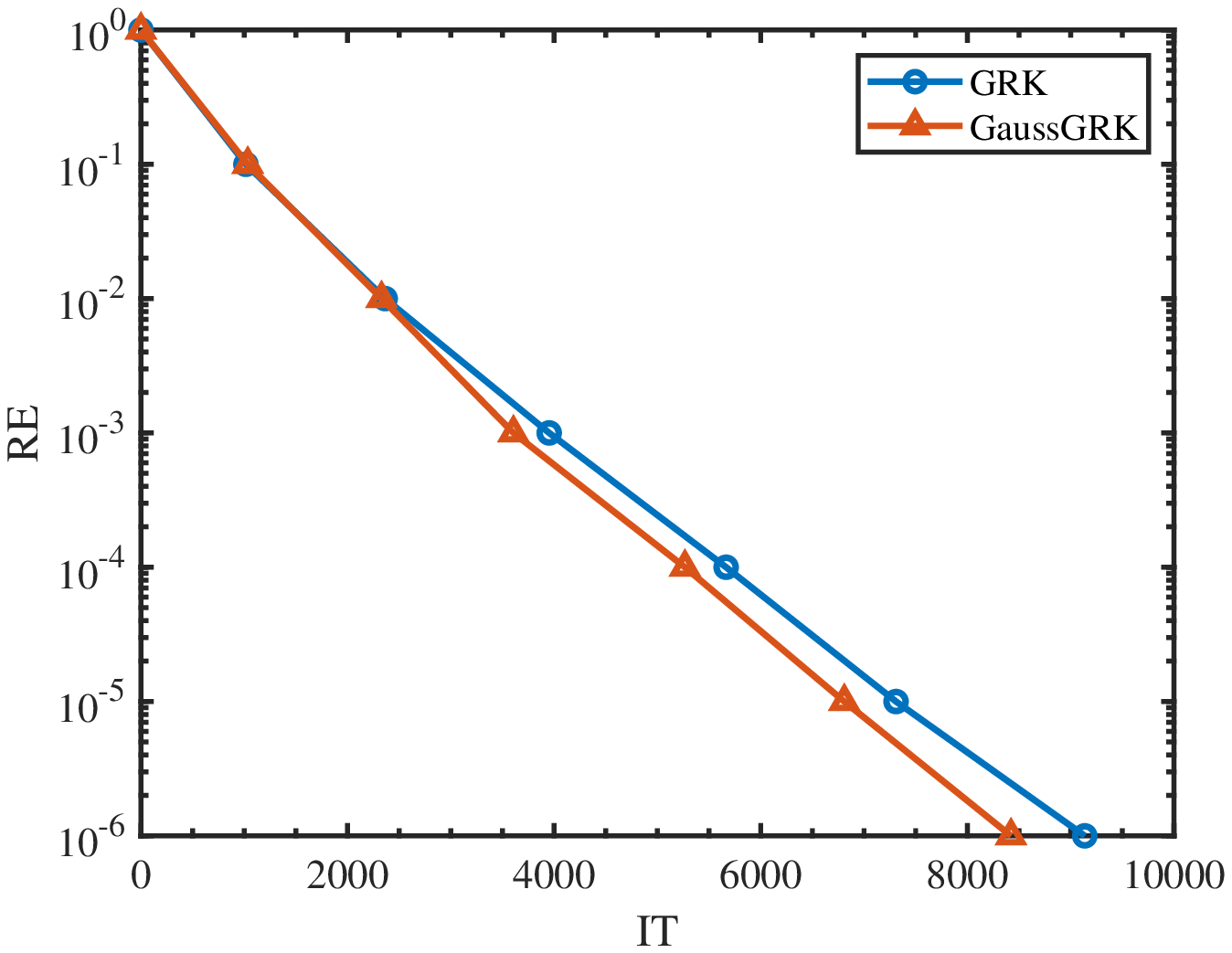}
   \includegraphics[width=3.8cm]{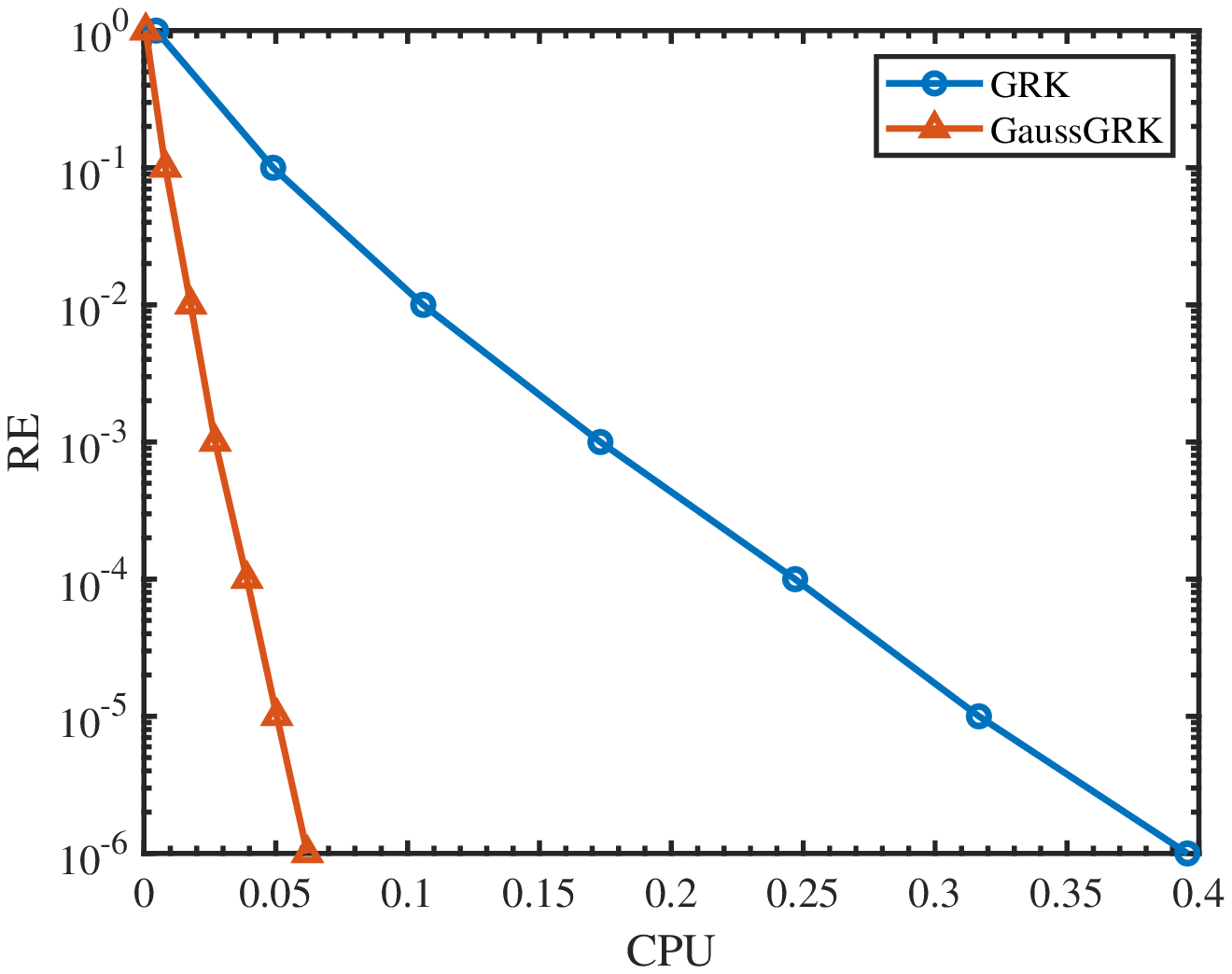}
  }
  \subfigure[Type II]
  {
   \includegraphics[width=3.8cm]{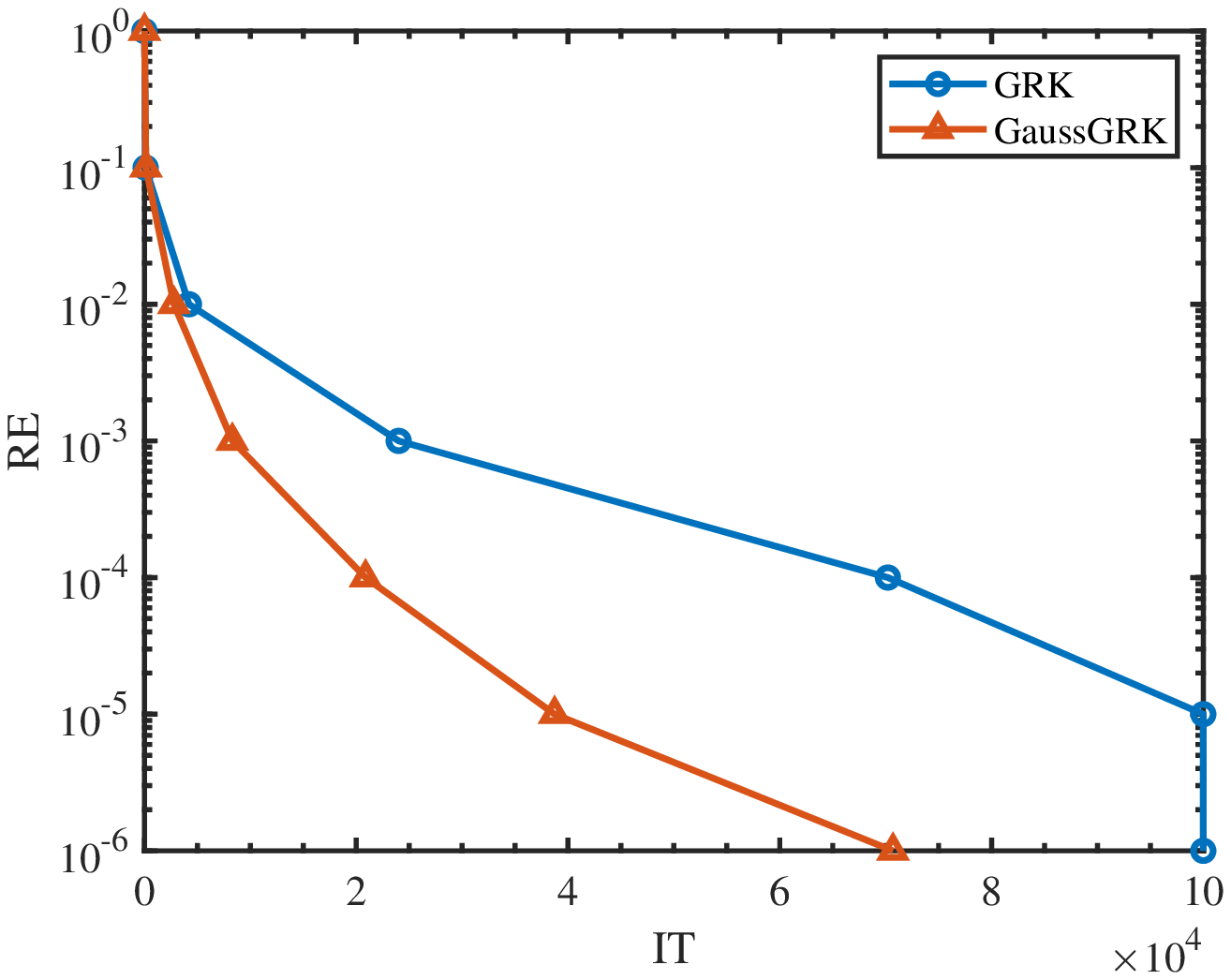}
   \includegraphics[width=3.8cm]{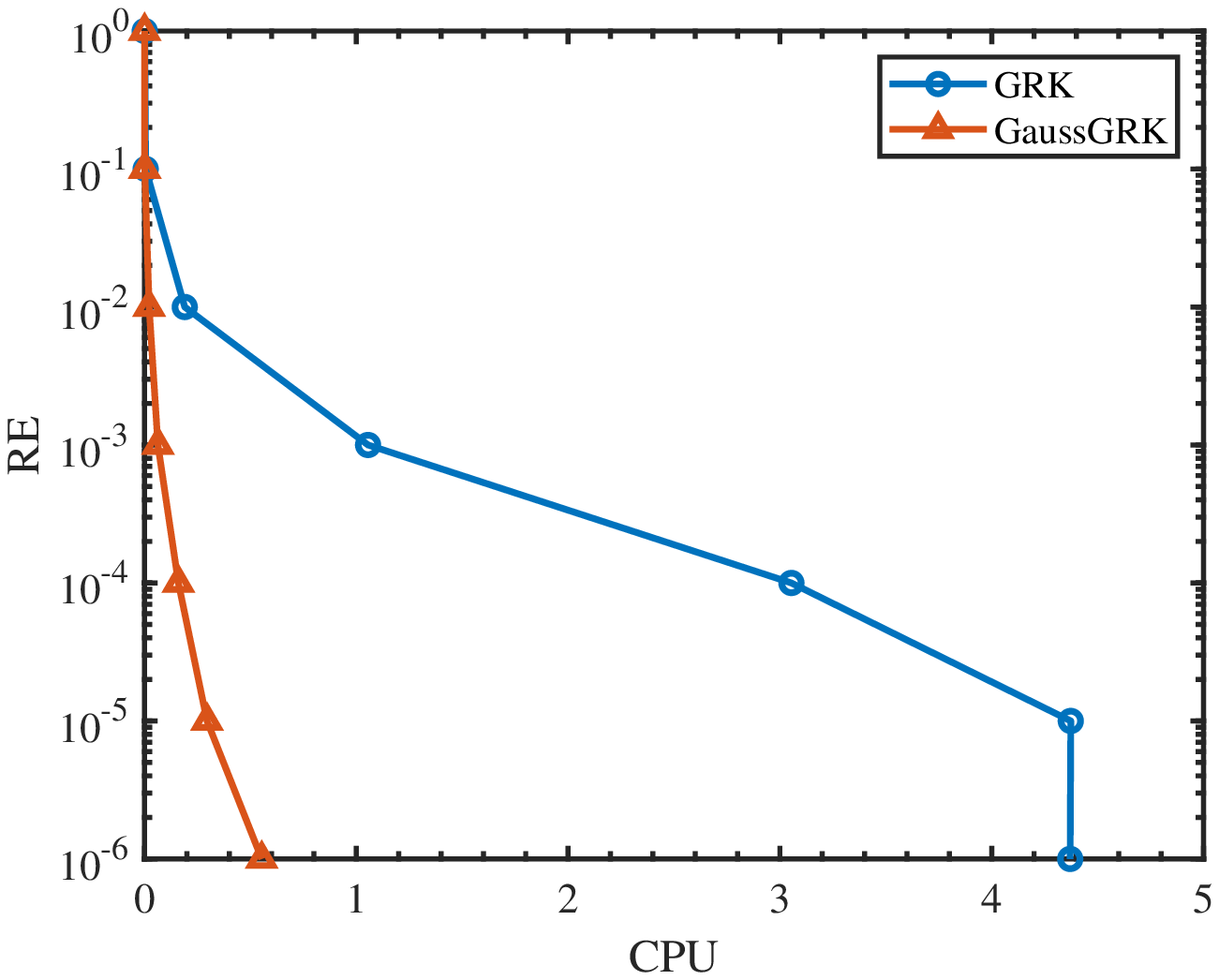}
  } \\
  \caption{The relative errors of GRK and GaussGRK for Types I and II}\label{figure:1}
\end{figure}
\begin{table}[htbp]
	\footnotesize
	\resizebox{\textwidth}{!}{
		\begin{threeparttable}
			\caption{The average IT and CPU of GRK, GaussGRK, GRBK, RCD, RK-A, and GaussRK-A with Type I.}
			\label{table:1}
			\begin{tabular}{ccccccccccccccc}
				\toprule  
				m & p & $\tau_1$ & n & q & $\tau_2$ &    & GRK  & GaussGRK  & GRBK  & RCD  & RK-A  & GaussRK-A  \\
				\hline
				\multirow{2}*{50} & \multirow{2}*{20} & \multirow{2}*{10} & \multirow{2}*{20} & \multirow{2}*{50} &\multirow{2}*{10}
				& IT      & 8596    & 8013    & \textbf{112}     & 267     & 287     & 291      \\
				&  &  &  &  &    & CPU     & 0.3760  & 0.0609  & 0.0175  & 0.0122  & 0.0089  & \textbf{0.003}   \\
				\hline
				\multirow{2}*{100} & \multirow{2}*{40} & \multirow{2}*{20} & \multirow{2}*{40} & \multirow{2}*{100} &\multirow{2}*{20}
				& IT      & 38985   & 37349   & \textbf{94}      & 570     & 663     & 643      \\
				&  &  &  &  &    & CPU     & 1.8965  & 0.5035  & 0.0294  & 0.0484  & 0.0335  & \textbf{0.0172}    \\
				\hline
				\multirow{2}*{100} & \multirow{2}*{40} & \multirow{2}*{20} & \multirow{2}*{100} & \multirow{2}*{500} &\multirow{2}*{100}
				& IT      & $-$    & $-$    & \textbf{26}      & 606     & 667     & 648    \\
				&  &  &  &  &    & CPU     & $-$    & $-$    & \textbf{0.0573}  & 0.2978  & 0.1685  & 0.1465    \\
				\hline
				\multirow{2}*{500} & \multirow{2}*{100} & \multirow{2}*{50} & \multirow{2}*{100} & \multirow{2}*{500} &\multirow{2}*{50}
				& IT      & $-$    & $-$    & \textbf{78}      & 1513    & 1561    & 1650    \\
				&  &  &  &  &    & CPU     & $-$    & $-$    & \textbf{0.1442}  & 4.7249  & 0.6455  & 0.9307    \\
				\hline
				\multirow{2}*{1000} & \multirow{2}*{200} & \multirow{2}*{200} & \multirow{2}*{100} & \multirow{2}*{500} &\multirow{2}*{50}
				& IT      & $-$   & $-$    & \textbf{24}     & 3122     & 3242    & 3322    \\
				&  &  &  &  &    & CPU     & $-$   & $-$    & \textbf{0.2241} & 19.1592  & 2.5464  & 4.2424    \\
				\bottomrule 
			\end{tabular}
		\end{threeparttable}
	}
\end{table}

From Tables \ref{table:1} and \ref{table:2}, we observe that the GRBK method vastly outperforms the RCD, RK-A, and GaussRK-A methods in terms of IT and CPU time, because the GRBK method selects multiple rows and columns in each iteration. Among these methods which select a single row or column for calculation in each iteration, the RK-A and GaussRK-A methods perform slightly better than than the RCD method. In detail, from Table \ref{table:1}, we observe that the RCD method requires fewer iteration steps, the GaussRK-A method takes less CPU when the matrix size is small and the RK-A method is more challenging when the matrix size is large. From Table \ref{table:2}, we find that the GaussRK-A method is competitive in terms of IT and CPU time.

\begin{table}[htbp]
    \footnotesize
    \resizebox{\textwidth}{!}{
    \begin{threeparttable}
		\caption{The average IT and CPU of GRK, GaussGRK, GRBK, RCD, RK-A, and GaussRK-A with Type II.}
		\label{table:2}
		\begin{tabular}{ccccccccccccc}
			\toprule  
			m & p & n & q &               & GRK    & GaussGRK & GRBK   & RCD     & RK-A   & GaussRK-A  \\
			\hline
			\multirow{2}*{30} & \multirow{2}*{10} & \multirow{2}*{10} & \multirow{2}*{30}
                                & IT      & 50057  & 11656   & \textbf{1}       & 770     & 530     & 256      \\
			         &  &  &    & CPU     & 2.0871 & 0.0708  & \textbf{0.0003}  & 0.0213  & 0.0139  & 0.0017    \\	
			\hline
			\multirow{2}*{50} & \multirow{2}*{20} & \multirow{2}*{20} & \multirow{2}*{50}
                                & IT      & $-$    & 74033   & \textbf{88}      & 2694    & 1541    & 825      \\
			         &  &  &    & CPU     & $-$    & 0.6492  & 0.0157  & 0.1233  & 0.0475  & \textbf{0.0075}    \\
			\hline
			\multirow{2}*{100} & \multirow{2}*{40} & \multirow{2}*{40} & \multirow{2}*{100}
                                & IT      & $-$    & $-$    & \textbf{588}     & 7442    & 4092    & 2160      \\
			         &  &  &    & CPU     & $-$    & $-$    & 0.0974  & 0.6245  & 0.2027  & \textbf{0.0556}    \\
			\hline
			\multirow{2}*{100} & \multirow{2}*{40} & \multirow{2}*{100} & \multirow{2}*{500}
                                & IT      & $-$    & $-$    & \textbf{966}     & 5494    & 3082    & 1697    \\
			         &  &  &    & CPU     & $-$    & $-$    & \textbf{0.2509}  & 2.3004  & 0.5499  & 0.2739    \\
			\hline
			\multirow{2}*{500} & \multirow{2}*{100} & \multirow{2}*{100} & \multirow{2}*{500}
                                & IT      & $-$    & $-$    & \textbf{1775}    & 8749    & 5259    & 2977    \\
			         &  &  &    & CPU     & $-$    & $-$    & \textbf{0.5908}  & 27.2126 & 2.0268  & 1.6648    \\
			\hline
			\multirow{2}*{1000} & \multirow{2}*{200} & \multirow{2}*{100} & \multirow{2}*{500}
                                & IT      & $-$   & $-$    & \textbf{3486}     & 18890     & 10251   & 5799    \\
			         &  &  &    & CPU     & $-$   & $-$    & \textbf{1.4599}   & 116.6466  & 8.3202  & 7.4095    \\
			\hline
			\multirow{2}*{1000} & \multirow{2}*{200} & \multirow{2}*{200} & \multirow{2}*{1000}
                                & IT      & $-$   & $-$    & \textbf{7250}     & $-$   & 9998    & 5896    \\
			         &  &  &    & CPU     & $-$   & $-$    & \textbf{4.3672}   & $-$   & 31.4897 & 23.0945    \\
           \bottomrule 
		\end{tabular}
	\end{threeparttable}
    }
\end{table}

To accelerate the convergence speed,  multiple columns (or rows) can be selected in the RCD and RK-A methods instead of a single column (or row) in each iteration, i.e., block versions. We will not experiment and show them here. However, the block variant of the GaussRK-A method is a dense matrix computation, so the computational complexity increases significantly from vector-matrix product to matrix-matrix product.

In Fig.\ref{figure:2}, we plot the relative errors of the GRBK method with different block sizes $\tau_1 = \tau_2 = \tau$ for Type I ($A = U_1D_1V_1^{\top}$ with $m = 100, p = 50, r_1 = 50$ and $B = U_2D_2V_2^{\top}$ with $n = 100, q = 50, r_2 = 50$). From Fig. \ref{figure:2}(a), we observe that increasing block sizes leads to a better convergence rate of the GRBK method. From Fig.\ref{figure:2}(b), we can find that as the block sizes $\tau$ increase, the IT and CPU time first decreases, and then increases after reaching the minimum.  From Fig.\ref{figure:2}(c), it is easy to see that when $\tau=14,15,16$, the IT and CPU time reach the minimum. The GRK method is the GRBK method with the sizes of block index sets  $\left|I_k \right|=\left|J_k\right|=1$. This also verifies that the GRK method is computationally expensive.

\begin{figure}[htbp]
	\centering
	\subfigure[]
	{
		\includegraphics[width=3.8cm]{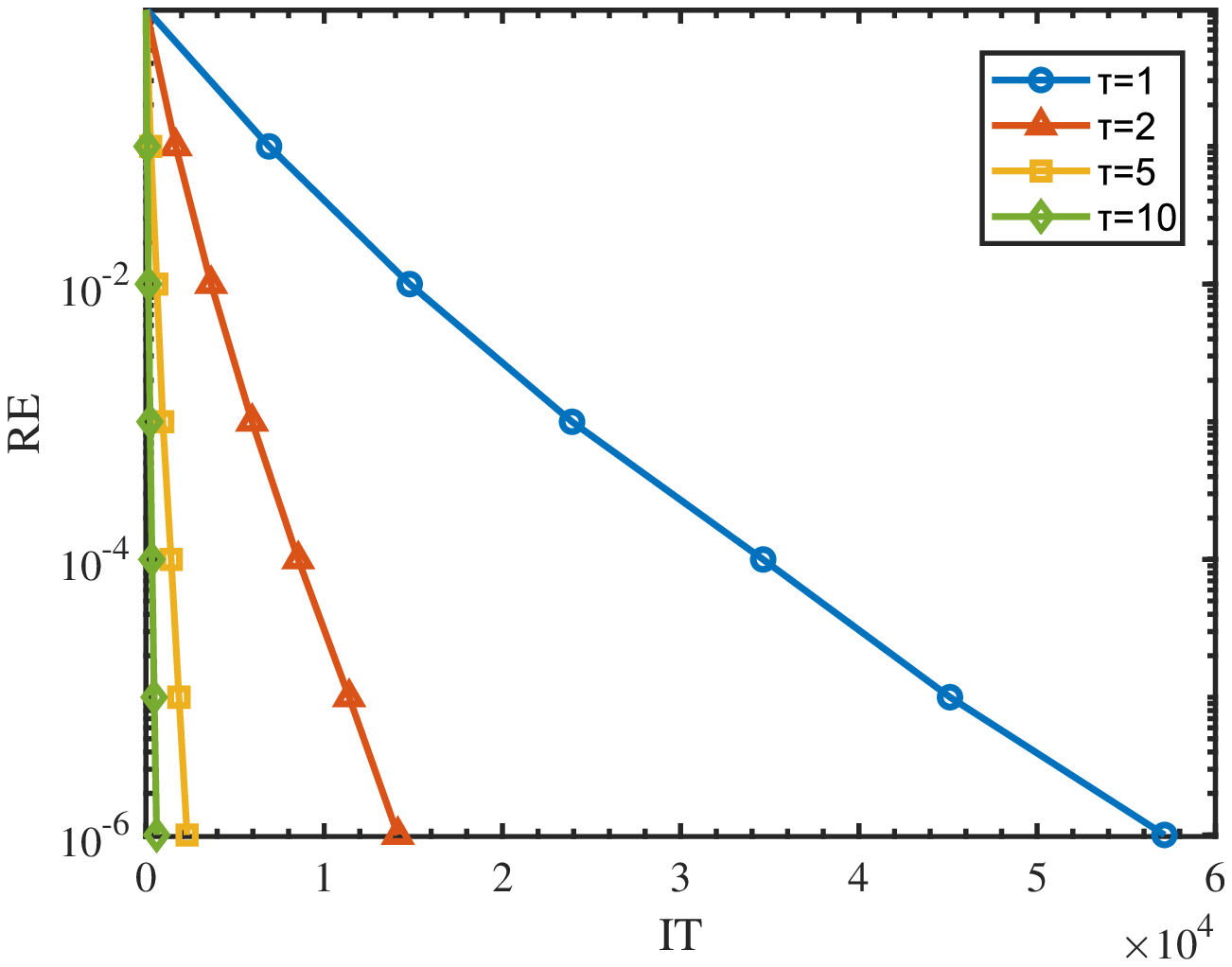}
		\includegraphics[width=3.8cm]{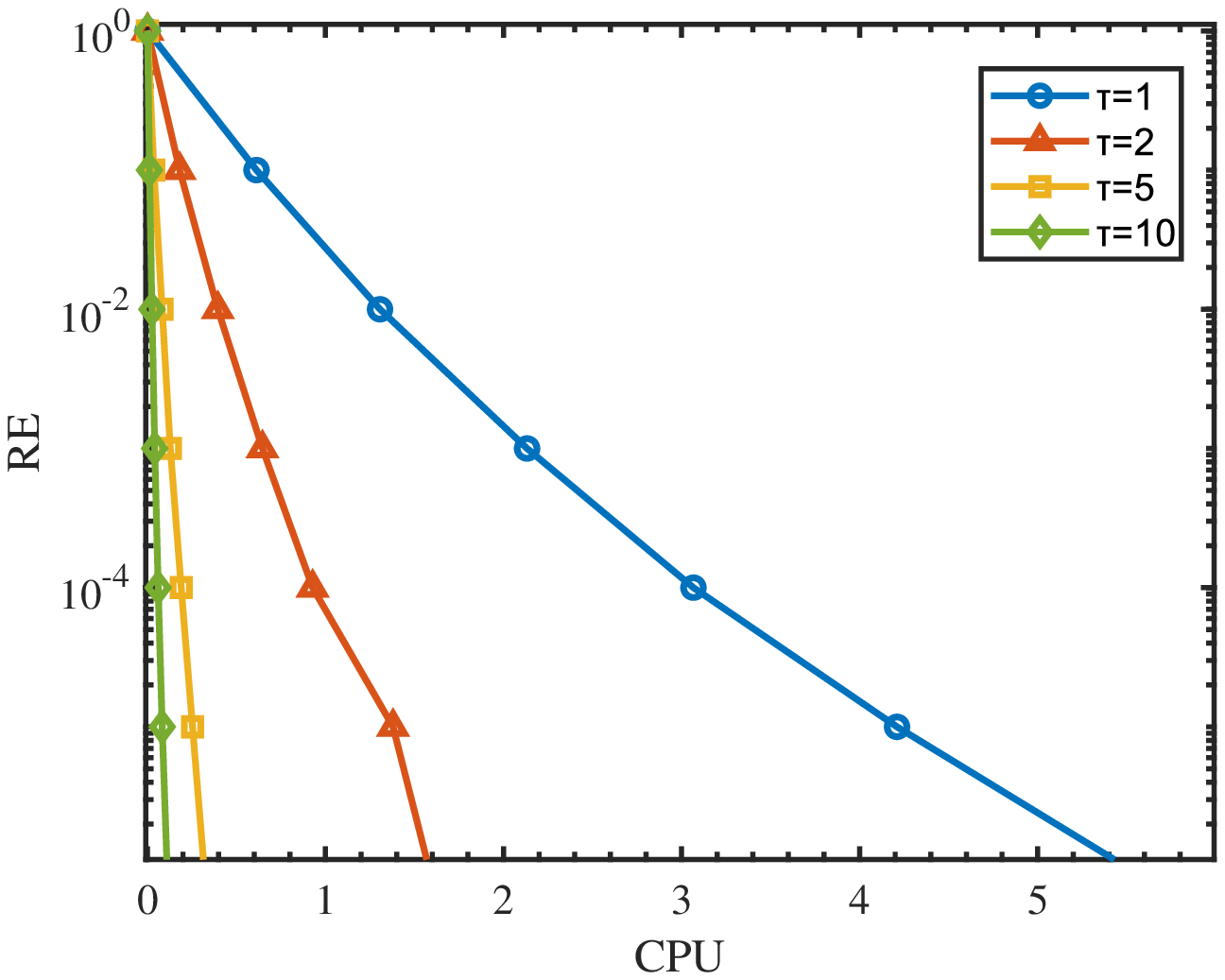}
	} \\
	\subfigure[]
	{
		\includegraphics[width=3.8cm]{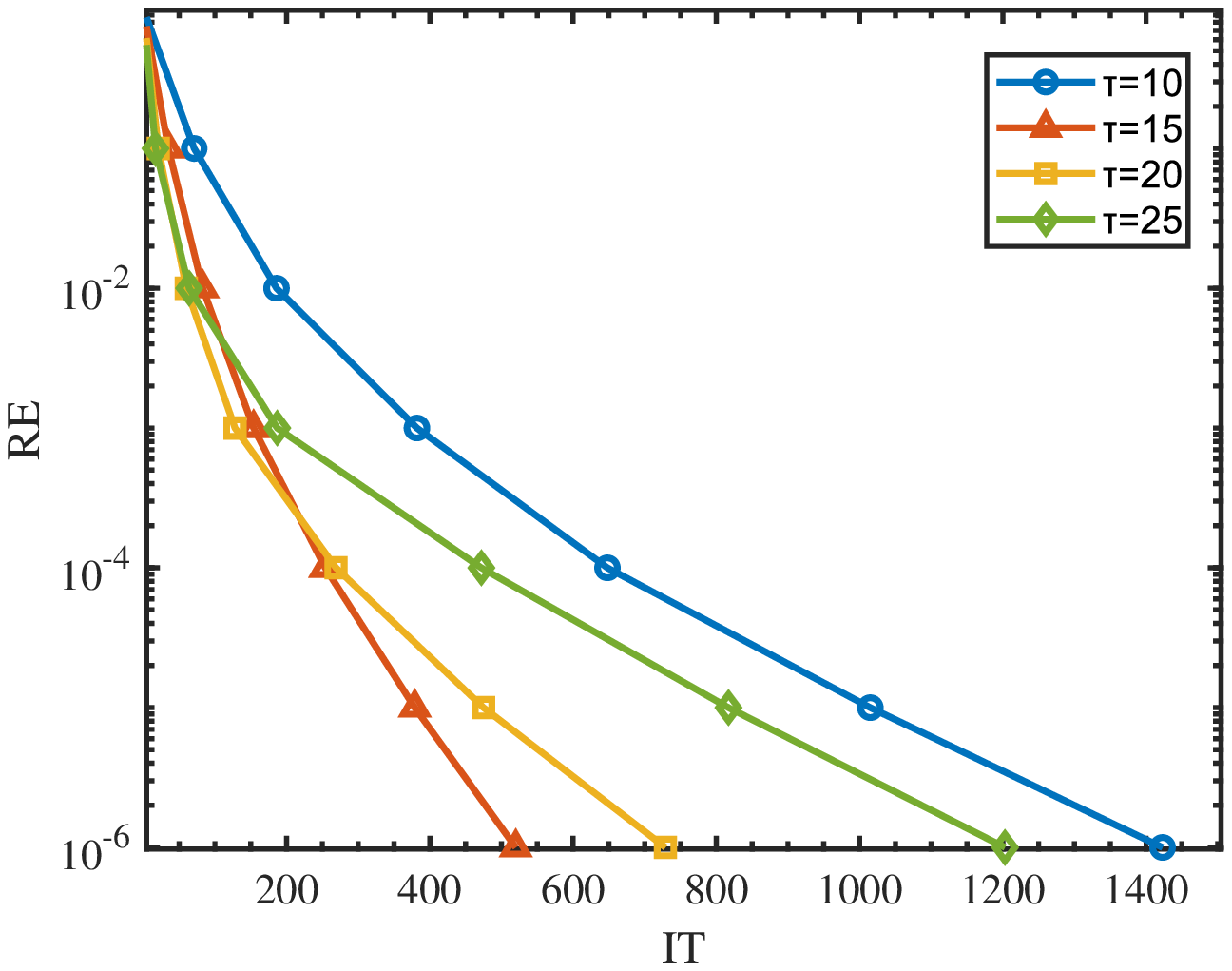}
		\includegraphics[width=3.8cm]{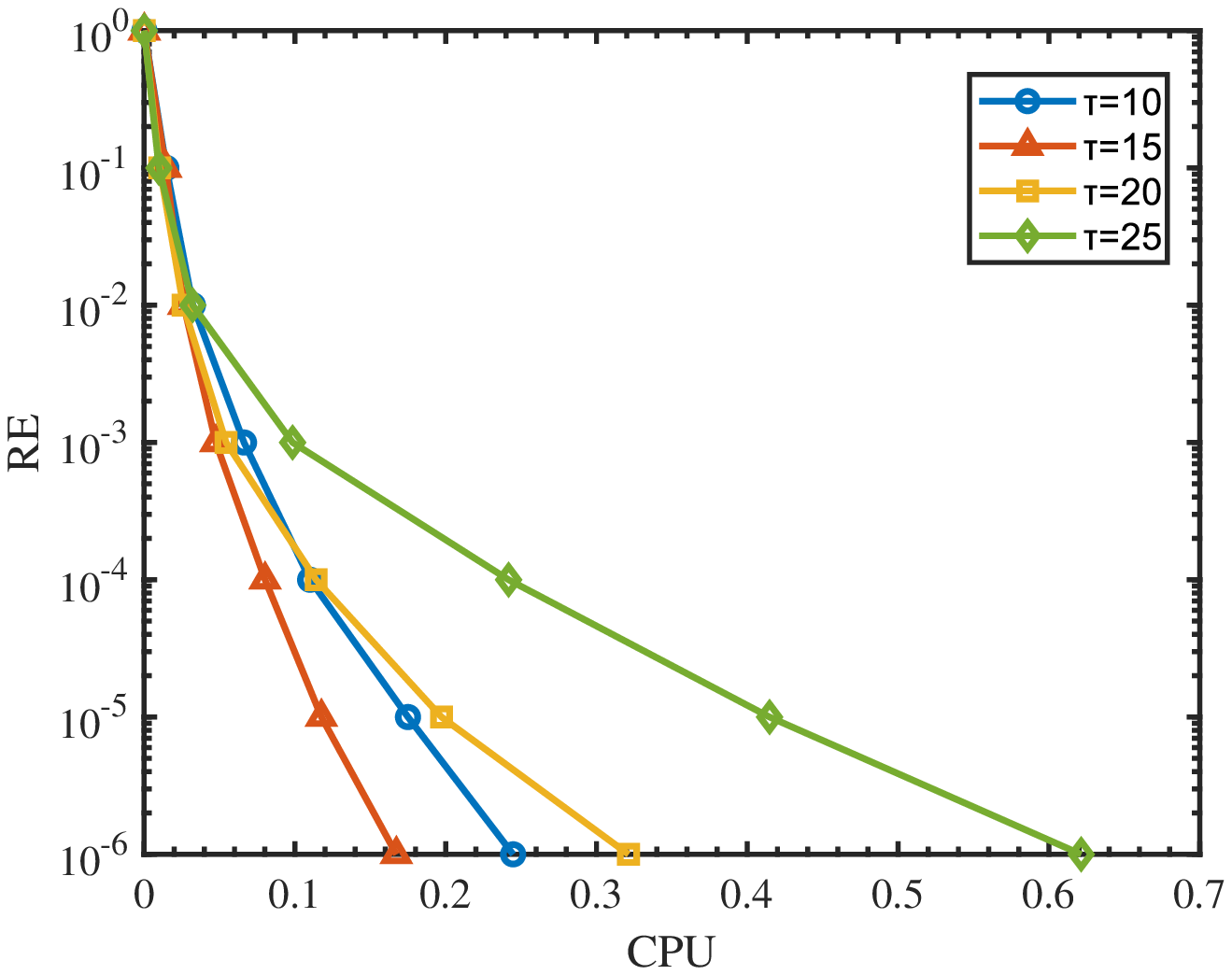}
	}
	\subfigure[]
	{
		\includegraphics[width=3.8cm]{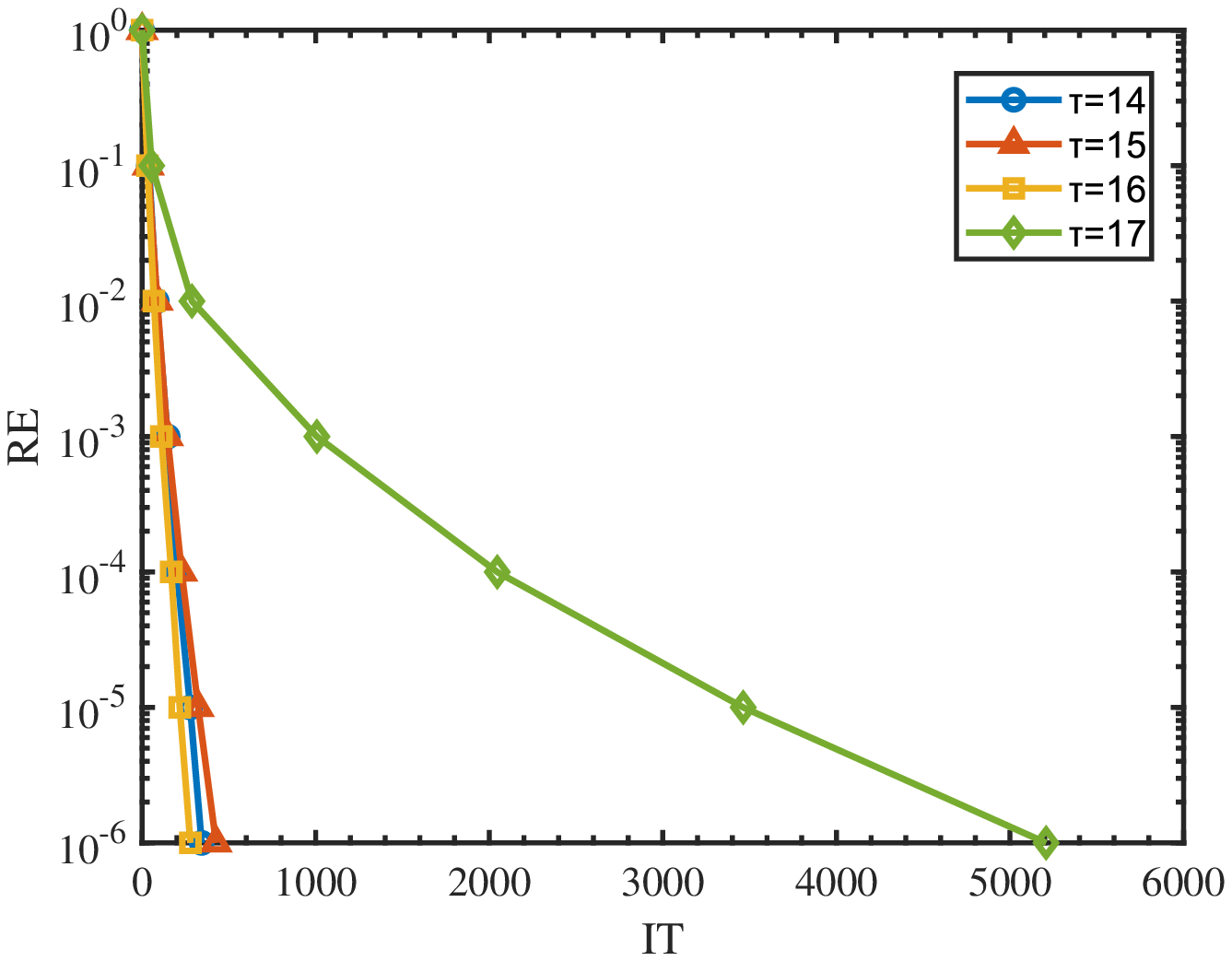}
		\includegraphics[width=3.8cm]{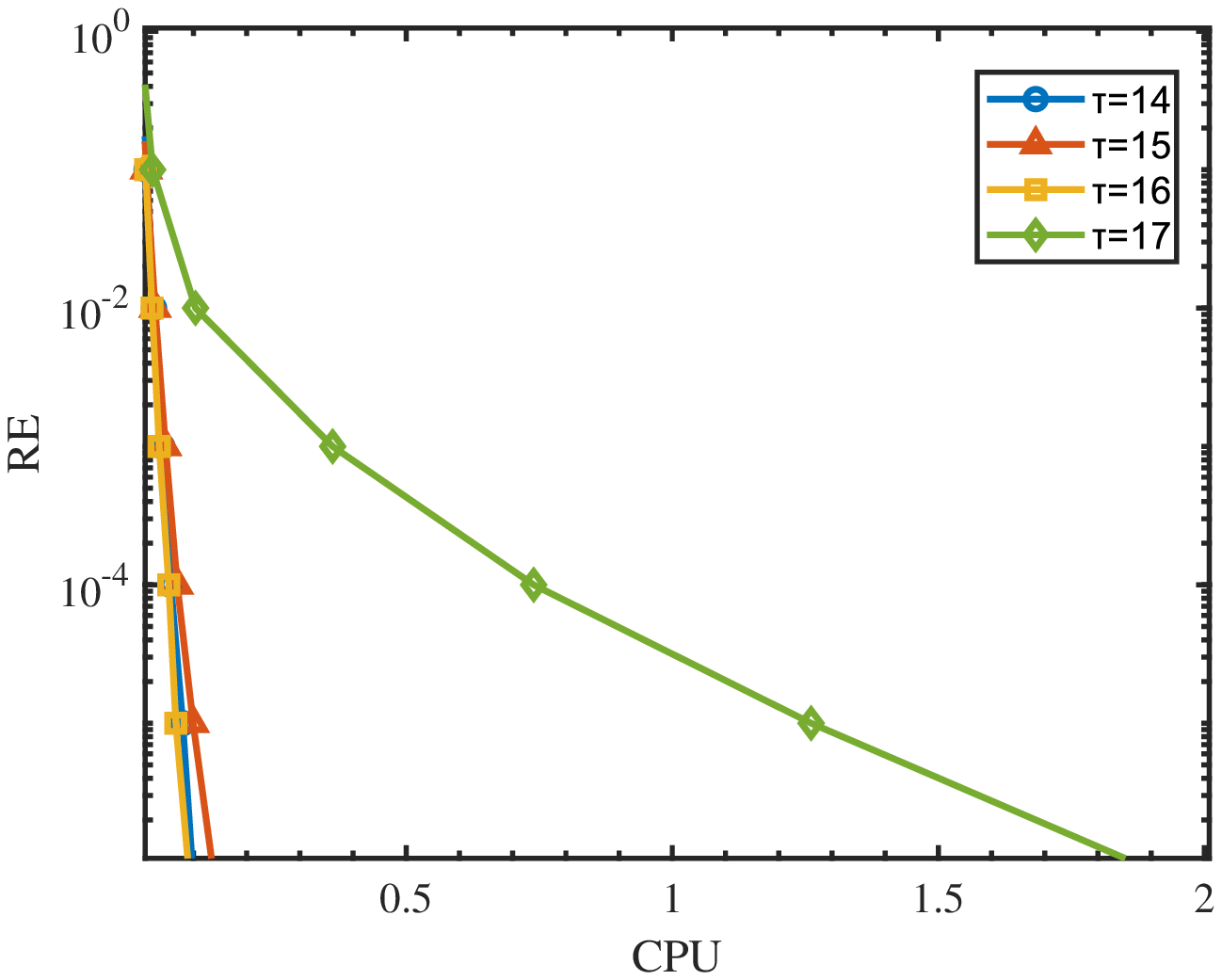}
	} \\
	\caption{Relative errors of GRBK with different block sizes $\tau_1 = \tau_2 = \tau$ for Type I.}\label{figure:2}
\end{figure}

Finally, we also compare them with the RBCD\cite{DK2021} method. To give an intuitive demonstration of the advantage, we define the speed-up as follows:
\begin{equation*}
  speed-up = \frac{CPU\ of\ RBCD}{CPU\ of\ NEW\ METHOD}.
\end{equation*}
In Fig.\ref{figure:4}, we plot the relative errors of RBCD, RCD, RK-A, GaussRK-A, and GRBK for matrix equation with Type II ($A = randn(100, 20)$, $B = randn(20, 100)$). For the GRBK method, we use the almost optimal block sizes $\tau_1=\tau_2=15$. We can see that the RCD, RK-A, GaussRK-A, and GRBK methods are better than the RBCD method in terms of IT and CPU time.
From Table \ref{table:5}, we see that the IT and CPU of the RCD, RK-A, GaussRK-A, and GRBK methods are smaller than the RBCD methods in terms of both iteration counts and CPU times with significant speed-ups.

\begin{table}[htbp]
    \footnotesize
    \resizebox{\textwidth}{!}{
    \begin{threeparttable}
		\caption{The average IT, CPU and speed-up of RBCD, RCD, RK-A, GaussRK-A, and GRBK with Type II.}
		\label{table:5}
		\begin{tabular}{ccccccccccccc}
			\toprule  
			   & $p\times m$  & $50\times 20$   & $100\times 20$    & $200\times 20$     & $500\times 20$    & $1000\times 20$ \\
			\hline
			\multirow{2}*{RBCD}
               & IT           & 4204          & 1025          & 376            & 302            & 2.0041  \\
			   & CPU          & 0.1208        & 0.0454        & 0.0342         & 0.4958         & 0.0655  \\	
			\hline
			\multirow{3}*{RCD}
               & IT           & 801              & 353               & 255          & 182          & 169 \\
			   & CPU          & 0.0338           & 0.0226            & 0.0332       & 0.3916       & 1.9815     \\
			   & speed-up     & 3.5740           & 2.0088            & 1.0301       & 1.2661       & 1.0114      \\
			\hline
			\multirow{3}*{RK-A}
               & IT           & 640             & 375             & 280           & 266            & 269  \\
			   & CPU          & 0.0181          & 0.0125          & 0.0181        & 0.0622         & 0.7622     \\
			   & speed-up     & 6.674           & 3.6320          & 1.8895        & 7.9711         & 2.6294     \\
			\hline
			\multirow{3}*{GaussRK-A}
               & IT           & 590              & 361             & 284           & 258            & 256  \\
			   & CPU          & \textbf{0.0053}           & \textbf{0.0053}          & 0.0117        & 0.0646         & 0.9349     \\
			   & speed-up     & \textbf{22.7925}          & \textbf{8.5660 }         & 2.9231        & 7.6749         & 2.1437    \\
			\hline
			\multirow{3}*{GRBK}
               & IT           & 70              & 31                & 23            & 19            & 19  \\
			   & CPU          & 0.0135          & 0.0060            & \textbf{0.0046}        & \textbf{0.0041}        & \textbf{0.0061}      \\
			   & speed-up     & 8.9481          & 7.5667            & \textbf{7.4348}        & \textbf{120.9268}      & \textbf{328.541}      \\
           \bottomrule 
		\end{tabular}
	\end{threeparttable}
    }
\end{table}
\begin{figure}[htbp]
	\centering
	\subfigure{\includegraphics[width=6.5cm]{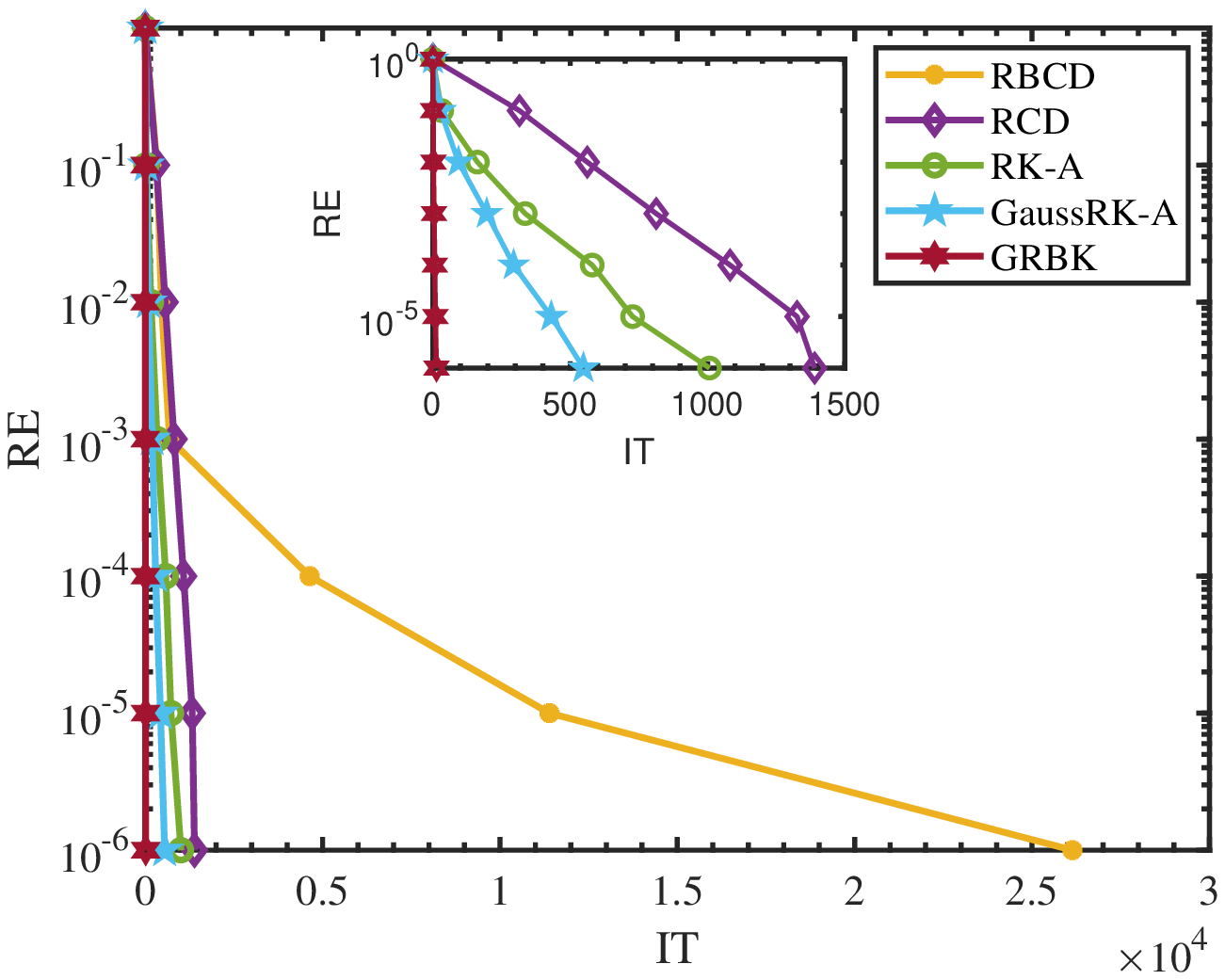}}
	\subfigure{\includegraphics[width=6.5cm]{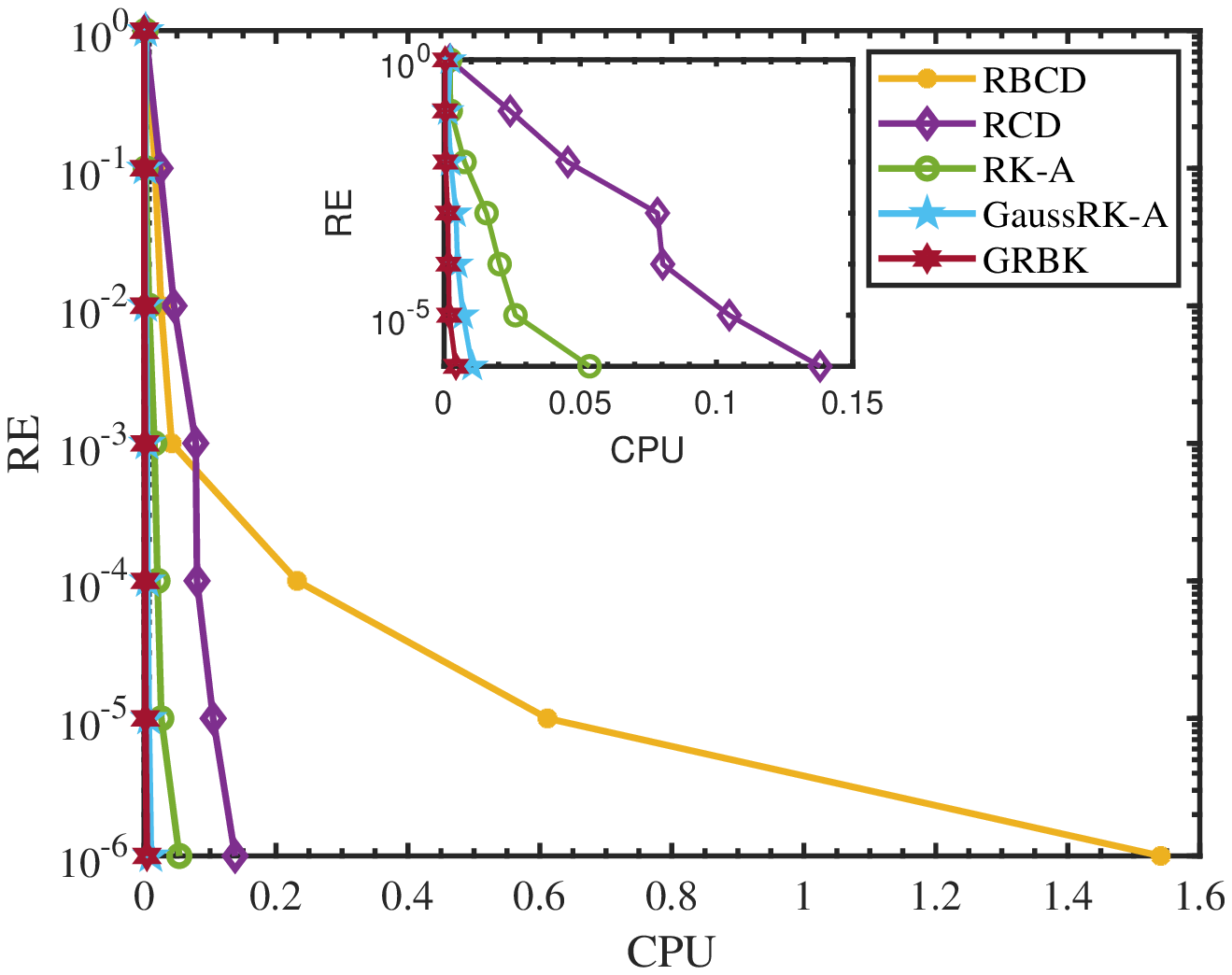}}
	\caption{Relative errors of RBCD, RCD, RK-A, GaussRK-A, and GRBK with Type II.} \label{figure:4}
\end{figure}
\begin{example} \label{exp6.2} Real-world sparse data.
 The entries of $A$ and $B $ are selected from the real-world sparse data \cite{AH2011}.
 \end{example} 
 Table \ref{table:4} lists the features of these sparse matrices, in which rank(A) denote the rank of the matrix A, respectively, and the density is defined as
\begin{equation*}
  density = \frac{\text{the number of non-zero elements of an $m$-by-$n$ matrix}}{mn},
\end{equation*}
which indicates the sparsity of the corresponding matrix.
 \begin{table}[htbp]
    \footnotesize\begin{center}
    \begin{threeparttable}
		\caption{The detailed features of sparse matrices from \cite{AH2011}.}
		\label{table:4}
		\begin{tabular}{ccccccccccccccc}
			\toprule  
			 name && size  && rank  && density     \\
			\hline
			ash219 && 219 $\times$ 85 && 85 &&  2.3529\%   \\
			\hline
			ash958 && 958 $\times$ 292 && 292 && 0.68493\%    \\
			\hline
			divorce && 50 $\times$ 9 && 9 && 50\%    \\
			\hline
			Worldcities && 315 $\times$ 100 && 100 && 53.625\%    \\
           \bottomrule 
		\end{tabular}
	\end{threeparttable}
\end{center}
\end{table}
Numerical results are shown in Fig.\ref{figure:3} and Table \ref{table:3}.
In Fig.\ref{figure:3}, we plot the relative errors of GRBK, RCD, RK-A, and GaussRK-A for the real-world matrix equations. For the GRBK method, we use the block sizes $\tau_1=\tau_2=15$. In Table \ref{table:3}, we report the average IT and CPU of GRK, GaussGRK, GRBK, RCD, RK-A, and GaussRK-A for solving real-world matrix equations. From them, we observe again that the curves of the GRBK methods are decreasing much more quickly than those of the RCD, RK-A, and GaussRK-A methods with respect to the increase of the iteration steps and CPU times. However, as the matrix size increases, the CPU of the GRBK method grows because it takes some time to compute the pseudoinverse. At this time, the RK-A and GaussGRK-A methods are more prominent in terms of IT and CPU times.

\begin{figure}[htbp]
	\centering
	\subfigure[$A=ash219, B=divorce^{\top}$]
	{
		\includegraphics[width=3.8cm]{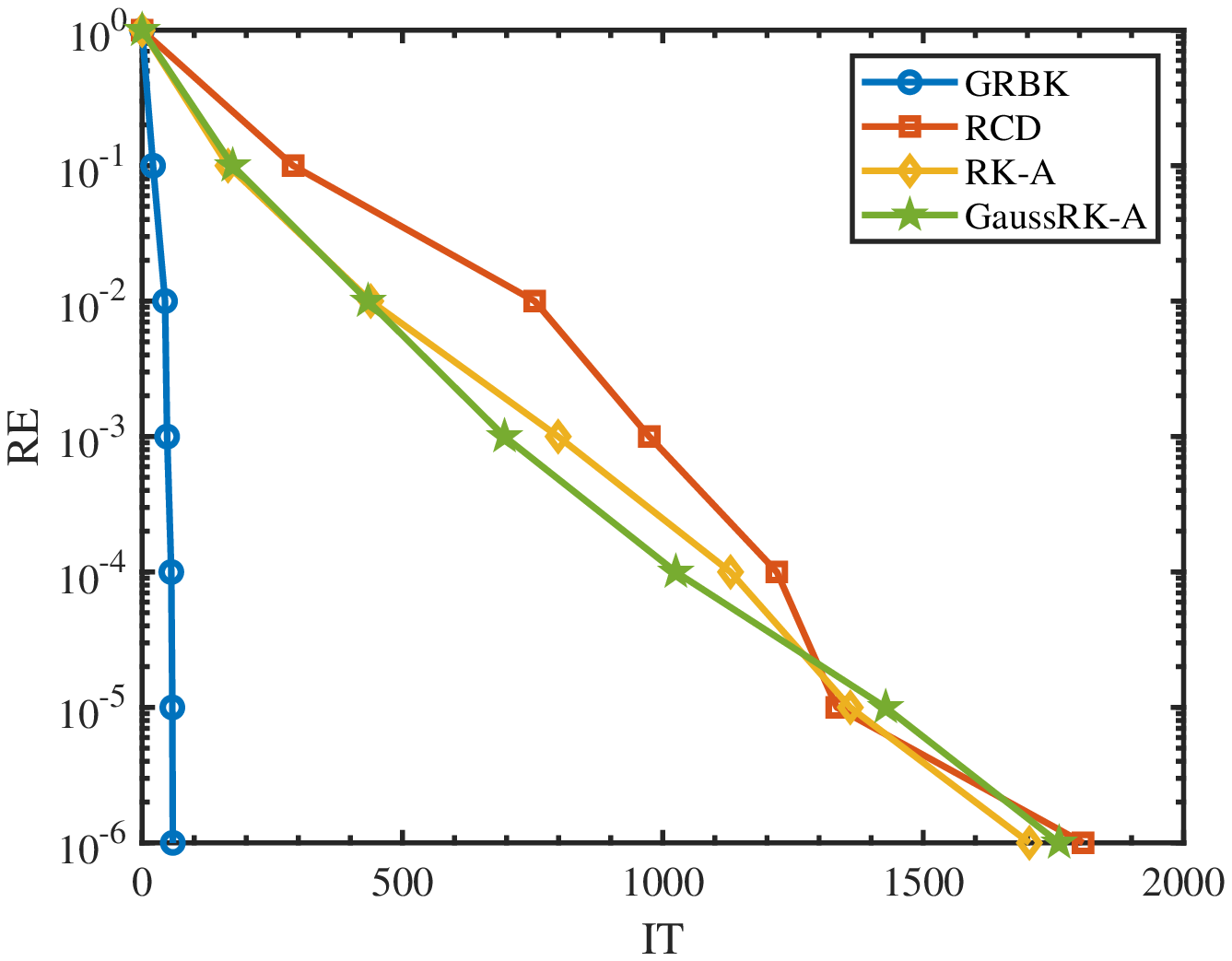}
		\includegraphics[width=3.8cm]{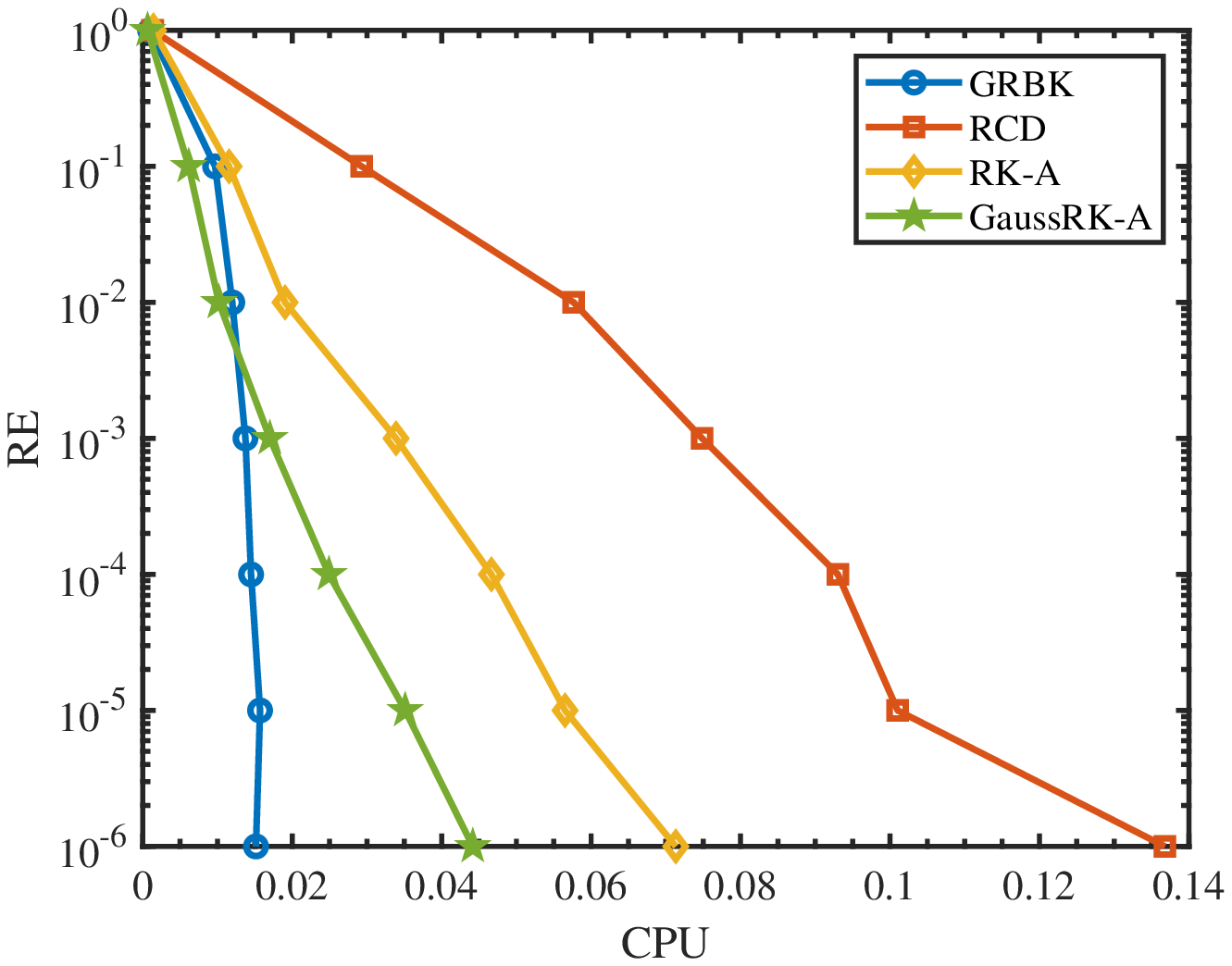}
	}
	\subfigure[$A=divorce, B=ash219^{\top}$]
	{
		\includegraphics[width=3.8cm]{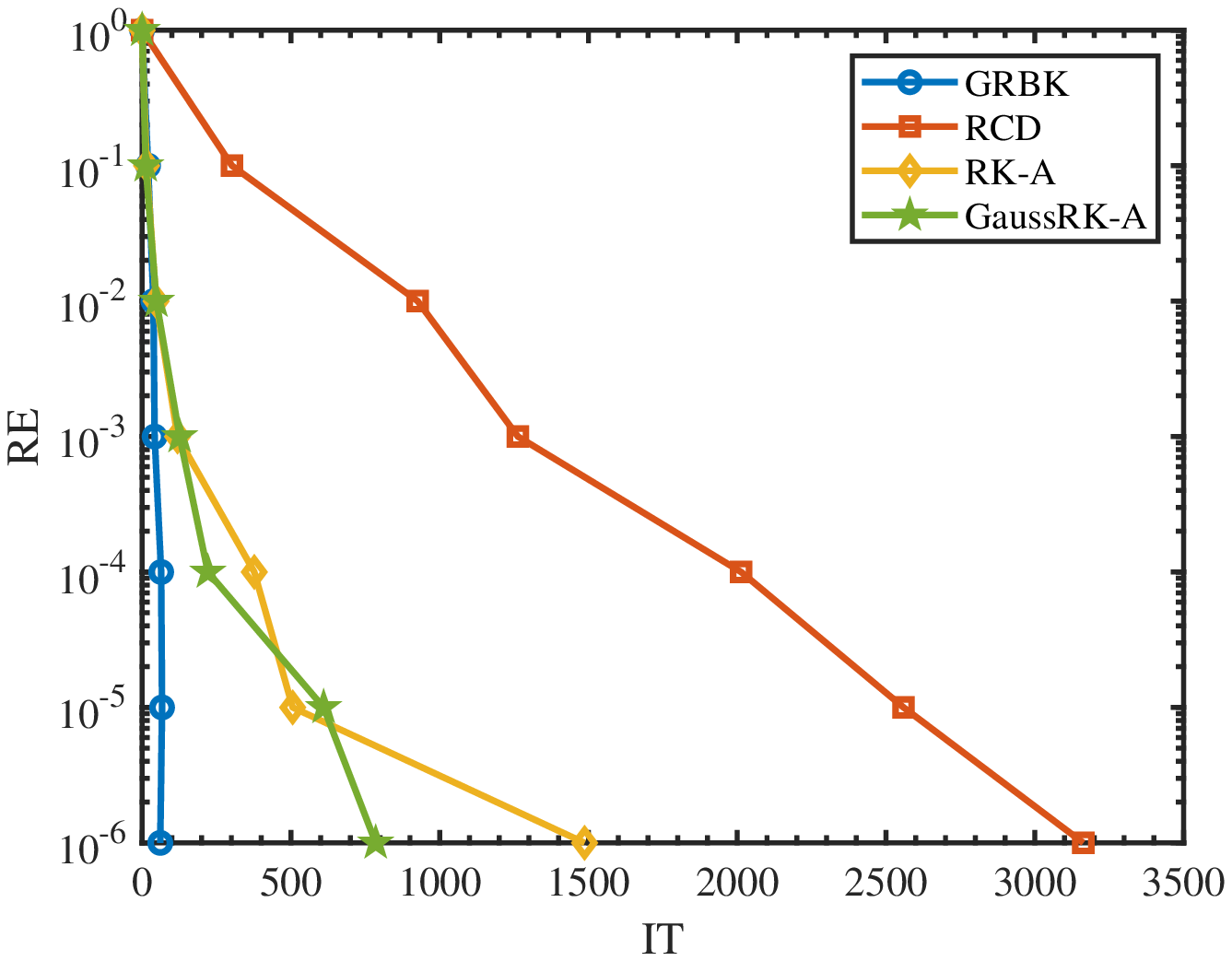}
		\includegraphics[width=3.8cm]{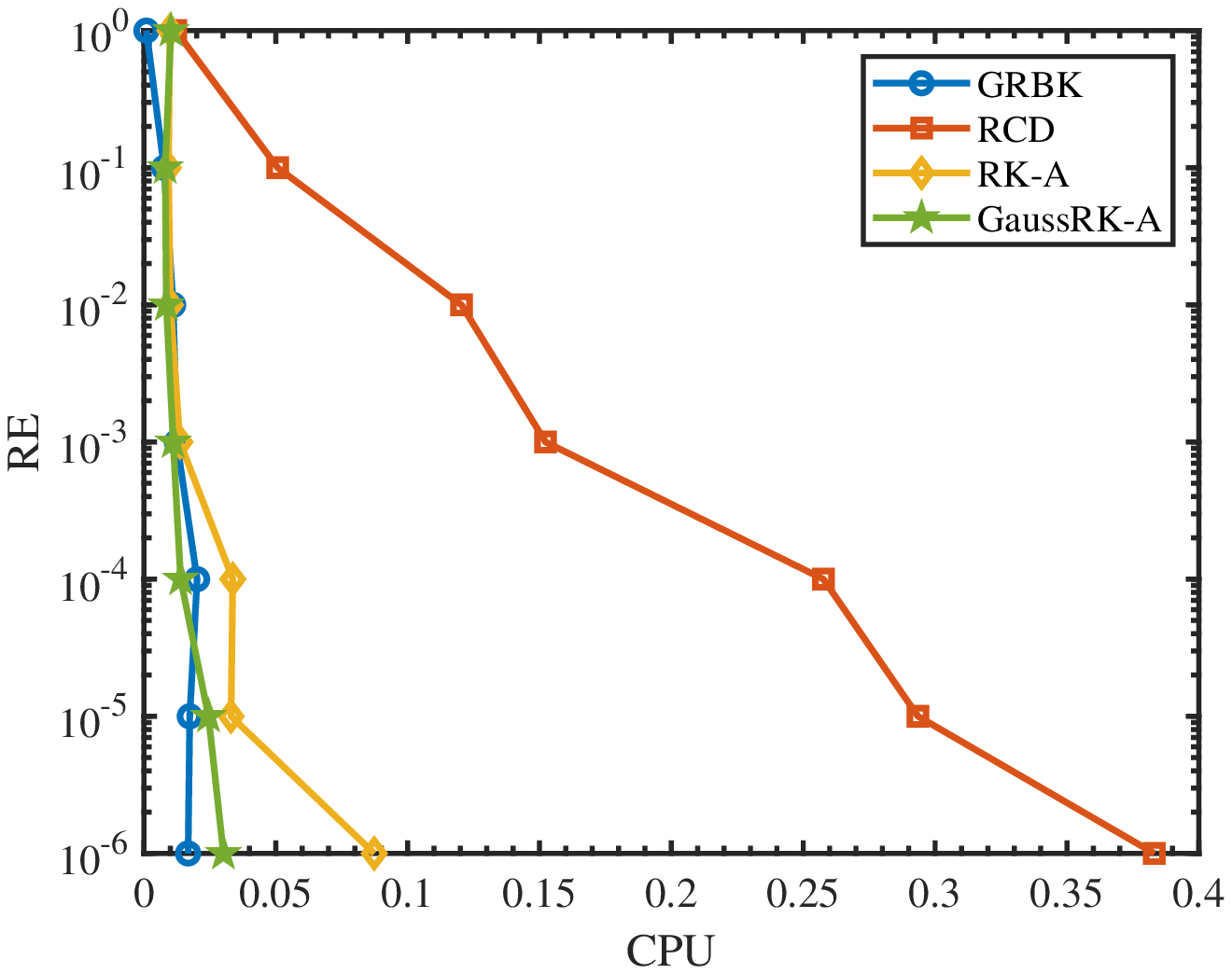}
	}\\
	\caption{Relative errors of GRBK, RCD, RK-A, and GaussRK-A.}\label{figure:3}
\end{figure}

\begin{table}[htbp]
    \footnotesize
    \resizebox{\textwidth}{!}{
    \begin{threeparttable}
		\caption{The average IT and CPU of GRK, GaussGRK, GRBK, RCD, RK-A and GaussRK-A.}
		\label{table:3}
		\begin{tabular}{ccccccccccccc}
			\toprule  
			A & B & $\tau_1$ & $\tau_2$ &               & GRK    & GaussGRK & GRBK   & RCD     & RK-A   & GaussRK-A  \\
			\hline
			\multirow{2}*{ash219} & \multirow{2}*{divorce$^{\top}$} & \multirow{2}*{15} & \multirow{2}*{15}
                                & IT      & $-$  & $-$   & \textbf{58}      & 1334     & 1360     & 1428     \\
			         &  &  &    & CPU     & $-$ & $-$  & \textbf{0.0152}  & 0.1368  & 0.0713  & 0.0442    \\	
			\hline
			\multirow{2}*{divorce} & \multirow{2}*{ash219$^{\top}$} & \multirow{2}*{15} & \multirow{2}*{15}
                                & IT      & $-$    & $-$   & \textbf{67}      & 2559    & 506    & 610      \\
			         &  &  &    & CPU     & $-$    & $-$   & \textbf{0.0167}  & 0.3823  & 0.0873  &0.0301    \\
			\hline
			\multirow{2}*{divorce} & \multirow{2}*{ash219} & \multirow{2}*{15} & \multirow{2}*{15}
                                & IT      & $-$    & $-$    & 1632     & 3024    & 1204    & \textbf{910}      \\
			         &  &  &    & CPU     & $-$    & $-$    & 0.4559  & 0.3216  & 0.0684  & \textbf{0.0291}    \\
			\hline
			\multirow{2}*{ash958} & \multirow{2}*{ash219$^{\top}$} & \multirow{2}*{15} & \multirow{2}*{14}
                                & IT      & $-$    & $-$    & \textbf{5038}     & 6105    & 5782    & 5265    \\
			         &  &  &    & CPU     & $-$    & $-$    & 3.2982  & 20.8641  & \textbf{2.6898}  & 4.2691    \\
			\hline
			\multirow{2}*{ash219} & \multirow{2}*{ash958$^{\top}$} & \multirow{2}*{15} & \multirow{2}*{15}
                                & IT      & $-$    & $-$    & 4976    & 1792    & \textbf{1709}    & 1718    \\
			         &  &  &    & CPU     & $-$    & $-$    & \textbf{3.0925}  & 9.0248 & 3.7558  & 3.8908    \\
			\hline
			\multirow{2}*{ash958} & \multirow{2}*{Worldcities$^{\top}$} & \multirow{2}*{15} & \multirow{2}*{15}
                                & IT      & $-$   & $-$    & 34572     & 5978     & 5756   & \textbf{5353}    \\
			         &  &  &    & CPU     & $-$   & $-$    & 23.6345   & 26.7466  & \textbf{4.035}  & 6.1356    \\
           \bottomrule 
		\end{tabular}
	\end{threeparttable}
    }
\end{table}

\begin{example}\label {exp6.3}
CT Data. The test problems of two-dimensional tomography are implemented in the function \textbf{seismictomo $(N, s, p)$} and the function \textbf{paralleltomo $(N, \theta, q)$} in the MATLAB package AIR TOOLS \cite{PCH2018}, where $N$ represents that a cross-section of the subsurface is divided into N equally spaced intervals in both dimensions creating $N^2$ cells and $s, p, \theta$ and $q$ denote the number of sources, number of receivers, angle of parallel rays and number of parallel rays. We set $N=40, \ \theta = 0:200$ and $q=100$ in the function \textbf{paralleltomo $(N, \theta, q)$}, which generates an exact solution $x_*$ of size $1600 \times 1$ and $X^* = reshape(x_*, 40, 40)$, and let $N = 30, s = 60$ and $p = 100$ in the function \textbf{seismictomo $(N, s, p)$}, which generates an exact solution $x_*$ of size $900 \times 1$ and $X^* = reshape(x_*, 30, 30)$.
For given $p, q$, the entries of $A$ and $B$ are generated from standard normal distributions, i.e., $ A=randn(p, m),\ \ B=randn(n, q)$. $C$ is obtained by $C = AX^*B$. 
\end{example}
All computations start from the initial matrix $X_0=O$ and run 4000 iterations on the paralleltomo function and run 5000 iterations on the seismictomo function.
In the following experiments, the structural similarity index (SSIM) between the two images X and Y was used to evaluate the quality of the recovered images. SSIM is defined as
\begin{equation*}
  SSIM=\frac{\left(2\mu_X\mu_Y+C_1\right)\left(2\delta_{XY}+C_2\right)}{\left(\mu_X^2+\mu_Y^2+C_1\right)\left(\delta_X^2+\delta_Y^2+C_2\right)},
\end{equation*}
where $\mu_X,\mu_Y$ and $\delta_X^2,\delta_Y^2$ are the means and variances of image $X,Y,$ respectively. $\delta_{XY}$ is the covariance of images X and Y, $C_1$ and $C_2$ are brightness and contrast constants. The mean of the image represents the brightness of the image and the variance of the image indicates the contrast of the image. Criteria for judging SSIM: SSIM is a number between 0 and 1, and the larger the SSIM value is, the smaller the difference between the two images is. Numerical results are shown in Fig.\ref{figure:5} and Fig.\ref{figure:6}. The convergence conclusions similar to the previous two groups of experiments are verified again.

\begin{figure}[htbp]
  \centering
  \subfigure[Reference image]{\includegraphics[width=4.2cm]{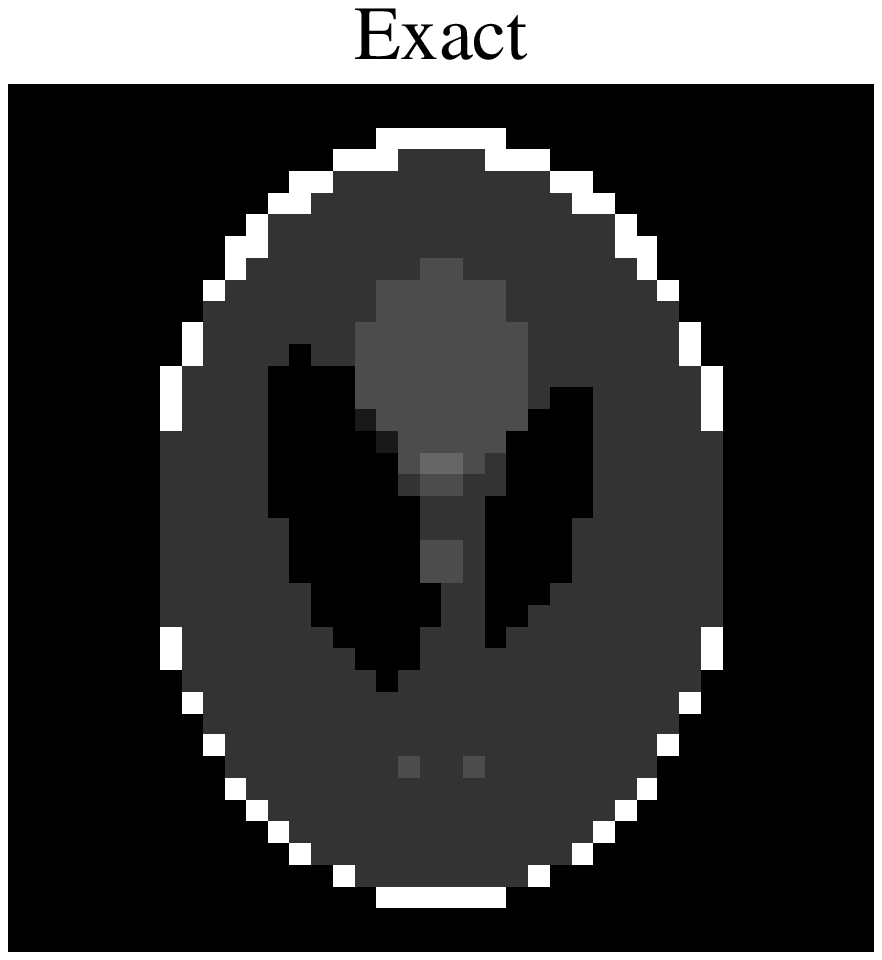}}
  \subfigure[SSIM=0.3658, IT=4000]{\includegraphics[width=4.2cm]{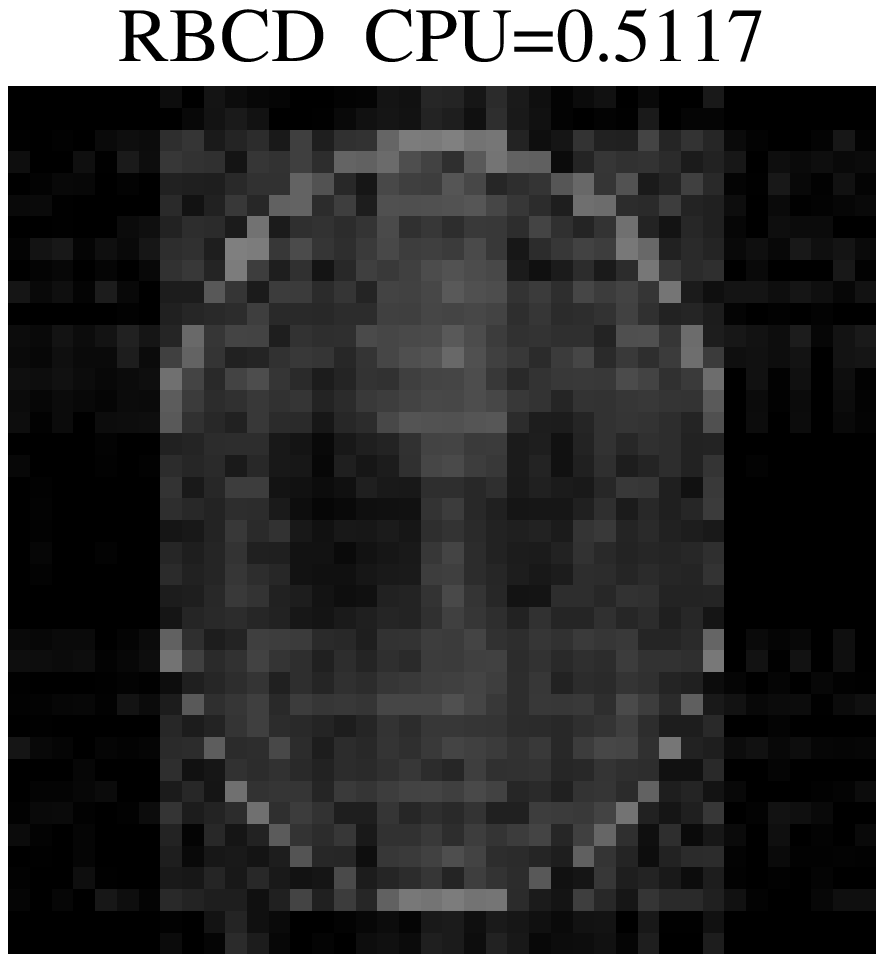}}
  \subfigure[SSIM=0.9992, IT=4000]{\includegraphics[width=4.2cm]{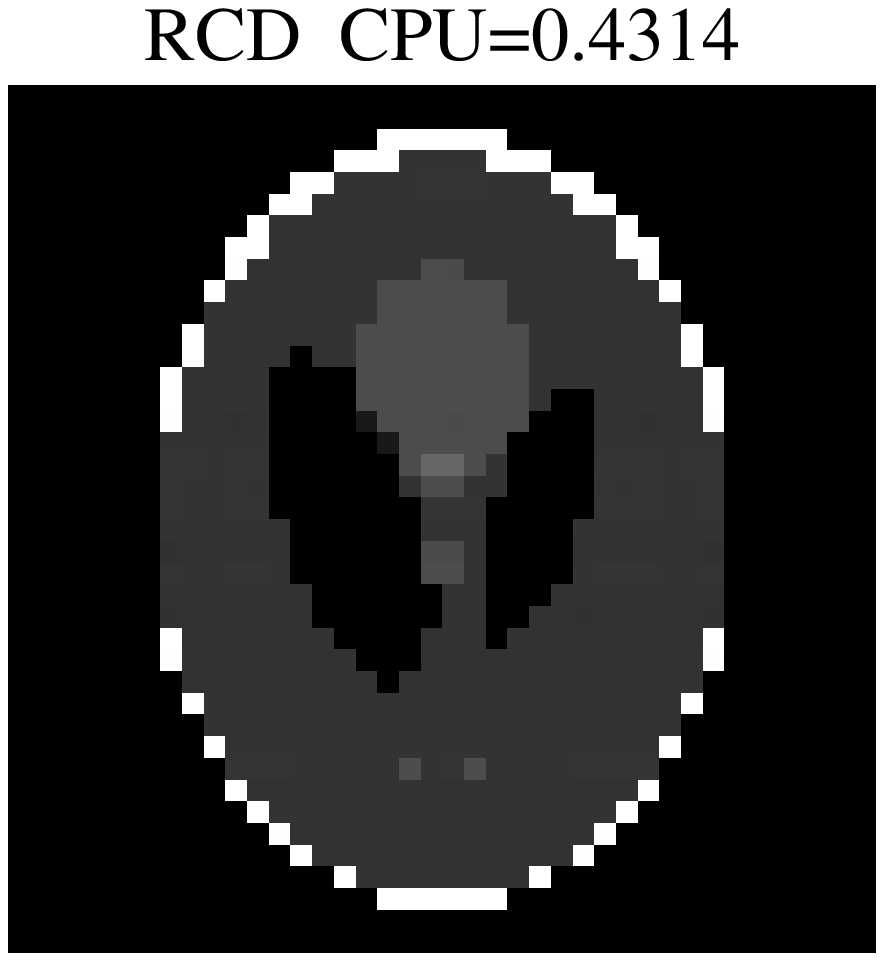}}\\ 
  \subfigure[SSIM=0.9999, IT=4000]{\includegraphics[width=4.2cm]{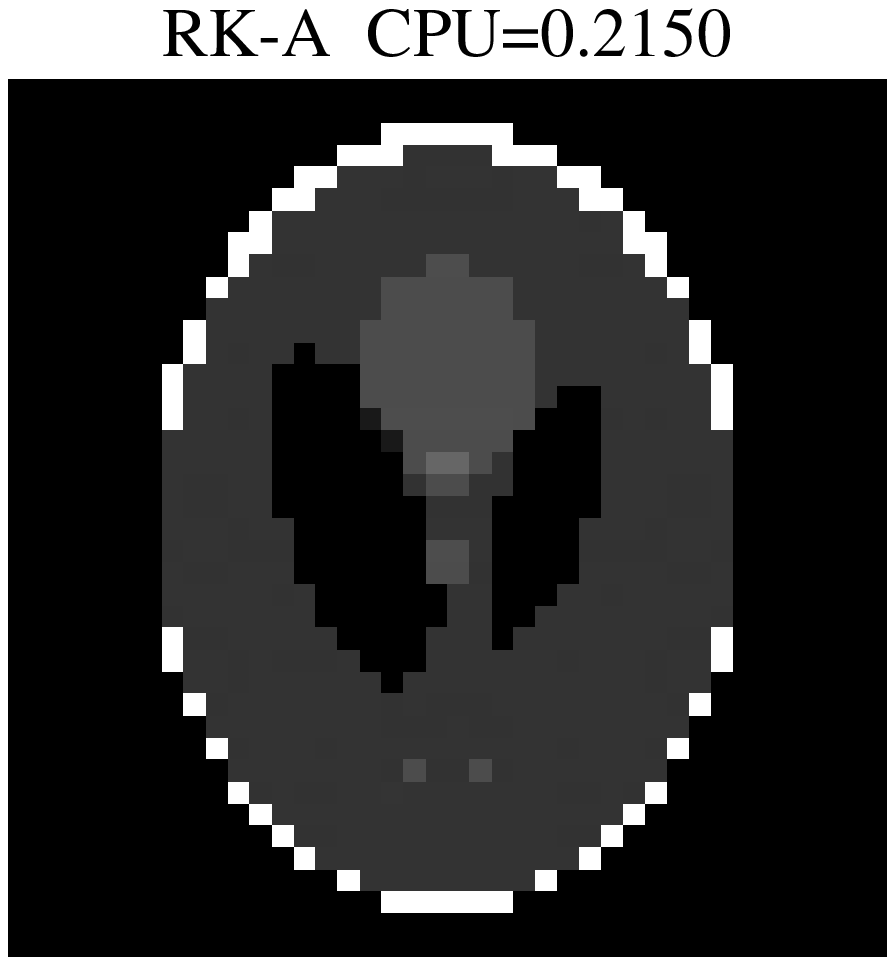}}
  \subfigure[SSIM=1.0000, IT=2764]{\includegraphics[width=4.2cm]{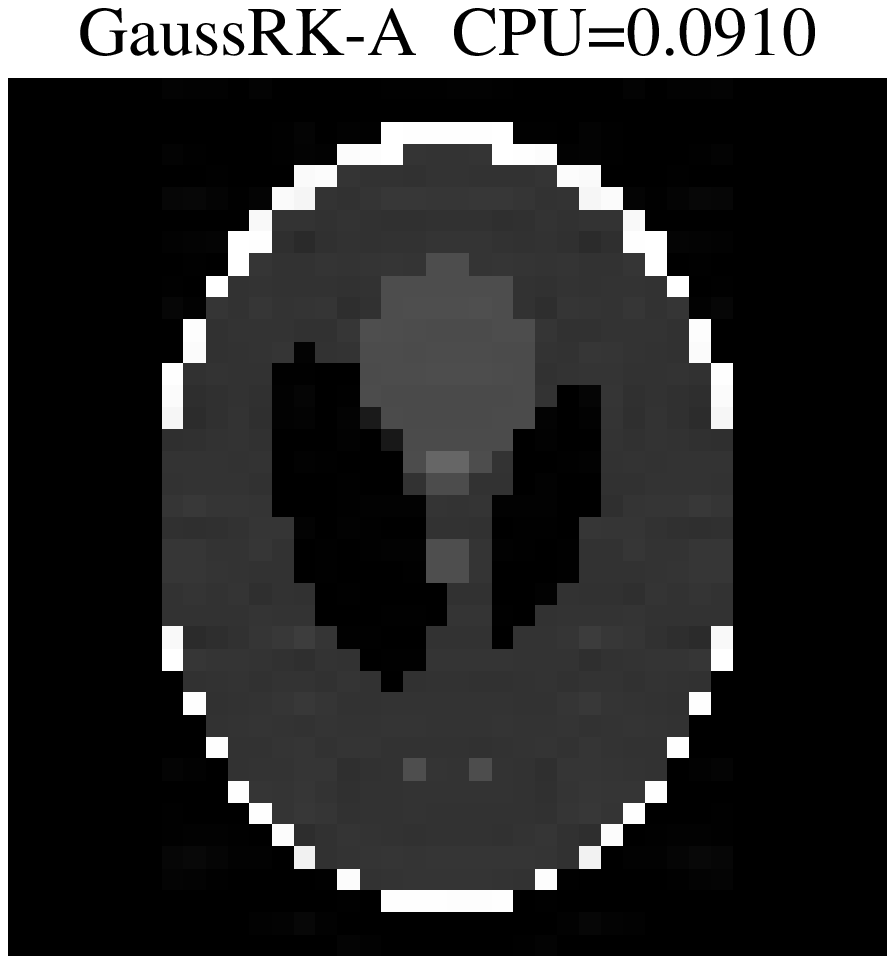}}
  \subfigure[SSIM=1.0000, IT=827]{\includegraphics[width=4.2cm]{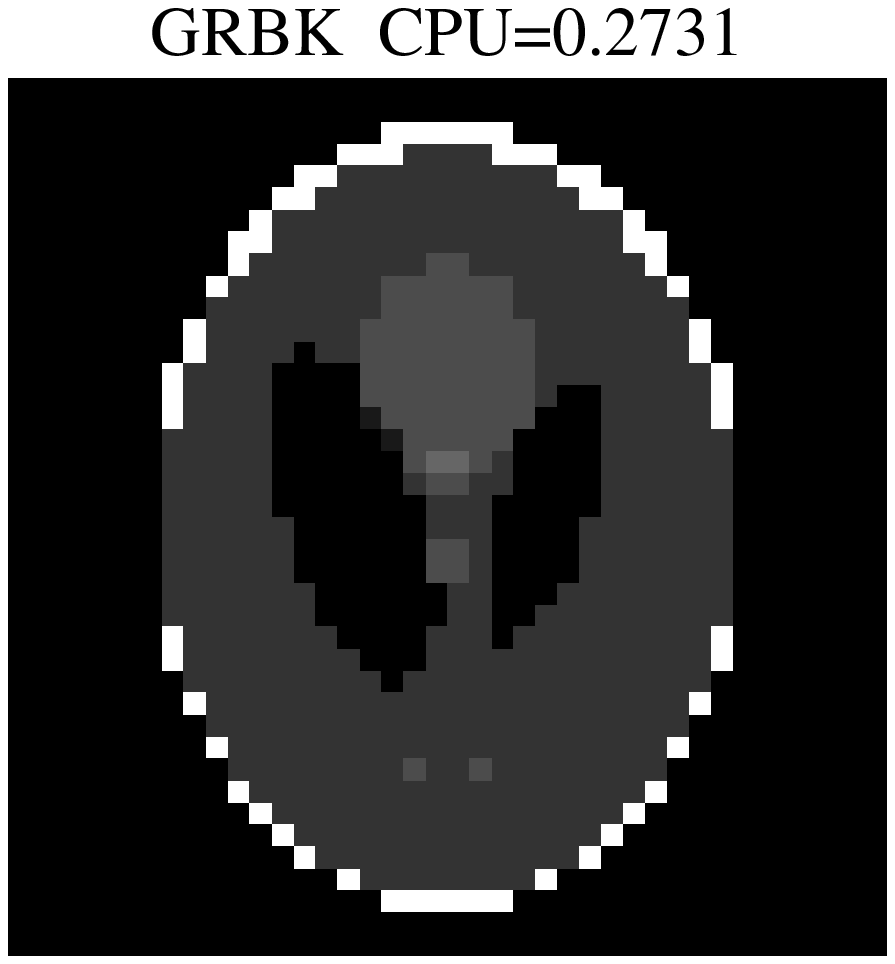}}

  \caption{Performance of the exact figure and five methods for paralleltomo test problem with n = 40.}\label{figure:5}
\end{figure}

\begin{figure}[htbp]
  \centering
  \subfigure[Reference image]{\includegraphics[width=4.2cm]{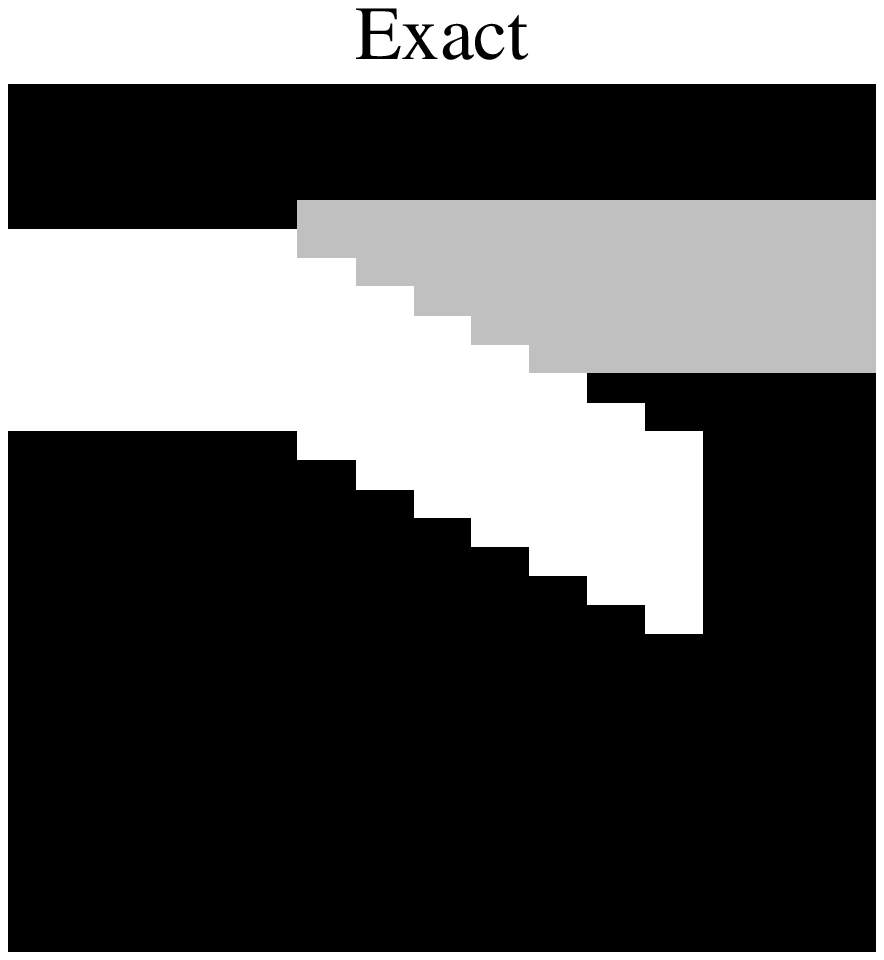}}
  \subfigure[SSIM=0.3782, IT=5000]{\includegraphics[width=4.2cm]{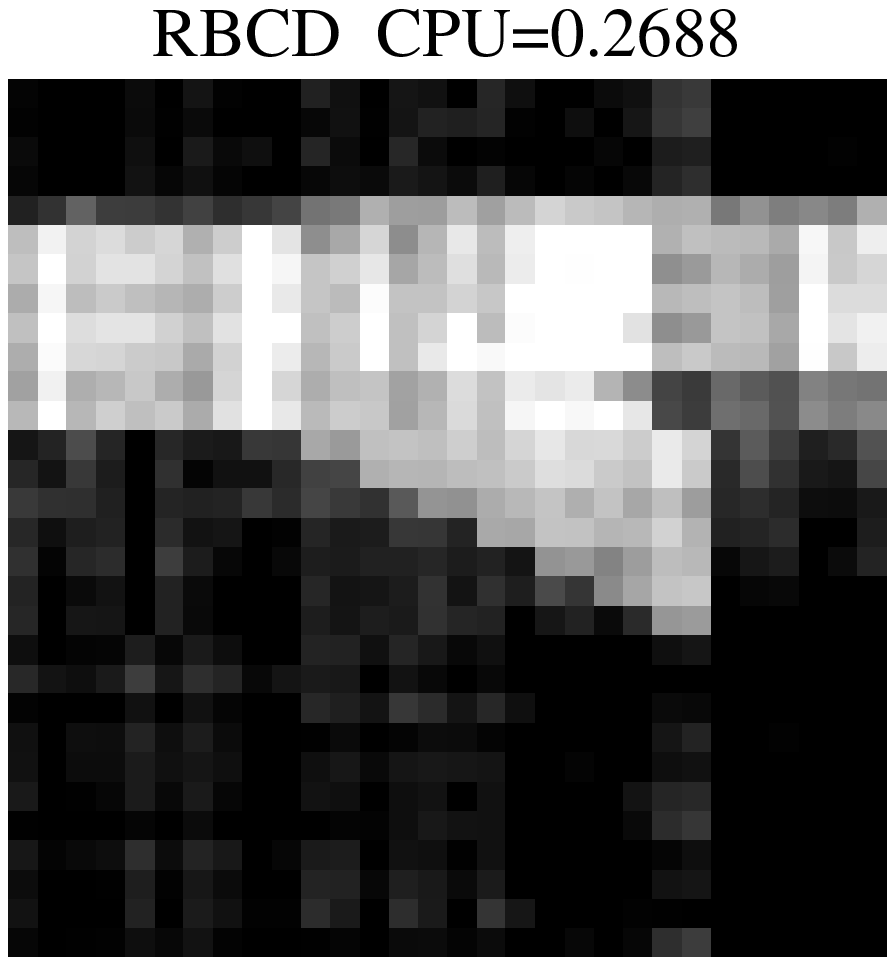}}
  \subfigure[SSIM=0.9997, IT=2999]{\includegraphics[width=4.2cm]{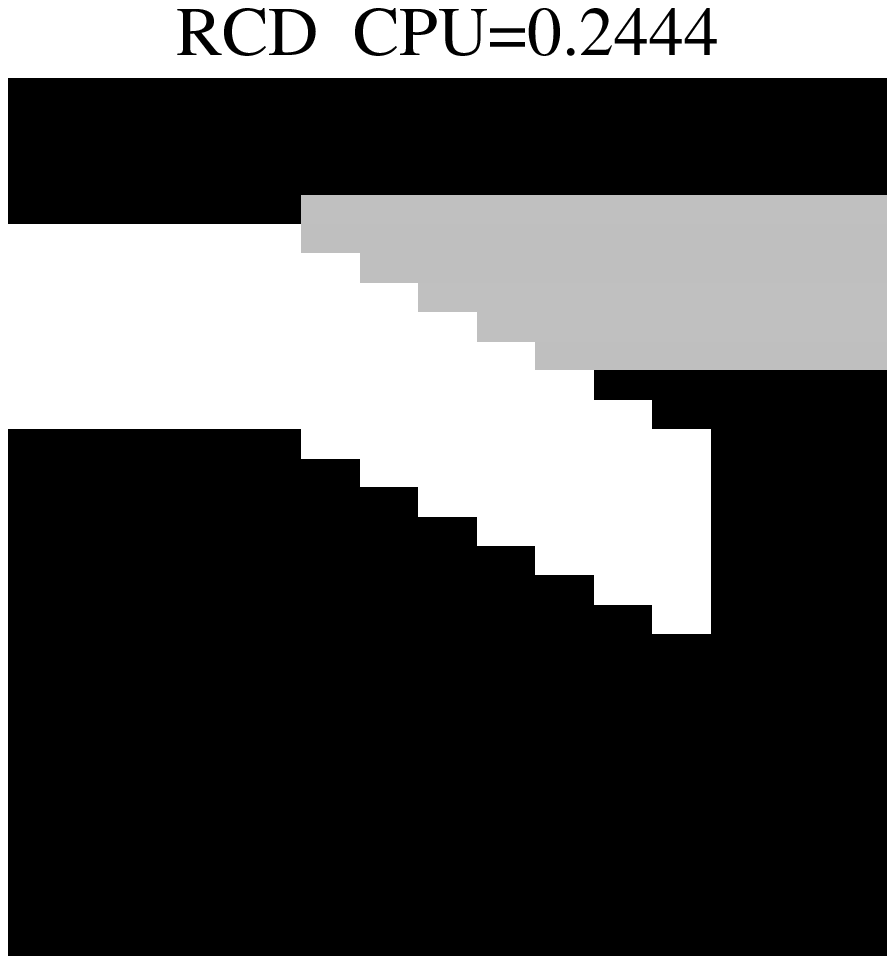}} \\
  \subfigure[SSIM=0.9998, IT=2620]{\includegraphics[width=4.2cm]{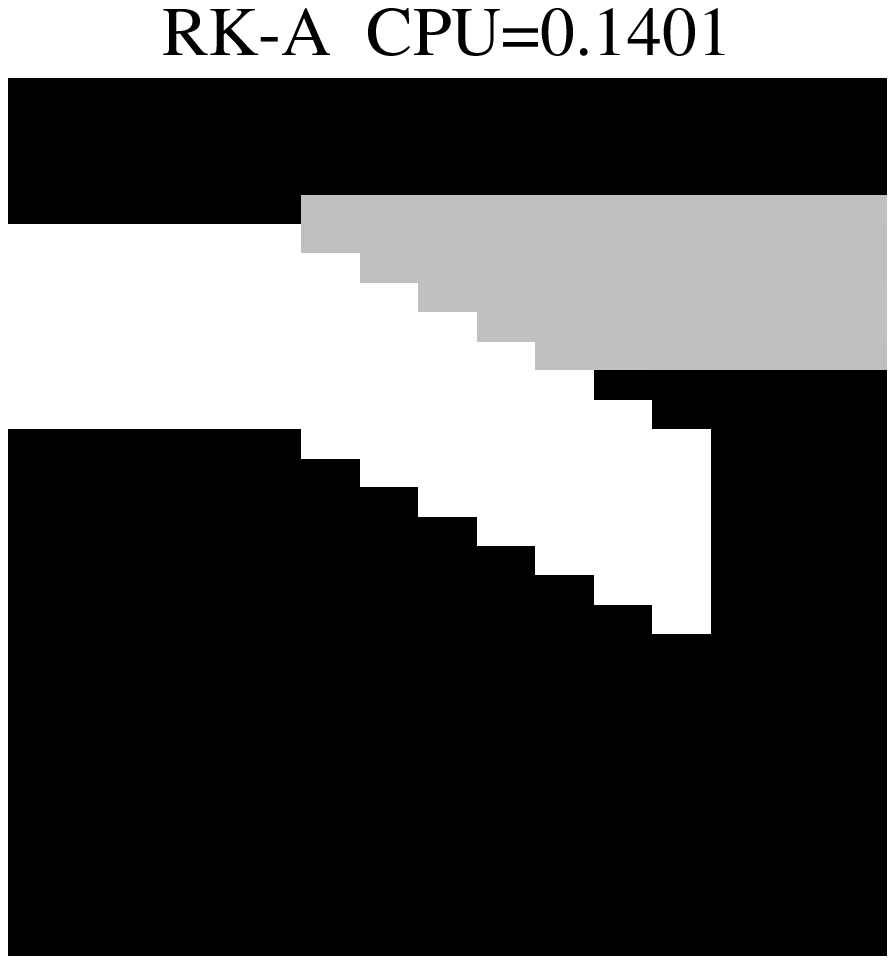}}
  \subfigure[SSIM=0.9999, IT=1549]{\includegraphics[width=4.2cm]{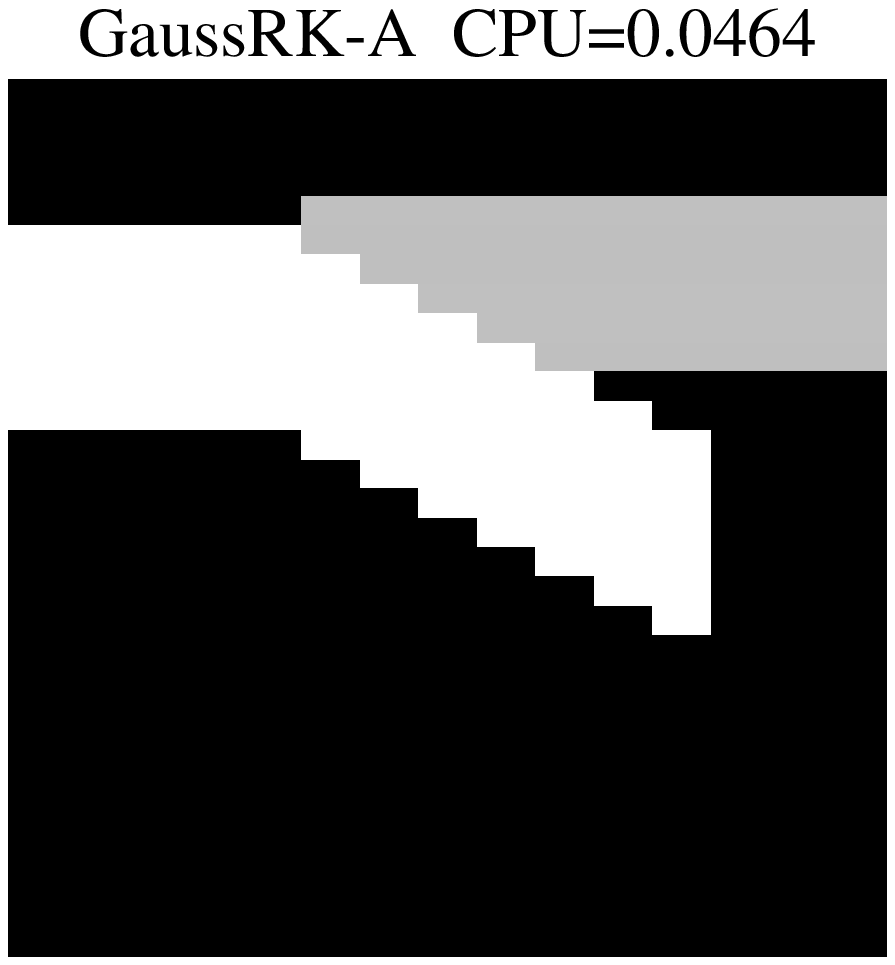}}  
  \subfigure[SSIM=0.9999, IT=158]{\includegraphics[width=4.2cm]{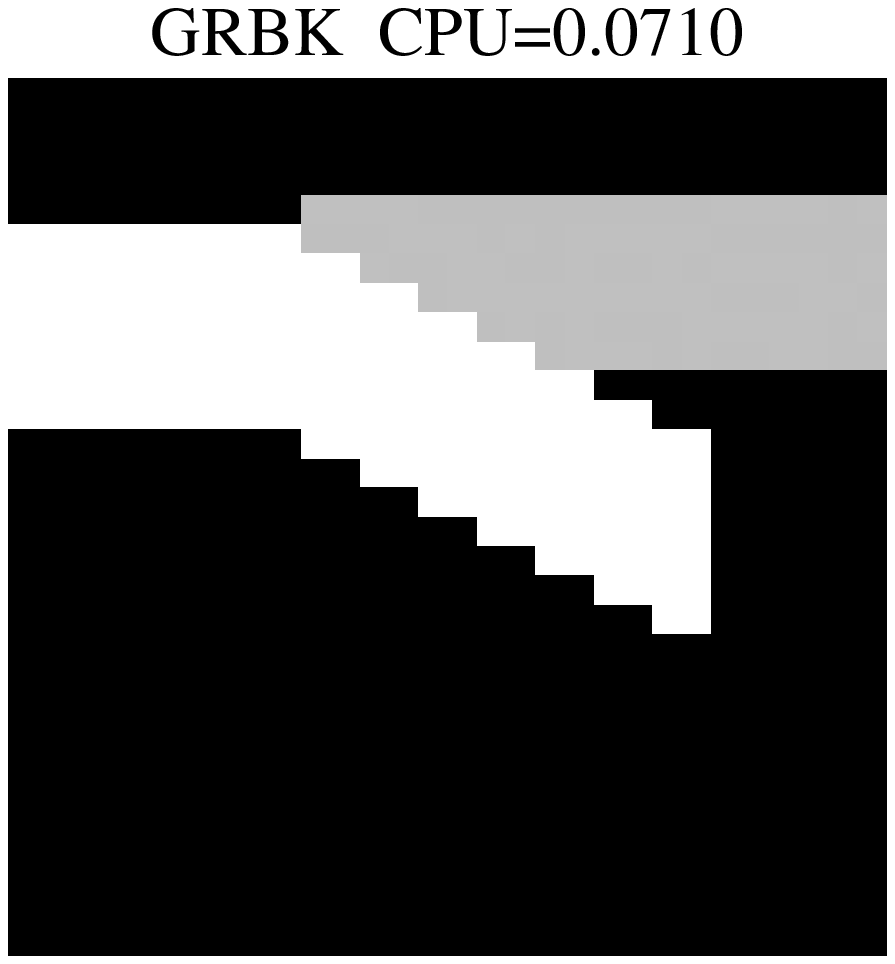}}

  \caption{Performance of the exact figure and five methods for seismictomo test problem with n = 30.}\label{figure:6}
\end{figure}
The maximum number of iterations for all these methods is set no more than 4000. From Fig.\ref{figure:5} and Fig.\ref{figure:6}, we can see that the RCD, RK-A, GaussRK-A, and GRBK methods recovered by the sketch-and-project method perform better than the RBCD method in terms of both image processing and CPU times. In Fig.\ref{figure:5}, the SSIM value of the recovery image through both the GaussRK-A and GRBK methods is around 1. Since the GRBK method needs to calculate the pseudoinverse, it requires more CPU times than the GaussRK-A method. In Fig.\ref{figure:6}, we can see that all methods have almost recovered this image except the RBCD method.

\section{Conclusions}
\label{sec:7}
In this paper, we have proposed a sketch-and-project method for solving the matrix equation $AXB=C$. The convergence of the generalized iterative method is explored. Meanwhile, by varying its three parameters, we recover some well-known algorithms as special cases. Numerical experiments show that in a series of methods of vector-matrix product, Gaussian-type methods are competitive in terms of IT and CPU time. It is clear to see that our method allows for a much wider selection of three parameters, which leads to a series of new specific methods. Based on this skecth-and project method, we will investigate new methods for solving nonlinear matrix equations in our future work.

\appendix

\section{Proof for Lemmas}

\noindent\textbf{Proof of Lemma \ref{EP}}
\begin{proof}
Since $\mathbf{E}\left[A^2\right]=\mathbf{E}\left[A^{\top}A\right],\  \mathbf{E}\left[A^{\top}\right]\mathbf{E}\left[A\right]=\left(\mathbf{E}\left[A\right]\right)^{\top}\mathbf{E}\left[A\right]$, to obtain the conclusion, we need to prove $\mathbf{E}\left[A^{\top}A\right]\succeq \mathbf{E}\left[A^{\top}\right]\mathbf{E}\left[A\right],$ i.e., $$\mathbf{E}\left[ \left(A^{\top}-\mathbf{E}\left[A^{\top}\right]\right)\left(A-\mathbf{E}\left[A\right]\right) \right]\succeq  0.$$ By the definition, it can be seen that for arbitrary column vector $c \in \mathbb{R}^{n}$, $A$ is called a positive semi-definite matrix  if $c^{\top}Ac \geq 0$. So we just need to prove that $$\mathbf{E}\left[ c^{\top}\left(A^{\top}-\mathbf{E}\left[A^{\top}\right]\right)\left(A-\mathbf{E}\left[A\right]c\right) \right]\geq 0,$$ i.e.,
	$$\mathbf{E}\left[ \left(c^{\top}A^{\top}-\mathbf{E}\left[c^{\top}A^{\top}\right]\right)\left(Ac-\mathbf{E}\left[Ac\right]\right) \right]\geq 0.$$
	Since $A \in \mathbb{R}^{n\times n}$, $Ac \in \mathbb{R}^{n}$ is a column vector. For convenience, let $c^{\top}A^{\top}=(y_1,y_2,...,y_n)=Y^{\top}$, we have
	\begin{equation*}
		\begin{split}
			\mathbf{E}\left[ \left(Y^{\top}-\mathbf{E}\left[Y^{\top}\right]\right)\left(Y-\mathbf{E}\left[Y\right]\right) \right] & = \mathbf{E}\left[Y^{\top}Y-\mathbf{E}\left[Y^{\top}\right]Y-Y^{\top}\mathbf{E}\left[Y\right]+\mathbf{E}\left[Y^{\top}\right]\mathbf{E}\left[Y\right]\right]\\
			& = \mathbf{E}\left[Y^{\top}Y\right]-\mathbf{E}\left[Y^{\top}\right]\mathbf{E}\left[Y\right]-\mathbf{E}\left[Y^{\top}\right]\mathbf{E}\left[Y\right]+\mathbf{E}\left[Y^{\top}\right]\mathbf{E}\left[Y\right] \\
			& = \mathbf{E}\left[Y^{\top}Y\right]-\mathbf{E}\left[Y^{\top}\right]\mathbf{E}\left[Y\right] \\
			& = \mathbf{E}\left[y_1^2+y_2^2+...+y_n^2\right]-\left[\left(\mathbf{E}\left[y_1\right]\right)^2+\left(\mathbf{E}\left[y_2\right]\right)^2+...+\left(\mathbf{E}\left[y_n\right]\right)^2\right] \\
			& = Dy_1+Dy_2+...+Dy_n \geq 0,
		\end{split}
	\end{equation*}
where $Dy_i$ is the variance $y_i$ for $i=1,2,\cdots, n$.
	Therefore, $\mathbf{E}\left[ \left(A^{\top}-\mathbf{E}\left[A^{\top}\right]\right)\left(A-\mathbf{E}\left[A\right]\right) \right]$ is positive semi-definite.
\end{proof}

\noindent\textbf{Proof of Lemma \ref{M12}}


\begin{proof}
	For $\forall X_2 \in M_2$, then $\exists Y_2\ s.t.\ A^{\top}AY_2BB^{\top}=X_2$. Let $Y_1=AY_2B$, there exsits
	$$A^{\top}AY_2BB^{\top}=A^{\top}Y_1B^{\top}=X_2,$$
	which means $X_2 \in M_1$.
	
	For $\forall X_1 \in M_1$, then $\exists Y_1\ s.t.\ A^{\top}Y_1B^{\top}=X_1$, which means matrix equation $A^{\top}XB^{\top}=C$ has a solution. Hence, we konw that $A^{\top}(A^{\top})^{\dag}X_1(B^{\top})^{\dag}B^{\top}=X_1$. Based on the nature of the pseudoinverse, we can obtain
	$$(A^{\top})^{\dag}=(A^{\dag})^{\top}=(A^{\dag}AA^{\dag})^{\top}=\left(A^{\dag}(AA^{\dag})\right)^{\top}=(AA^{\dag})^{\top}(A^{\dag})^{\top}=AA^{\dag}(A^{\dag})^{\top},$$
	$$(B^{\top})^{\dag}=(B^{\dag})^{\top}=(B^{\dag}BB^{\dag})^{\top}=\left((B^{\dag}B)B^{\dag}\right)^{\top}=(B^{\dag})^{\top}(B^{\dag}B)^{\top}=(B^{\dag})^{\top}B^{\dag}B.$$
	Thus $X_1$ can be written as
	$X_1=A^{\top}AA^{\dag}(A^{\dag})^{\top}X_1(B^{\dag})^{\top}B^{\dag}BB^{\top}=A^{\top}AW_2BB^{\top}$
	where $W_2=A^{\dag}(A^{\dag})^{\top}X_1(B^{\dag})^{\top}B^{\dag}$. It holds $$A^{\top}Y_1B^{\top}=A^{\top}AW_2BB^{\top}=X_1.$$
	Thus, $X_1$ is also in the set $M_2$.
	Therefore, $M_1=M_2.$
\end{proof}

\noindent\textbf{Proof of Lemma \ref{lem:7.1}}
\label{appendix1}
\begin{proof}
	For any matrix $M$, the pseudoinverse satisfies the identity $M^{\dagger}MM^{\dagger} = M^{\dagger}$. Let $M=S^{\top}AG^{-1}A^{\top}S$, we get
	\begin{equation*}
		\begin{split}
			\left(Z_1'\right)^2 & = G^{-1}A^{\top}S\left(S^{\top}AG^{-1}A^{\top}S\right)^{\dagger}S^{\top}AG^{-1}A^{\top}S\left(S^{\top}AG^{-1}A^{\top}S\right)^{\dagger}S^{\top}A \\
			& = G^{-1}A^{\top}SM^{\dagger}MM^{\dagger}S^{\top}A  \\
			& = G^{-1}A^{\top}SM^{\dagger}S^{\top}A  \\
			& = G^{-1}A^{\top}S\left(S^{\top}AG^{-1}A^{\top}S\right)^{\dagger}S^{\top}A = Z_1',
		\end{split}
	\end{equation*}
	then have
	\begin{equation*}
		\left(Z_2 \otimes Z_1'\right)^2 = \left(Z_2 \otimes Z_1'\right)\left(Z_2 \otimes Z_1'\right) = Z_2^2 \otimes \left(Z_1'\right)^2 = Z_2 \otimes Z_1',
	\end{equation*}
	and thus both $Z_2 \otimes Z_1'$ and $I-Z_2 \otimes Z_1'$ are projection matrices.
	To show that $Z_2 \otimes Z_1'$ is an orthogonal projection with respect to the $\left(I \otimes G\right)$-inner product, we need to verify that $\left(Z_2 \otimes Z_1'\right)\left[(BP) \otimes \left(G^{-1}A^{\top}S\right)\right]= (BP)\otimes (G^{-1}A^{\top}S)$ and for every $y\in Null\left((P^{\top}B^{\top}) \otimes (S^{\top}A)\right)$ there exists $\left(Z_2 \otimes Z_1'\right)y=0$.
	
	The first relation is obtained from the properties of the pseudoinverse: $\left(M^{\top}M\right)^{\dagger}M^{\top}=M^{\dagger}$ and $MM^{\dagger}M=M$. Setting $M=G^{-\frac{1}{2}}A^{\top}S$, we have
	\begin{equation*}
		\begin{split}
			Z_1'\left(G^{-1}A^{\top}S\right) & = G^{-1}A^{\top}S\left(S^{\top}AG^{-1}A^{\top}S\right)^{\dagger}S^{\top}AG^{-1}A^{\top}S \\
			& = G^{-\frac{1}{2}}M\left(M^{\top}M\right)^{\dagger}M^{\top}M \\
			& = G^{-\frac{1}{2}}MM^{\dagger}M  \\
			& = G^{-\frac{1}{2}}M  \\
			& = G^{-\frac{1}{2}}G^{-\frac{1}{2}}A^{\top}S = G^{-1}A^{\top}S.
		\end{split}
	\end{equation*}
	Similarly, denoting $M=BP$, we have
	\begin{equation*}
		\begin{split}
			Z_2BP & = BP\left(P^{\top}B^{\top}BP\right)^{\dagger}P^{\top}B^{\top}BP \\
			& = M\left(MM^{\dagger}\right)M^{\top}M  \\
			& = MM^{\dagger}M = M =BP.
		\end{split}
	\end{equation*}
	Thus the first relation holds.
	For the second relation, it exists
	\begin{equation*}
		\begin{split}
			\left(Z_2 \otimes Z_1'\right)y  =& \left\{\left[BP\left(P^{\top}B^{\top}BP\right)^{\dagger}P^{\top}B^{\top}\right] \otimes \left[G^{-1}A^{\top}S\left(S^{\top}AG^{-1}A^{\top}S\right)^{\dagger}S^{\top}A \right]\right\}y \\
			 =& \left\{\left[BP\left(P^{\top}B^{\top}BP\right)^{\dagger}\right] \otimes \left[G^{-1}A^{\top}S\left(S^{\top}AG^{-1}A^{\top}S\right)^{\dagger} \right]\right\}\\
			&\left[(P^{\top}B^{\top}) \otimes (S^{\top}A)\right]y  \\
			 =& 0.
		\end{split}
	\end{equation*}
\end{proof}

\noindent\textbf{Proof of Lemma \ref{lem:7.3}}

\label{appendix3}
\begin{proof}
Let $\widetilde{Z_1}= GZ_1'$. Since $\mathbf{E}\left[Z_2 \otimes \widetilde{Z_1}\right]$ is invertible  and $ G^{-\frac{1}{2}}\widetilde{Z_1}G^{-\frac{1}{2}}$  is an idempotent matrix, the spectrum of $\left(I \otimes G^{-\frac{1}{2}}\right)\left(Z_2 \otimes \widetilde{Z_1}\right)\left(I \otimes G^{-\frac{1}{2}}\right)$ is contained in $\left\{0,1\right\}$, we have$\left(I \otimes G^{-\frac{1}{2}}\right)\mathbf{E}\left[Z_2 \otimes \widetilde{Z_1}\right]\left(I \otimes G^{-\frac{1}{2}}\right)$ is positive definite.
	With
	\begin{equation*}
		\begin{split}
			\rho & = 1- \lambda_{min}\left(\mathbf{E}\left[Z_2 \otimes Z_1'\right]\right) \\
			& = 1- \lambda_{min}\left(\mathbf{E}\left[Z_2 \otimes \left(G^{-\frac{1}{2}}\widetilde{Z_1}G^{-\frac{1}{2}}\right)\right]\right).
		\end{split}
	\end{equation*}
	it holds $\rho < 1$. If $B^{\top}\otimes A$ is not full
	column rank, then there would be $0\neq x \in \mathbb{R}^{nm}$ such that $\left(B^{\top}\otimes A\right)x=0$. Therefore, we have $ \widetilde{Z_1}XZ_2=0$ and $\mathbf{E}\left[Z_2^{T} \otimes \widetilde{Z_1}\right]vec(X)=0$, which contradicts the assumption that $\mathbf{E}\left[Z_2^{T} \otimes \widetilde{Z_1}\right]$ is invertible. Analogously, $(BP)^{\top} \otimes (S^TA)$ is also full column rank.  Finally, since $B^{\top} \otimes A$ is full column rank, $X^*$ must be unique (recall that assume throughout the paper that $AXB=C$ is consistent).
\end{proof}

\noindent\textbf{Proof of Lemma \ref{lem:4.3}}
\label{appendix}

\begin{proof}
Using the properties of Kronecker product and considering the matrices are symmetric semi-definite,
we have
\begin{equation*}
  \lambda_{\min}\left( A^{\top} \otimes B \right) = \lambda_{\min} \left( A\right) \lambda_{\min} \left( B\right).
\end{equation*}
To prove \eqref{eq:4.5}, we only need to demonstrate that
\begin{equation*}
  \lambda_{\min} \left( A_2\right) \lambda_{\min} \left( B_2\right) \geq  \lambda_{\min} \left( A_1\right) \lambda_{\min} \left( B_1\right).
\end{equation*}
Since $A_2$ and $A_2-A_1$ are symmetric positive definite matrices, by Lemma \ref{lem:4.2} we can obtain the following inequality
\begin{equation*}
  0\leq  \lambda_{\min} \left( A_2-A_1\right) \leq \lambda_{\min} \left( A_2\right)+ \lambda_{\max} \left( -A_1\right).
\end{equation*}
From the fact that $\lambda_{\max} \left( -A_1\right)= -\lambda_{\min} \left( A_1\right)$, hence it results in 
\begin{equation*}
  0\leq  \lambda_{\min} \left( A_2-A_1\right) \leq \lambda_{\min} \left( A_2\right)-\lambda_{\min} \left( A_1\right).
\end{equation*}
Therefore, we get
\begin{equation}\label{eq:4.6}
  \lambda_{\min} \left( A_2\right)\geq \lambda_{\min}\left( A_1\right) \geq 0.
\end{equation}
Similarly, for $B_1,B_2$, there exists 
\begin{equation}\label{eq:4.7}
  \lambda_{\min} \left( B_2\right)\geq \lambda_{\min}\left( B_1\right) \geq 0.
\end{equation}
Combining \eqref{eq:4.6} and \eqref{eq:4.7}, we can obtain
$
  \lambda_{\min} \left( A_2\right) \lambda_{\min} \left( B_2\right) \geq  \lambda_{\min} \left( A_1\right) \lambda_{\min} \left( B_1\right)
$,
i.e., $
  \lambda_{\min}\left( A_2^{\top} \otimes B_2 \right) \geq \lambda_{\min} \left( A_1^{\top} \otimes B_1 \right).
$ The proof is completed.
\end{proof}

\bibliographystyle{model1-num-names}
\bibliography{references}

\end{document}